\documentclass[11pt]{article}
\usepackage[utf8]{inputenc}
\usepackage[T1]{fontenc}
\usepackage{amsmath,amssymb,amsthm,bbm}
\usepackage{authblk}
\usepackage{subcaption}
\usepackage{verbatim}
\usepackage[usenames,dvipsnames]{color}
\usepackage{hyperref}
\usepackage[left=2cm,right=2cm,top=2cm,bottom=2cm]{geometry}
\usepackage{tikz}
\usetikzlibrary{decorations.pathreplacing}
\usetikzlibrary[patterns]
\usepackage[numeric,initials,nobysame,msc-links,abbrev]{amsrefs}
\usepackage{thmtools}
\usepackage{thm-restate}

\renewcommand{\eprint}[1]{\href{https://arxiv.org/abs/#1}{arXiv:#1}}
\newcommand{\pageafter}[1]{#1~pp.}
\BibSpec{article}{
+{} {\PrintAuthors} {author}
+{,} { \textit} {title}
+{.} { } {part}
+{:} { \textit} {subtitle}
+{,} { \PrintContributions} {contribution}
+{.} { \PrintPartials} {partial}
+{,} { } {journal}
+{} { \textbf} {volume}
+{} { \PrintDatePV} {date}
+{,} { \issuetext} {number}
+{,} { \pageafter} {pages}
+{,} { } {status}
+{,} { \PrintDOI} {doi}
+{,} { available at \eprint} {eprint}
+{} { \parenthesize} {language}
+{} { \PrintTranslation} {translation}
+{;} { \PrintReprint} {reprint}
+{.} { } {note}
+{.} {} {transition}
+{} {\SentenceSpace \PrintReviews} {review}
}
\BibSpec{collection.article}{
+{} {\PrintAuthors} {author}
+{,} { \textit} {title}
+{.} { } {part}
+{:} { \textit} {subtitle}
+{,} { \PrintContributions} {contribution}
+{,} { \PrintConference} {conference}
+{} {\PrintBook} {book}
+{,} { } {booktitle}
+{,} { \PrintDateB} {date}
+{,} { \pageafter} {pages}
+{,} { } {status}
+{,} { \PrintDOI} {doi}
+{,} { available at \eprint} {eprint}
+{} { \parenthesize} {language}
+{} { \PrintTranslation} {translation}
+{;} { \PrintReprint} {reprint}
+{.} { } {note}
+{.} {} {transition}
+{} {\SentenceSpace \PrintReviews} {review}
}

\usetikzlibrary{arrows}
\usetikzlibrary[patterns]

\newtheorem{theorem}{Theorem}
\newtheorem{corollary}[theorem]{Corollary}
\newtheorem{lemma}[theorem]{Lemma}
\newtheorem{proposition}[theorem]{Proposition}

\theoremstyle{definition}
\newtheorem{question}[theorem]{Question}
\newtheorem{definition}[theorem]{Definition}
\newtheorem{algorithm}[theorem]{Algorithm}
\newtheorem{remark}[theorem]{Remark}

\newtheorem{claim}[theorem]{Claim}

\newcommand{\bbH}{\mathbb{H}}

\newcommand{\bbP}{\mathbb{P}}

\newcommand{\bbR}{\mathbb{R}}

\newcommand{\bbT}{\mathbb{T}}

\newcommand{\bbZ}{\mathbb{Z}}

\newcommand{\cC}{\mathcal{C}}
\newcommand{\cD}{\mathcal{D}}
\newcommand{\cE}{\mathcal{E}}
\newcommand{\cF}{\mathcal{F}}

\newcommand{\cK}{\mathcal{K}}
\newcommand{\cL}{\mathcal{L}}

\newcommand{\cQ}{\mathcal{Q}}

\newcommand{\cS}{\mathcal{S}}

\makeatletter
\pgfdeclarepatternformonly[\LineSpace,\tikz@pattern@color]{my north east lines}{\pgfqpoint{-1pt}{-1pt}}{\pgfqpoint{\LineSpace}{\LineSpace}}{\pgfqpoint{\LineSpace}{\LineSpace}}
{
    \pgfsetcolor{\tikz@pattern@color}
    \pgfsetlinewidth{0.4pt}
    \pgfpathmoveto{\pgfqpoint{0pt}{0pt}}
    \pgfpathlineto{\pgfqpoint{\LineSpace + 0.1pt}{\LineSpace + 0.1pt}}
    \pgfusepath{stroke}
}
\pgfdeclarepatternformonly{mydots}{\pgfqpoint{-1pt}{-1pt}}{\pgfqpoint{1pt}{1pt}}{\pgfqpoint{1.999pt}{1.999pt}}
{
  \pgfpathcircle{\pgfqpoint{0pt}{0pt}}{.5pt}
  \pgfusepath{fill}
}
\makeatother

\newdimen\LineSpace
\tikzset{
    line space/.code={\LineSpace=#1},
    line space=10pt
}

\newcommand{\Hb}{\overline\bbH}

\newcommand{\<}{\langle}
\renewcommand{\>}{\rangle}

\newcommand{\diam}{\mathrm{diam}}
\newcommand{\eps}{\varepsilon}

\title{\scshape Explosive appearance of cores and bootstrap percolation on lattices}

\author[1]{Ivailo Hartarsky}
\author[2]{Lyuben Lichev}
    \affil[1]{Universite Claude Bernard Lyon 1, CNRS, Centrale Lyon, INSA Lyon, Université Jean Monnet, ICJ UMR5208, 69622 Villeurbanne, France, \texttt{hartarsky@math.univ-lyon1.fr}}
\affil[2]{TU Wien, Faculty of Mathematics and Geoinformation, Institute of Statistics and Mathematical Methods in Economics, Research Unit of Mathematical Stochastics, Wiedner Hauptstra\ss e 8-10, A-1040 Vienna, Austria, \texttt{lyuben.lichev@tuwien.ac.at}}

\begin{document}

\maketitle

\begin{abstract}
Consider the process where the $n$ vertices of a square $2$-dimensional torus appear consecutively in a random order.
We show that typically the size of the $3$-core of the corresponding induced unit-distance graph transitions from $0$ to $n-o(n)$ within a single step. 
Equivalently, by infecting the vertices of the torus in a random order, under two-neighbour bootstrap percolation, the size of the infected set transitions instantaneously from $o(n)$ to $n$. This hitting time result answers a question of Benjamini. 

We also study the much more challenging and general setting of bootstrap percolation on two-dimensional lattices for a variety of finite-range infection rules.
In this case, powerful but fragile bootstrap percolation tools such as the rectangles process and the Aizenman--Lebowitz lemma become unavailable.
We develop a new method complementing and replacing these standard techniques, thus allowing us to prove the above hitting time result for a wide family of threshold bootstrap percolation rules on the $2$-dimensional square lattice, including neighbourhoods given by large $\ell^p$ balls for $p\in[1,\infty]$.
\end{abstract}


\section{Introduction}
\label{sec:intro}
\subsection{Background}
\label{subsec:background}

The $k$-core of a graph $G$ is the maximal subgraph of $G$ with respect to inclusion whose minimum degree is at least $k$. 
A standard way of finding the $k$-core of a graph $G$ is running a suitable bootstrap percolation dynamics on $G$. Namely, one iteratively deletes an arbitrarily chosen vertex of $G$ with degree less than $k$ together with its incident edges until no such vertices remain. 
If $G$ is realised as a subgraph of a $d$-regular host graph $\Gamma$, the above process may be reformulated as follows: starting with the vertex set $V = V(\Gamma\setminus G)$, iteratively and as long as possible, find a vertex in $G\setminus V$ with $r=d-k+1$ or more $\Gamma$-neighbours in $V$ and add it to $V$.

A classical question originating from statistical physics is to study when a non-empty $k$-core appears in an iteratively constructed subgraph $G$ of $\Gamma$ where vertices or edges arrive consecutively in a random order. For instance, when $\Gamma$ is a complete graph whose edges arrive one at a time, the above question asks about the time when the Erd\H{o}s--R\'enyi process creates a graph with a non-trivial $k$-core. 
Of course, this question has a natural analogue when $\Gamma$ is an infinite graph: in this case, one is interested in the critical probability $p_{\mathrm{c}}$ such that, almost surely, the random subgraph of $\Gamma$ obtained after $p$-percolation has a non-empty $k$-core when $p > p_{\mathrm{c}}$ and an empty $k$-core when $p < p_{\mathrm{c}}$.
When $\Gamma$ is a lattice (e.g.\ $\bbZ^2$ with the usual graph structure) or a finite approximation thereof (e.g.\ the torus), this question is often reformulated in terms of the density of initial infections (i.e.\ the vertices of $\Gamma\setminus G$) necessary for bootstrap percolation to infect the entire graph.
If $\Gamma$ is the infinite $d$-regular tree and $k=2$, the question asks for the critical probability for the appearance of a bi-infinite path, that is, percolation.
Each of these questions and many more have been thoroughly studied in the last half a century: see \cite{Kong19} for a recent review emphasising numerous applications, \cites{Luczak91,Pittel96,Ferber23} for classical works on the Erd\H os--R\'enyi model, \cites{Chalupa79,Schwarz06,Balogh06a} for regular trees, \cites{Kogut81,Aizenman88,Bollobas23,Hartarsky24locality,Balogh12} for lattices, among many others. 
In each of these settings, one of the most compelling features of the problem is that, for suitable values of $k$, it exhibits a \emph{discontinuous phase transition}.
More precisely, when viewed as a function of the density of the initially infected vertices or edges of $G$, the density of the core in the host graph $\Gamma$ suddenly jumps from 0 to a strictly positive value. Such discontinuous transitions with diverging length scale (known as hybrid or mixed transitions) are particularly relevant from the viewpoint of the physics of jamming and glassy phenomena.

Besides determining the critical probability/density for the appearance of a non-empty $k$-core, other `critical' properties are also commonly of interest in this setting. 
For instance, one wants to know the size of the critical window in which the transition takes place (see \cites{Balogh03,Gravner08,Pittel96}, as well as~\cites{Dembo08,Gao18,Gao18a} for generalisations to random hypergraphs). 
One of the finest types of critical results concerns hitting times. Consider a large finite graph $\Gamma$ and its dynamic random subgraph $G$.
The \emph{hitting time} of a $k$-core in $G$ is the first time when $G$ contains a non-trivial $k$-core. 
What is the size of this core at the hitting time? In the Erd\H os--R\'enyi case, the answer to this question was provided by \L uczak~\cite{Luczak91} who proved that the core density is bounded away from 0 at the hitting time with high probability (as $|\Gamma|\to\infty$). 
Since then even more refined properties of the core at the hitting time are known, see e.g.\ \cites{Anastos21,Krivelevich14,Montgomery19}.

The goal of the present paper is to establish the even more brutal appearance of cores on lattices while proposing a more robust route to proving such critical results. 
Such phenomena on lattices are only beginning to be considered in physics \cite{Xue24}.

\subsection{Models}
\label{subsec:models}
In the remainder of this paper, we stick to bootstrap percolation terminology and leave the immediate translation to $k$-core language via the equivalence in Section~\ref{subsec:background} to the interested reader. 
We will only consider \emph{threshold} bootstrap percolation models since they are the ones in direct correspondence to $k$-core models; nevertheless, our methods also work in some other settings (see Section~\ref{subsec:future}). 
Threshold models are defined by an integer \emph{dimension} $d\ge 2$, a finite \emph{neighbourhood} $\cK\subset \bbZ^d$, an integer \emph{threshold} $r\ge 2$ (the case $r=1$ being trivial) and a set $X\subset\bbZ^d$ of \emph{initial infections}. The dynamics is given by the sets $X_0=X$ and, for every integer $\theta\ge 0$,
\begin{equation}
\label{eq:def:BP}
X_{\theta+1}=X_{\theta}\cup \left\{x\in\bbZ^d:|(x+\cK)\cap X_{\theta}|\ge r\right\}
\end{equation}
is the set of sites \emph{infected by time $\theta+1$}.
Here and below, we use the standard notation $x+Y=\{x+y:y\in Y\}$, $Y+Z=\{y+z:y\in Y,z\in Z\}$, $cY=\{cy:y\in Y\}$ for $c\in\bbR$, $x\in\bbR^d$ and $Y,Z\subset\bbR^d$. 
The \emph{closure} of $X\subset \bbZ^d$ is 
\[[X]=\bigcup_{\theta\ge 0} X_{\theta}.\]
As it will often be convenient to consider continuous domains, for $X\subset\bbR^d$, we set $[X]=[X\cap\bbZ^d]$.

We consider bootstrap percolation on the torus $\bbT=(\bbZ/n\bbZ)^d$ for large integer $n$. Note that this choice is somehow arbitrary and, in fact, any `reasonable' finite approximation of $\bbZ^d$ would do; see Section~\ref{subsec:future}. 
Rather than working with the usual initial condition $X_0$ sampled from a product Bernoulli measure, we dynamically add infections one at a time. More formally, we consider a uniformly random permutation $\sigma:\bbT\to\{1,\dots,n^d\}$ of the sites of $\bbT$ and denote \[A(t)=\{x\in\bbT:\sigma(x)\le t\}.\]
We define the (random) \emph{percolation time}
\begin{equation}
\label{eq:def:tau}
\tau=\min\left\{t\ge 0:[A(t)]=\bbT\right\}.
\end{equation}

It remains to specify the neighbourhoods and thresholds we focus on. 
In the sequel, we restrict our attention to planar lattices (where $d=2$) and only briefly discuss higher-dimensional settings in Section~\ref{subsec:future}.
We begin with two classical cases introduced already in the first work on bootstrap percolation on lattices \cite{Kogut81} with thresholds $r=2$ on the square lattice and $r=3$ on the triangular lattice defined by 
\begin{align*}
\cK_\square&{}=\{(0,0),(1,0),(0,1),(-1,0),(0,-1)\},&r_\square&{}=2,\\
\cK_\triangle&{}=\cK_\square\cup\{(1,-1),(-1,1)\},&r_\triangle&{}=3.
\end{align*}
We note that, for each of the neighbourhoods $\cK_\square$ and $\cK_\triangle$, the choice of threshold made above pinpoints the unique regime with genuinely interesting behaviour. 
Indeed, for smaller values of $r$, there are finite sets of infections which infect the entire plane (or torus) while, for larger values of $r$, there are finite sets of non-infected sites which remain such at all times.
We remark that, while we defined the model $\cK_\triangle$ on the square lattice for notational purposes, it is more naturally viewed on the triangular lattice where 3 infected neighbours (out of all 6) are needed to infect a given site.

As we shall see, the analysis of the preceding two models is considerably simplified by a special property absent in essentially all other models. 
We therefore also study certain `nice' larger neighbourhoods.
Let $\|\cdot\|_2$ and $\<\cdot,\cdot\>$ denote the Euclidean norm and scalar product on $\bbR^2$ and $B(x,\rho)$ be the corresponding ball with centre $x$ and radius $\rho$. 
We fix a closed convex set $K\subset\bbR^2$ invariant under rotation at angle $\pi/2$ (around the origin) such that $\max\{\|k\|_2:k\in K\}=1$. For instance, one could take $K$ to be the $\ell^p$-ball for $p\in[1,\infty]$ of appropriate radius. For $s\ge 1$, we set $\cK_s=(sK)\cap\bbZ^2$. Given $\cK_s$, the value of the threshold $r$ we will consider is the smallest one making the model non-trivial in the sense described above. Following~\cite{Gravner99}, given $K$ and $s$, we set 
\begin{equation}\label{eq:defrs}
r_s=1+\min_{u\in \bbR^2\setminus\{0\}}\left|\left\{x\in\cK_s:\<x,u\><0\right\}\right|.
\end{equation}
Models with threshold $r<r_s$ are called \emph{supercritical} and admit finite sets of infections infecting the entire plane, while ones with threshold $r>|\cK_s|/2$ are called \emph{subcritical} and admit finite non-infected sets that necessarily remain non-infected at all times (see \cite{Gravner99}*{Proposition~2.1}).
The remaining values being called \emph{critical}, $r_s$ becomes the first critical threshold. Notice that, for large $s$, there are only $\Theta(s)$ critical threshold values out of the $\Theta(s^2)$ possible ones.

\subsection{Results}
\label{subsec:results}

Our main result states that the percolation closure spreads extremely abruptly. Namely, only a single step before the percolation time, a negligible proportion of the torus is infected and, actually, essentially all infected sites are initial infections.

\begin{theorem}
\label{th:main}
Fix a convex $\pi/2$-rotation invariant set $K\subset\bbR^2$ defining the neighbourhoods $\cK_s$ for $s\ge 1$. Assume that $(\cK,r)\in\{(\cK_\square,r_\square),(\cK_\triangle,r_\triangle),(\cK_s,r_s)\}$. There is a constant $C>0$ such that, for every large enough $s$,
\begin{equation}
\label{eq:main}
\lim_{n\to \infty}\bbP\left(|[A(\tau-1)]|\ge \tau(1+C/\log n)\right)=0.
\end{equation}
\end{theorem}
We further note that, for each $(\cK,r)$ as in the theorem, there exists $\lambda>0$ solving a certain variational problem such that $(\tau \log n)/n^2\to \lambda$ in probability as $n\to\infty$ \cite{Duminil-Copin24}. In particular, $|[A(\tau-1)]|=o(n^2)$ with high probability while $|[A(\tau)]|=|\bbT|=n^2$ by \eqref{eq:def:tau}. Moreover, using that typically $\tau = \Theta(n^2/\log n)$, it is not hard to check that, with high probability, the set $(A(\tau-1))_1$ of sites infected after one step of the bootstrap percolation dynamics at time $\tau-1$ has cardinality $\tau(1+\Theta(1)/(\log n)^{r-1})$, where $\Theta(\tau)/(\log n)^{r-1}$ essentially corresponds to the expected number of non-infected sites with $r$ infected sites in their neighbourhood. 
In particular, the $1/\log n$ error term in Theorem~\ref{th:main} is sharp for $(\cK,r)=(\cK_\square,r_\square)$.

While the $\square$ and $\triangle$ cases of Theorem~\ref{th:main} partially answer a question of Benjamini~\cite{Benjamini19} about the effect of geometry (planar, hyperbolic, etc.) on the transition of the $k$-core, the large neighbourhood case in our result is much more challenging and represents our main contribution. Indeed, a large number of results in bootstrap percolation are only available for the two-neighbour model (on $\bbZ^d$, boxes or tori) thanks to a classical approach pioneered by Aizenman and Lebowitz~\cite{Aizenman88}.  
While their proof technique is very efficient, we shall see that it is also very fragile and fails to apply in reasonable generality. 
Our approach to large neighbourhoods is different and has the potential to apply much more broadly. 
We direct the reader to Section~\ref{sec:outline} for an outline of the proof. 
In this respect, the large neighbourhood part of Theorem~\ref{th:main} is intended as a proof of concept, even though it already covers interesting models.

\subsection{Extensions}
\label{subsec:future}
Since our goal is to showcase a new method, we have opted for simplicity over generality of the setting. Yet, our proof and results can already be extended in various ways.

\paragraph{Boundary conditions.}
The choice of the torus as our domain is not fundamental. For example, the same method can be used to show that, for the same random process on a lattice approximation of a smooth finite domain in the plane with mesh tending to 0, the density of the closure jumps from close to 0 to close to 1 in one step. 
Yet, some corners of the boundary may be hard (or even impossible) to infect without waiting for an initial infection to be added there. For instance, one could consider the $\square$ model on a large square box or the $\triangle$ model on a large regular hexagon without encountering this issue, but not if one works on a triangular domain.

\paragraph{Variants of the $\square$ and $\triangle$ models.}
Following our arguments, analogues of Theorem~\ref{th:main} can be shown for several commonly considered variants.
Examples include the modified two-neighbour model on the square lattice, in which one requires two non-opposite neighbours; 
the Frob\"ose model, in which one infects the fourth vertex of a 4-cycle with three infections in the square lattice; 
the modified three-neighbour model on the triangular lattice, in which three non-opposite neighbours are required; 
the triangular Frob\"ose model, in which one infects the fourth vertex of a 4-cycle with three infections not forming a triangle in the triangular lattice. Other closely related variants are also possible.

\paragraph{Higher-dimensional $\square$ model.}
The conclusion of Theorem~\ref{th:main} can be extended to the two-neighbour model on the hypercubic lattice in arbitrary dimension $d\ge 2$. 
We do not know of other higher dimensional results beyond the modifications of this model along the lines of the previous paragraph (modified two-neighbour, Frob\"ose, etc.).

\paragraph{The $\boxtimes$ model.}
Consider the model given by $\cK_\boxtimes=\{-1,0,1\}^d$ and $r_\boxtimes=3^{d-1}+1$ for $d\ge 2$. 
If $d=2$, a bit more complicated version of our argument still implies the conclusion of Theorem~\ref{th:main} for this model. 
Namely, Lemma~\ref{lem:single:site} below is no longer true but the exceptions to it can be controlled by hand (also see \cite{Balogh03}*{Section~3.3} for a closely related claim).
As shown by the example in Figure~\ref{fig:2,3}, this is about as far as the Aizenman--Lebowitz approach (described in Section~\ref{sec:AL4}) can go. We mention in passing that the fact that $r_\boxtimes$ is the smallest critical threshold for this neighbourhood for any $d\ge 2$ answers \cite{Bresar24}*{Question~23}, correcting the guess for the answer. This fact is a direct application of the universality results of \cite{Balister25}.

\paragraph{Edge bootstrap percolation.} A commonly studied process in the random graphs literature consists in constructing the edges of a host graph consecutively in a random order, instead of its vertices.
This setting may be naturally combined with \emph{$H$-graph bootstrap percolation} where an edge becomes infected if it participates in a copy of a graph $H$ where all other edges are already infected.
In the case when $H$ is a 3-star and the host graph is the unit distance graph on $\mathbb Z^2$, our technique may be used to derive an analogue of Theorem~\ref{th:main}.

\paragraph{Other `critical' results.} 
To our knowledge, the only existing `critical' results (that is, results regarding the behaviour within the critical window of values of the time $t\in[0,n^d]$) for critical models (as defined at the end of Section~\ref{subsec:models}) are
\cite{Balogh03}, \cite{Gravner08}*{Theorem~2}, \cite{Hartarsky22phd}*{Proposition~1.4.3} and \cite{Bartha15}.  
They all apply to the (modified) $\square$ and $\boxtimes$ models and crucially rely on the Aizenman--Lebowitz approach. 
We believe that our proof allows the extension of these results to the models we cover.

\paragraph{General critical models.}
Along the lines of the conjecture after \cite{Balogh03}*{Theorem~2}, we ask the following.

\begin{question}\label{qn:BB}
For which critical boootstrap percolation models do we have that, for all $\varepsilon>0$,
\[\lim_{n\to\infty}\bbP\left(\frac{|[A(\tau-1)]|}{n^d}>\varepsilon\right)=0?\]
We conjecture that this is the case when $d=2$ and $(\cK_s,r_s)$ is defined by a $\pi/2$-rotation invariant convex $K$ and any $s\ge 1$ (as opposed to $s$ large enough).
\end{question}

We suspect that Question~\ref{qn:BB} has a positive answer for essentially all critical models. 
However, a trivial counterexample is given by the two-dimensional $\square$ model on a sublattice, that is,
\begin{align*}
\cK_\diamond&{}=\{(1,1),(-1,1),(-1,-1),(1,-1)\},&r_\diamond&{}=2.
\end{align*}
In this case, our methods readily yield that $|[A(\tau-1)]|/n^2$ converges to $1/2$ in probability as $n\to\infty$ for even $n$ and to $0$ for odd $n$.\footnote{Indeed, the unit distance graph on the underlying torus is bipartite if and only if $n$ is even; thus, roughly speaking, the process decomposes into two independent copies of $(\cK_{\square},r_{\square})$.} We believe such periodicity issues to be the only obstruction (see also \cite{Duminil-Copin24}*{Figure~6} for a less straightforward problematic example). 
Similarly to other critical results discussed above, the only difficulty is establishing a result like Proposition~\ref{prop:long:rectangle}, whose proof already contains several robust steps.
Contrary to other deterministic statements like Lemma~\ref{lem:single:site} and Corollary~\ref{cor:AL}, we expect this to be possible in significant generality, as shown by our results.

\section{Outline}
\label{sec:outline}
In this section, we outline the proof of Theorem~\ref{th:main}.
\subsection{The Aizenman--Lebowitz approach}
Let us first explain the proof of Theorem~\ref{th:main} for the $\square$ model, the $\triangle$ one being very similar. We use relatively standard ideas going back to \cite{Aizenman88}. 
The proof is divided into two parts -- a probabilistic and a combinatorial one, contained in Sections~\ref{sec:proba} and~\ref{sec:AL4}, respectively.

\subsubsection{Probabilistic argument}
\label{subsec:proba}
We start by recalling from \cite{Aizenman88} that the percolation time $\tau$ is of order $n^2/\log n$.
We consider the closure at time $\tau-1$. A site may be in the closure for one of the following reasons:
\begin{itemize}
    \item the site itself is initially infected;
    \item there are two initial infections very close to the site;
    \item there is a group of several initial infections close to each other and not very far from the site;
    \item none of the above.
\end{itemize}
The first three cases are easy to control: the first item contributes the first order term $\tau$ in \eqref{eq:main}; the second one contributes the correction $C\tau/\log n$, while the third one is a lower order correction. 
We therefore focus on the fourth case. In this case, infection needs to penetrate a large zone poor in infections around the target site in order to infect it.

At this point, we arrive at the deterministic combinatorial input discussed in the next section. 
Namely, if infection does penetrate such a zone, it necessarily creates a fully infected rectangle of size proportional to the size of the infection-deprived zone. 
However, at the typical values of $\tau-1$, such a large rectangle is likely to find initial infections at distance 1 from its longer sides. 
Notice that, in this configuration, a fully infected new row or column can be added to the rectangle (see Lemma~\ref{lem:single:site}). Therefore, the initial infections typically present allow the rectangle to grow one line at a time and infect the entire torus. 
This yields the contradiction $\tau-1\ge \tau$. Thus, the fourth scenario above can only occur if the environment of infections around the rectangle is atypical, which is unlikely.

\subsubsection{Combinatorial argument}
\label{subsec:combi}
The probabilistic part of the argument reduces Theorem~\ref{th:main} to proving that, if infection reaches the centre of a box with only very small and well-separated groups of initial infections and no initial infections close to its centre, then there has to be a large rectangle in the closure. To prove this, we use the so-called rectangles process (see Corollary~\ref{cor:AL}). It states that the closure of a set can be determined by merging axis-parallel rectangles, as soon as they are close, by replacing them by the smallest rectangle containing them. 

In order for infection to penetrate a region poor in initial infections, it needs to supply help from the boundary of that region. 
But the rectangles process then guarantees that an infected rectangle will be formed touching both the boundary of the region and the site we need to infect in its middle. This is exactly the desired deterministic statement concluding the proof.

Unfortunately, the rectangles process is only available for the $\square$ model. A hexagon process applies to the $\triangle$ model, but one cannot go much further than this. 
Indeed, in Section~\ref{sec:AL4}, we provide an example showing that, already for $\ell^1$ balls of radius 4, there is no hope for the above proof: indeed, in this case, infection is able to travel unbounded distances without creating any solid completely infected region. 
Additionally, solid infected regions can arise in a more complicated way than by merging comparably sized smaller regions. It is this issue that we seek to address.

\subsection{Large neighbourhoods}
The probabilistic part of the argument described in Section~\ref{subsec:proba} is very general and requires no significant change (up to using \cite{Bollobas23} instead of \cite{Aizenman88}). 
Thus, our goal is proving the same deterministic statement about penetrating sparsely infected regions as in Section~\ref{subsec:combi}. 
However, instead of looking for a simple characterisation of closed infected sets, we use the fact that infection is somehow propagating through a region with few initial infections. 
We mostly rely on two types of arguments: looking forward or backward in time. Looking forward in time refers to taking (part of) the closure of infections we have already produced. Looking backwards in time refers to asking how a given set of infections appeared, usually yielding further infections at its boundary.

On the high level, we will ensure the sparsity of $[A(\tau-1)]$ by studying the mechanism allowing the centre of a large region $R$ poor in infection to be infected.
By looking backwards in time, we show that the above event requires a large connected set $P\subset R$ to be included in $[A(\tau-1)]$, hopefully allowing $P$ to spread the infection further when we look forward in time.

Unfortunately, matters are much more complicated than before since it is hard to:
\begin{itemize}
    \item find a suitably large infected region $P$ to start this procedure with;
    \item keep its shape from becoming too degenerate or irregular, given that we have no control over the positions of the infections when looking backwards in time;
    \item deal with the fact that, while sparse, initial infections are not completely absent from $R$ and it may happen that precisely these rare initial infections serve as entry points in $P$.
\end{itemize} 
With these issues in mind, the proof of Section~\ref{sec:penetration} proceeds in several steps described below, each using somewhat different arguments. 
In the rest of this section, for concreteness, we focus on the neighbourhood $\cK_s$ given by a large Euclidean ball and fix a set $A\subset\bbZ^2$ of initial infections with $0\in[A]$, $\min_{x\in A} \|x\|_2$ large and no groups of more than a fixed number (as $s\to\infty$) of infections close to one another. 
Indeed, this can be obtained by slightly changing our target site to ensure that it is fairly far from $A$.

\subsubsection{Finding an initial seed}
The first step (Section~\ref{subsec:seed} and Lemma~\ref{lem:big square}) consists in finding a \emph{seed}: a fully infected square of side length $C_2s$ for some large $C_2$. This is based on a double-counting argument taking into account the following fact. Each site in $[A]\setminus A$ uses $r_s$ infections to become infected but can contribute to the infection of at most $|\cK_s|-r_s-1$ other sites. 
Since $r_s=|\cK_s|/2-\Theta(s)$ and $|\cK_s|=\Theta(s^2)$, if there are no initial infections close to the origin, 
on average, most sites in $[A]$ close to the origin should have all but $O(s)$ of the sites in their neighbourhood in $[A]$ (see Lemma~\ref{lem:good2}), 
which easily yields a completely infected ball of radius $s/50$ around them (see Lemma~\ref{lem:good1}). In particular, only a small fraction of $[A]$ is not surrounded by such an infected ball. 
Taking into account the effect of the boundary provides us with a square of side length $C_2s$ in which the density of such sites is so high that it actually forbids the presence of any non-infected site.
 
\subsubsection{Adding bumps to convex sets}
We have obtained a seed $P$ close to the origin and we want to make it grow, still within the zone completely free of initial infections around the origin. 
The seed is so small that its curvature is still visible on scale $s$, so making $P$ grow requires significant help from infections outside $P$. Let us imagine $P$ is a ball for the moment. 
Looking at the first site $z\in P$ to become infected, we necessarily find $\Theta(s^2)$ infections in $(z+\cK_s)\setminus P$. 
Looking forward in time now, we discover that a small ball of radius $\varepsilon s$ becomes infected at the point of the boundary of $P$ closest to $z$. 
We could view this `ball with a bump' as our new region, but iterating this approach quickly leads to problems as bumps on bumps lead to infected sets which are hard to control.
Therefore, we always maintain our infected region to be a convex set which does not become `too thin' (see Definition~\ref{def:polygon}). 

In order to preserve these features, we localise where the bump appears. Namely, we look for a bump near the endpoints of a vertical chord $xy$ 
such that, first, the slope of $\partial P$ around each endpoint is not too large and, second, this slope barely changes at distance $C_4s$ from $xy$ for some large constant $C_4$ (see Definition~\ref{def:admissible}).
Then, the part of $P$ at distance at most $3s$ from the line $xy$ 
is approximately a long thin trapezoid $S$.
We consider the first site $z\in S$ to become infected. It is not hard to show that, by our assumptions on $\cK_s$ (in particular, the $\pi/2$-rotation invariance), if $S$ were an infinite vertical strip, infection could not enter it. Therefore, $z$ is close to the top or bottom boundary of $S$. 
Combining this observation with the slope assumptions, we can show (see Lemma~\ref{lem:blob}) that a bump $B$ of size $\varepsilon s$ appears on the corresponding top or bottom boundary of $S$.

Simply adding this bump to $P$ would break convexity. Thus, we need to look how $P\cup B$ spreads forward in time and hope to extract a larger convex set from it. 
Unfortunately, if $P$ is rather flat but then has a corner (e.g.\ if $P$ is a large square), the effect of the bump $B$ may have trouble reaching that corner. 
For this reason, we may need to cut out parts of the convex set near such corners. Cutting out parts of $P$ has the undesirable effect that the leftmost and rightmost slices of the resulting convex set may become too thin, so additional local modifications around them are necessary. Moreover, these need to be done carefully so that the total area lost is smaller than the gain achieved by using the bump. This is the most technical part of the proof.

\subsubsection{Extending droplets once they are large}
So far, we managed to produce a fully infected rectangle $R$ of length much larger than $s^3$ and width $s$. 
We use a slicing argument similar to the previous section (see Lemma~\ref{lem:band}) in order to find infections next to the long sides of $R$ and away from the short ones. 
This allows $R$ to become large in both directions, eventually yielding an infected square of side length larger than $s^3$ (see Lemma~\ref{prop:very:big:square}). 
Out of this square, we extract a `droplet', that is, a convex set with $O(s^2)$ sides in well-chosen prescribed directions and length at least $Cs$ for a suitably large $C$ (in fact, we use different $C$ for different types of directions).
Such droplets can grow as follows. Each growth step, called \emph{extension}, consists in keeping all corners except two consecutive ones $v,w$ fixed, and adding a trapezoid with base $vw$ to the droplet while keeping the directions of all sides unchanged. 
If the normal $u$ to the side $vw$ satisfies $|\{x\in\cK_s:\<u,x\>< 0\}
|\ge r_s$, extension works automatically by taking the closure of the droplet. 
Such sides are called \emph{unstable}. When the length of all unstable sides come close to the threshold $10s$, we turn our attention to long stable sides. For these, we require an additional infection outside the droplet close to its corresponding side in order to perform the extension. 

We show that such additional infections are necessarily present, as long as initial infections are sparse, as assumed (see Lemma~\ref{lem:terminate}). 
To establish this, we argue by contradiction. Indeed, if all the unstable sides are short and none of the long stable sides has a helping infection, 
there is a frame around the droplet with no element of $[A]$ except around short sides and corners. 
But this small boundary condition together with the sparse initial infections inside the droplet cannot suffice to infect the entire droplet. 
As a result, the growth is uninterrupted until the droplet increases its size enough to exit the region guaranteed to be poor in initial infections. 
This is exactly the deterministic statement we set off to prove. 

\section{Proof of Theorem~\ref{th:main}}
\label{sec:proba}
In this section, we reduce Theorem~\ref{th:main} to a deterministic statement (Proposition~\ref{prop:long:rectangle}). 
\subsection{Preliminaries}
We define the \emph{radius} of the neighbourhoods by setting 
\begin{align}
\label{eq:def:norm:K}
\|\cK_s\|&{}=s,&\|\cK_\square\|&{}=1,&\|\cK_\triangle\|&{}=\sqrt 2.\end{align}
We also define the half-planes
\begin{align*}
\bbH_u(l)=\left\{x\in\bbR^2:\<x,u\>< l\right\}\qquad\qquad\text{and}\qquad\qquad\Hb_u(l)=\left\{x\in\bbR^2:\<x,u\>\le l\right\}
\end{align*}
for a \emph{direction} $u\in S^1=\{v\in\bbR^2:\|v\|_2=1\}$ in the \emph{unit circle} and $l\in\bbR$. If $l=0$, we omit it in the above notation. 
We say that a direction $u\in S^1$ is \emph{unstable} (for a threshold bootstrap percolation model $(\cK,r)$) if $|\bbH_u\cap\cK|\ge r$, and \emph{stable} otherwise. 
We denote by $\cS = \cS(\cK,r)\subset S^1$ the set of stable directions. Setting
\begin{align}
\nonumber\cS_\square&{}=\{(1,0),(0,1),(-1,0),(0,-1)\},\\
\label{eq:stable}\cS_\triangle&{}=\cS_\square\cup\left\{(1/\sqrt 2,1/\sqrt 2),(-1/\sqrt 2,-1/\sqrt 2)\right\},\\
\nonumber\cS_\boxtimes&{}=\cS_\triangle\cup\left\{(1/\sqrt 2,-1/\sqrt 2),(-1/\sqrt 2,1/\sqrt 2)\right\},
\end{align}
notice that, for the models we consider, it holds that $\cS=\cS_\square$ if $(\cK,r)=(\cK_\square,r_\square)$ and $\cS=\cS_\triangle$ if $(\cK,r)=(\cK_\triangle,r_\triangle)$. 
The next lemma loosely follows \cite{Gravner99}.
\begin{lemma}
\label{lem:Ksrs}
    Let $(\cK,r)=(\cK_s,r_s)$ for some $s\ge 4$ and $\pi/2$-rotation invariant $K\subset \bbR^2$. Then
    \begin{align*}
    r&{}=1+|\{(x_1,x_2)\in\cK:x_2<0\}|,&\cS&{}\in\{\cS_\square,\,\cS_\boxtimes\}.\end{align*}
\end{lemma}
\begin{proof}
    Recall that $\cK=(sK)\cap\bbZ^2$ with $\max\{\|k\|_2:k\in K\}=1$. Fix $x=(x_1,x_2)\in\cK\setminus((\{0\}\times\bbZ)\cup(\bbZ\times \{0\}))$ with $x_1, x_2$ coprime. Define $x^\perp=(-x_2,x_1)\in\cK$ and $\alpha_x=\max\{\alpha>0:\alpha x/\|x\|_2\in K\}\le 1$.
    Note that, for every $x$ as above, $\alpha_x\ge 1/\sqrt{2}$: indeed, the property $\max\{\|k\|_2:k\in K\}=1$ and the $\pi/2$-rotational symmetry of $K$ ensure that $K$ contains the ball with radius $1/\sqrt{2}$ centred at the origin.
    Moreover, setting $h_x=|y\in\cK_s:\<y,x^\perp\><0|$ and $l_x=|\cK\cap \bbR x|$, we have $h_x=(|\cK|-l_x)/2$. At the same time, 
    \begin{equation}
    \label{eq:lxl10}l_x=1+2\lfloor s\alpha_x/\|x\|_2\rfloor\le 1+2\lfloor s/\sqrt 2\rfloor\le 1+2\lfloor s\alpha_{(1,0)}\rfloor=l_{(1,0)}.
    \end{equation}
    Together with \eqref{eq:defrs}, this proves the desired expression for $r_s$. Moreover, in \eqref{eq:lxl10}, we cannot have equality if $\|x\|_2\ge \sqrt 5$ since $s/\sqrt 5\le s/\sqrt 2-1$. Therefore, all directions in $S^1\setminus\cS_\boxtimes$ are unstable, while those in $\cS_\square$ are necessarily stable, so the result follows by rotation-invariance.
\end{proof}

\subsection{The proof}\label{sec:theproof}

We next prove Theorem~\ref{th:main} assuming two technical results. The first one, Lemma~\ref{lem:long:rectangle:suffices}, follows from simple ingredients from \cite{Bollobas15a}. Since they will be introduced only in Section~\ref{subsec:extensions}, we delegate the proof of the lemma to Appendix~\ref{app}. Secondly, Proposition~\ref{prop:long:rectangle} is our main technical result and its proof will be the subject of Section~\ref{sec:AL4} for the $\square$ and $\triangle$ models, and of Section~\ref{sec:penetration} for the $\cK_s$ models.

\begin{restatable}{lemma}{lemBSU}
\label{lem:long:rectangle:suffices}
Fix a model $(\cK,r)$ as in Theorem~\ref{th:main}, a suitably small constant $\varepsilon = \eps(\cK, r)>0$, large enough integer $n$ and $d\in[1/\varepsilon^4,\varepsilon n]$. 
Let $R\subset (\bbR/n\bbZ)^2$ be a rectangle with side lengths $\|\cK\|$ and $d$, and sides perpendicular or parallel to 
directions in $\cS=\cS(\cK,r)$. 
Let $A\subset\bbT=(\bbZ/n\bbZ)^2$ satisfy $R\cap\bbZ^2\subset A$. Further suppose that $A$ intersects every (integer point) straight-line segment in $\bbT$ of length $\varepsilon^3 d$ and perpendicular to a direction in $\cS$.
Then, $[A]=\bbT$.
\end{restatable}

Given a real number $\kappa\ge 1$, we say that a set of vertices $V\subset \bbT$ is \emph{$\kappa$-connected} if, for every two vertices $u,v\in V$, there are vertices $v_0 = u, v_1, \ldots, v_k = v$ in $V$ such that $\|v_{i-1}-v_i\|_2\le \kappa$ for all $i\in \{1,\dots,k\}$.

\begin{restatable}{prop}{propmain}
\label{prop:long:rectangle}
Fix a convex $\pi/2$-rotation invariant set $K\subset\bbR^2$ defining the neighbourhood $\cK_s$ for $s\ge 1$. Let $(\cK,r)\in\{(\cK_\square,r_\square),(\cK_\triangle,r_\triangle),(\cK_s,r_s)\}$.
Assume $s$ is large enough and fix a sufficiently large constant $C = C(\cK,r) > 0$. Let $d\ge C^4$ and $\Lambda_d=[-d,d]^2\cap\bbZ^2$. Consider $A\subset\bbZ^2$ such that $A\cap \Lambda_d$ does not contain a $C$-connected set of $8$ sites and $|A\cap \Lambda_{C^3}|\le 1$. If $0\in [A]\setminus[A\cap \Lambda_d]$, then $[A]$ contains a rectangle with dimensions $\|\cK\|\times d/9$ and sides perpendicular to directions in $\cS$.
\end{restatable}

\begin{proof}[Proof of Theorem~\ref{th:main}, assuming Lemma~\ref{lem:long:rectangle:suffices} and Proposition~\ref{prop:long:rectangle}]
Fix a model $(\cK,r)$ as in Theorem~\ref{th:main} and any sufficiently small constant $\varepsilon > 0$ depending on $\cK$ and $r$. 
By~\cite{Bollobas23} (see the proofs of Theorem~4.1 and of Theorem~7.1 therein),
setting\footnote{Here and below, we omit integer parts and ignore related divisibility issues for simplicity.}
\begin{align*}T_-&{}=\frac{\varepsilon n^2}{\log n},& T_+&{}=\frac{n^2}{\varepsilon \log n},&\cE_1&{}=\{T_-<\tau<T_+\},\end{align*}
we have $\bbP(\cE_1)\rightarrow 1$ as $n\to\infty$.

We say that $x\in\bbT$ is \emph{isolated for a set $A\subset\bbT$} if there is no $1/\eps$-connected set $I\subset A$ containing $|I|=8$ sites, each at distance at most $(\log n/\varepsilon)^3$ from $x$.
When $A=A(t)$ for some $t\le n^2$, we say that a site is \emph{$t$-isolated} for short. Note that the property of being $t$-isolated is monotone decreasing: if $t'<t''$ and $x\in\bbT$ is $t''$-isolated, then $x$ is also $t'$-isolated.
By a union bound, for every $x\in\bbT$,
\begin{equation}
\label{eq:Pxtisolated}
\bbP\left(x\text{ is not $(2T_+)$-isolated}\right)\le \frac{5(\log n)^{6}}{\varepsilon^{6}}\cdot \frac{200^7}{\eps^{14}}\cdot\left(\frac{2T_+}{n^2}\right)^{8} \le \frac{1}{2\varepsilon^{29}(\log n)^2} =: q,
\end{equation}
where we used the following facts:
\begin{itemize}
    \item there are at most $(2(\log n/\varepsilon)^3+1)^2\le 5(\log n)^{6}/\eps^{6}$ sites within distance $(\log n/\varepsilon)^{3}$ from $x$ (serving as choices for the first site in $I$);
    \item for each $i\in\{2,\dots,8\}$, there are at most $(1+14/\eps)^2\le 200/\eps^2$ choices for the $i$-th site in $I$;
    \item the probability that each of the chosen sites is in $A(2T_+)$ is bounded from above by
    \[\prod_{i=0}^7\frac{2T_+-i}{n^2-i}\le\left(\frac{2T_+}{n^2}\right)^{8}.\]
    \end{itemize}

Set $p=1.5T_+/n^2$ and define $\tilde A$ to be a random subset of $\bbT$ such that the events $\{x\in\tilde A\}_{x\in \bbT}$ are independent and each of them has probability $p$.
We define the events $\cF_1$ that $T_+\le |\tilde A|\le 2T_+$, and $\cF_2$ that at most $2n^2q$ sites are non-isolated for $\tilde A$. Note that, by~\eqref{eq:Pxtisolated} and Chernoff's inequality (see Theorem~2.1 in \cite{Janson00}), the expected number of non-isolated sites for $\tilde A$ is at most
\[n^2q+n^2\mathbb P(|\tilde A|\ge 2T_+)\le n^2q+n^2\mathbb P(\cF_1^c)=(1+o(1))n^2q\qquad \text{as $n\to\infty$.}\] 
Thus, since resampling the state of a site alters the number of isolated sites for $\tilde A$ by at most $(\log n)^{7}$, an application of the McDiarmid inequality (see~\cite{McDiarmid89}) shows that $\mathbb P(\cF_2^c)=o(1)$ as $n\to\infty$.
Combining this bound with \eqref{eq:Pxtisolated} and the fact that, under the event $\cE_1$, we have $2n^2q\le \tau/(\eps^{30} \log n)$. Thus,
\begin{multline}
\label{eq:PxT}
\bbP\left(\cE_1\cap \left\{\left|\left\{x\in\bbT:x\text{ is not $(\tau-1)$-isolated}\right\}\right|\ge \tau/ \left(\varepsilon^{30
}\log n\right)\right\}\right)\\
\le \bbP\left(\left|\left\{x\in\bbT:x\text{ is not $T_+$-isolated}\right\}\right|\ge 2n^2q\right)\le \mathbb P(\cF_1^c)+\mathbb P(\cF_2^c)=o(1).
\end{multline}

Anticipating the application of Proposition~\ref{prop:long:rectangle}, we call a site $x\in \bbT$ \emph{$t$-lonely} if $|A(t)\cap (x+\Lambda_{1/\varepsilon^3})|\le 1$.
To estimate the number of lonely sites, for every $x\in \bbT$, note that
\begin{equation}
\mathbb P(\{\text{$x$ is not $(2T_+)$-lonely}\})\le \bigg(1+\frac{2}{\varepsilon^3}\bigg)^2 
\left(\frac{2T_+}{n^2}\right)^2 \le \frac{1}{2\eps^{9} (\log n)^2}.
\end{equation}
A similar use of Chernoff and McDiarmid inequalities to the one showing~\eqref{eq:PxT} implies that, for suitably small $\eps$ and $n\to\infty$, 
\begin{equation}\label{eq:2sites}
\bbP\left(\cE_1\cap \bigg\{\left|\left\{x\in\bbT:x\text{ is not $(\tau-1)$-lonely}\right\}\right|\ge \frac{\tau}{\eps^{10}
\log n}\bigg\}\right)=o(1).
\end{equation}

\begin{claim}\label{cl:cases}
For every $(\tau-1)$-lonely $(\tau-1)$-isolated site $x\in\bbT\setminus A(\tau-1)$, we have \[x\not \in\left[A(\tau-1)\cap\left(x+\Lambda_{(\log n/\varepsilon)^3}\right)\right].\]
\end{claim}
\begin{proof}[Proof of Claim~\ref{cl:cases}]
Fix $x$ as required and set $X=A(\tau-1)\cap(x+\Lambda_{(\log n/\varepsilon)^3})$. Assume for a contradiction that $x\in[X]$.
Since $x$ is $(\tau-1)$-isolated, each $1/\varepsilon$-connected component $\cC$ of $X$ contains at most $7$ sites. Since $\cS\supset\cS_\square$ (recall \eqref{eq:stable} and Lemma~\ref{lem:Ksrs}), we have that $[\cC]$ is contained in the set $R$ resulting from the following algorithmic procedure following \cite{Bollobas15a}*{Definition~4.2}. 
Set $R_0=\{\{c\}:c\in \cC\}$. For integer $i\ge 0$, if two distinct elements $\rho,\rho'$ of $R_i$ are at distance at most $2\|\cK\|$, set $R_{i+1}=(R_i\setminus\{\rho,\rho'\})\cup\{\rho\vee\rho'\}$, 
where $\rho\vee\rho'$ is the smallest axis-parallel rectangle containing $\rho\cup\rho'$. 
The pair of elements $\rho,\rho'$ are chosen arbitrarily if multiple choices exist. If no such $\rho,\rho'$ exist, we set $R_{i+1}=R_i$. 
In particular, since $|R_0|=|\cC|$, we have that $(R_i)_{i\ge |\cC|-1}$ is constant.
The set $R$ is defined as the union of the elements of $R_{|\cC|-1}$.
Then, it is not hard to see (e.g.\ following \cite{Bollobas15a}) that $R\supset [\cC]$ and that each $\rho\in R_{|\cC|-1}$ is an axis-parallel rectangle with longest side at most $2^{|\cC|}\|\cK\|-2\|\cK\|$ containing an element of $\cC$. 
Since the distance between connected components is at least $1/\varepsilon> 2^{8}\|\cK\|$, we have that $[X]=\bigcup_\cC [\cC]$. But each $[\cC]$ has diameter at most $8/\varepsilon$, so if $x\in[\cC]\setminus\cC\neq \varnothing$, then $\cC$ contains at least two sites at distance at most $8/\varepsilon$ from $x$. This contradicts the assumption that $x$ is $(\tau-1)$-lonely.
\end{proof}

We next consider the event $\cE_2$ that, for each $u\in\cS$ and (integer point) straight-line segment perpendicular to $u$ of length $\varepsilon (\log n)^3$ with non-empty intersection with $\bbT=(\bbZ/n\bbZ)^2$, that segment contains an element of $A(T_-)$. Using the inclusion $\cS\subset\cS_\boxtimes$ (recall \eqref{eq:stable} and Lemma~\ref{lem:Ksrs}) 
implying that every such segment contains at least $\varepsilon (\log n)^3/2$ integer points, we have
\[\bbP\left(\cE_2^{c}\right)\le |S_{\boxtimes}| n^2\prod_{i=1}^{\varepsilon(\log n)^3/2}\left(1-\frac{T_-}{n^2-i+1}\right)\le 8n^2\left(1-\frac{T_-}{n^2}\right)^{\varepsilon(\log n)^3/2}=o(1)\]
as $n\to\infty$. Consequently, on $\cE_1\cap\cE_2$, any segment as above contains an element of $A(\tau-1)\supset A(T_-)$.

We next argue that, on the event $\cE_1\cap\cE_2$, no rectangle with dimensions $\|\cK\|\times (\log n)^3/\varepsilon^2$ and sides perpendicular or parallel to directions in $\cS$ can be contained in $[A(\tau-1)]$. 
Indeed, by assuming the contrary, the event $\cE_2$ allows us to apply Lemma~\ref{lem:long:rectangle:suffices} with $\eps^3 d = \eps (\log n)^3$ which then implies that $[A(\tau-1)] = \bbT$, contradicting the definition of $\tau$. Combining this fact with Proposition~\ref{prop:long:rectangle} applied with $d=9(\log n)^3/\varepsilon^2$ and Claim~\ref{cl:cases} yields that, on the event $\cE_1\cap\cE_2$, any $x\in [A(\tau-1)]$ is either in $A(\tau-1)$, not $(\tau-1)$-isolated or not $(\tau-1)$-lonely. The conclusion of Theorem~\ref{th:main} is then implied by \eqref{eq:PxT} and \eqref{eq:2sites}.
\end{proof}

\section{The Aizenman--Lebowitz approach}
\label{sec:AL4}
In this section, we prove Proposition~\ref{prop:long:rectangle} in the simpler $\square$ and $\triangle$ cases, using the classical Aizenman--Lebowitz approach \cite{Aizenman88}. 
We say that two distinct sites $x,y\in\bbZ^2$ are \emph{neighbours} if $x-y\in\cK$ (that is, if they are neighbours in the graph sense on the square or triangular lattice). 
Recall that $\cS\subset S^1$ denotes the set of stable directions defined in~\eqref{eq:stable}.
To simplify notation, we will refer to these directions as $\rightarrow,\uparrow,\leftarrow,\downarrow,\nearrow,\swarrow$. For $l = (l_u)_{u\in \cS}\in\bbZ^\cS$, the \emph{droplet} with \emph{radii} $l$ is the set
\[D(l)=\bbZ^2\cap\bigcap_{u\in\cS} \Hb_u(l_u\rho_u),\]
where $\rho_u=1/\min\{x>0:xu\in\bbZ^2\}$ (in particular, $\rho_u=1$ for the axis directions and $\rho_u=1/\sqrt 2$ for the diagonal ones).
Thus, droplets are rectangles (for $\square$) or hexagons (for $\triangle$) with possibly degenerate sides. 
In order to have uniquely defined radii for any non-empty droplet, we always assume $l\in\bbZ^\cS$ above to be minimal coordinate-wise among all radii giving rise to the same droplet.
A droplet $D$ is \emph{internally filled} by $A\subset\bbZ^2$ if
\[[D\cap A]=D.\]
Notice that, since $\cS$ is the set of stable directions, we always have $[D\cap A]\subset D$.

The proof of Proposition~\ref{prop:long:rectangle} hinges on the following elementary but fragile property. While it is classical and was immediately noticed when the models were introduced, we include a proof for completeness.
\begin{lemma}
\label{lem:single:site}
Let $(\cK,r)\in\{(\cK_\square,r_\square),(\cK_\triangle,r_\triangle)\}$. Let $D$ be a droplet and $x\in\bbZ^2$ be a neighbour of $y\in D$. Let $D'$ be the smallest droplet (with respect to inclusion) containing $D$ and $x$. Then, $[D\cup\{x\}]=D'$.
\end{lemma}

\noindent
Note that $D'$ in Lemma~\ref{lem:single:site} is uniquely defined since the intersection of all droplets containing $D$ and $x$ is a droplet~itself.

\begin{proof}
If $D$ is empty, there is nothing to prove.
Fix a non-empty droplet $D$ with radii $l\in\bbZ^\cS$. We assume that $x\not\in D$, as otherwise $D'=D$ and there is nothing to prove. We begin with the simpler $\square$ case. By the symmetries of the model, we assume that $x=0$, $l_{\uparrow}=-1$, $l_{\downarrow}\ge 1$, $l_\leftarrow\ge 0$, $l_\rightarrow\ge 0$, that is, $x$ is above the top-most line of $D$ and at distance $1$ from it. Then, $l'=(l_\uparrow+1,l_\downarrow,l_\leftarrow,l_\rightarrow)$ defines the desired droplet $D'=D(l')$. It remains to show that $[D\cup\{x\}]\supset \Lambda=\{-l_\leftarrow,\dots,l_\rightarrow\}\times \{0\}$. 
This is the case since $x$ is infected and, as each site in $\Lambda$ has a neighbour in $D$, the induced process on the top-most line $\Lambda$ of $D'$ is a 1-neighbour bootstrap percolation on a connected set.

Moving on to the $\triangle$ case, using the (triangular lattice) symmetries of the model, we assume that $x=0$, $l_\uparrow=-1$, $l_\downarrow\ge 1$, $l_\leftarrow\ge 0$, $l_\rightarrow\ge 0$, $l_\nearrow\ge -1$, $l_\swarrow\ge 1$, that is, the neighbour $y=(0,-1)$ of $x$ lies on the top-most line of $D$.

To begin with, note that the case $l_\leftarrow=l_\rightarrow=0$ is trivial since then $D'=D\cup\{x\}$ and there is nothing to prove. 
Suppose that $l_\leftarrow\ge 1$, so also $l_\swarrow\ge 2$ by minimality of $l$ (if $l_\swarrow=1$, then we could decrease $l_\leftarrow$ to $0$ without changing $D$, since $l_\uparrow=-1$).
The set $\Lambda_\uparrow=\{(y,0):y\in\{-l_\leftarrow,\dots,l_\nearrow\}\}$ becomes infected as above starting from $x$, since each of its sites has two neighbours in $D$. 
We are done unless $x\not\in\Lambda_\uparrow$, that is, if $l_\nearrow=-1$. 
In this case, we can show by induction that the set $\Lambda_\nearrow=\{(z,-z):z\in\{0,\dots,l_\rightarrow\}\}\ni x$ becomes infected. To see this, notice that $\Lambda_\nearrow$ is connected and
\[\Lambda_\nearrow+\{(-1,0),(0,-1)\}=\left\{(z-1,-z):z\in\{0,\dots,l_\rightarrow+1\}\right\}\subset D\cup\{(-1,0)\}\subset D\cup\Lambda_\uparrow,\] 
since $l_\downarrow\ge l_\rightarrow+1$ by minimality (otherwise, decreasing $l_\rightarrow$ by 1 would not change the droplet $D$, since $l_\nearrow=-1$).

Suppose that $l_\leftarrow=0 <l_\rightarrow$ instead. We need two cases again. 
If $l_\nearrow\ge 0$, then it suffices to infect $\Lambda_\uparrow=\{0,\dots,l_\nearrow\}\times\{0\}$: this is possible starting from $x$ since each $z\in\Lambda_\uparrow$ has $z+\{(0,-1),(1,-1)\}\subset D$.
Finally, if $l_\nearrow=-1$, the set $\Lambda_\nearrow=\{(z,-z):z\in\{0,\dots,l_\rightarrow\}\}\ni x$ is infected as above.
\end{proof}

Next, we present a standard algorithm in bootstrap percolation, which is the basis of the Aizenman-Lebowitz approach.

\begin{algorithm}[Droplet algorithm]
\label{algo:droplet}
Fix $(\cK,r)\in\{(\cK_\square,r_\square),(\cK_\triangle,r_\triangle)\}$ and a finite set $A\subset\bbZ^2$. Start with a collection $\cD^{(0)}=\{\{a\}:a\in A\}$ of droplets given by the singletons of $A$.
For integer $t\ge 0$, arbitrarily fix $x\in\bbZ^2$ outside the droplets in $\cD^{(t)}$ and a set of droplets
\[\cD_x \subset \{D\in \cD^{(t)}: D\cap (x+\cK)\neq\varnothing\}\quad\text{such that}\quad \bigg|(x+\cK)\cap\bigcup_{D\in\cD_x} D\bigg|\ge r.\] 
Further, denote by $D_x$ the smallest droplet containing $\{x\}\cup \bigcup_{D\in\cD_x}D$ and set $\cD^{(t+1)}=(\cD^{(t)}\setminus\cD_x)\cup \{D_x\}$. 
When no such $x$ and $\cD_x$ can be found, the algorithm terminates.
\end{algorithm}

\begin{corollary}[Droplet algorithm]
\label{cor:AL}
For $(\cK,r)\in\{(\cK_\square,r_\square),(\cK_\triangle,r_\triangle)\}$ and any finite set $A\subset\bbZ^2$, the final collection of the droplet algorithm consists of droplets which are internally filled (by $A$) with union $[A]$.
\end{corollary}
\begin{proof}
We first prove by induction on the number of steps that all droplets in the current collection are internally filled. To see this, fix the site $x$ and the droplets $D_1,\dots, D_k$ (with $k\ge 1$) 
merged at this step. 
By the induction hypothesis, $\Delta=\bigcup_{i=1}^k D_i\subset[A]$ and $|(x+\cK)\cap\Delta|\ge r$, so $x\in[A]$. Let $(x_i)_{i=1}^m$ be an enumeration of the sites in $\Delta\cup\{x\}$ such that $x_1\in D_1$ and, for every $i\in\{1,\dots,m-1\}$, $x_{i+1}$ has a neighbour among $x_1,\ldots,x_i$. Then, by Lemma~\ref{lem:single:site} and induction, for every $i\in \{1,\dots,m\}$, the smallest droplet $D^{(i)}$ containing $D_1\cup \{x_j:j\le i\}$ for all $i\in\{1,\dots,m\}$ satisfies $D^{(i)}\subset [A\cap\Delta]$.
In particular, $D^{(m)}\supset\Delta$ is internally filled, which completes the induction. 

Recall that, for each droplet $D$, we have $[D]=D$, that is, for any $x\in \bbZ^2\setminus D$, we have $|(x+\cK)\cap D|<r$. Consequently, $\cD_x$ is never be a singleton, so the number of droplets $|\cD^{(t)}|$ is strictly decreasing in $t$. Hence, the algorithm terminates and outputs a final collection $(D_i')_{i=1}^{\ell}$ of droplets. We established that $\Delta'=\bigcup_{i=1}^{\ell} D_i'$ is contained in $[A]$, as all droplets are internally filled. 
However, by construction, $[\Delta']=\Delta'$ (otherwise, the first site to become infected yields a merging step, thus contradicting the assumption that the algorithm has terminated). Since $A\subset \Delta'$, we also have $[A]\subset [\Delta']$. Finally, $A\subset \Delta'=[\Delta']\subset [A]\subset [\Delta']$, which implies that $\Delta'=[A]$ and finishes the proof.
\end{proof}

With Corollary~\ref{cor:AL} at hand, we are ready to prove the easy part of Proposition~\ref{prop:long:rectangle}.
\begin{proof}[Proof of Proposition~\ref{prop:long:rectangle} for $(\cK,r)\in\{(\cK_\square,r_\square),(\cK_\triangle,r_\triangle)\}$.]
First of all, note that some finite subset of $A$ is sufficient to infect the origin. Thus, up to extracting such a subset from $A$ if necessary, assume that $A$ is a finite set.
We apply Corollary~\ref{cor:AL} and let $(D_i)_{i=1}^\ell$ be the final collection of droplets. Since $0\in[A]=\bigcup_{i=1}^\ell D_i$, we may assume that $0\in D_1$. Yet, $0\not\in[A\cap\Lambda_d]$ and $D_1$ is internally filled, so $D_1\not\subset\Lambda_d$. 

We will extract the desired rectangle from $D_1$. Since $\|\cK_\square\|=1$ and $D_1$ is an axis-parallel rectangle, this is always possible for this update family. 
In the $\triangle$ case, $\|\cK_\triangle\|=\sqrt 2$ and the only possible problem we may encounter is $D_1$ being a single diagonal segment.
However, in this case, $D_1$ being internally filled is equivalent to $D_1\subset A$. But $0\in D_1$ then implies $0\in A$, contradicting the assumption $0\not\in[A\cap\Lambda_d]$. 
\end{proof}

\begin{figure}[t]
\centering
\begin{subfigure}{\textwidth}
\begin{tikzpicture}[scale=0.65,x=1cm,y=1cm]
\draw [color=gray, xstep=1cm,ystep=1cm, opacity = 0.2] (-10,-4) grid (13,4);
\clip(-12.3,-4.2) rectangle (17.4,4.2);
\draw [line width=0.5pt, opacity = 0.1] (9,4)-- (13,0);
\draw [line width=0.5pt, opacity = 0.1] (13,0)-- (9,-4);
\draw [line width=0.5pt, opacity = 0.1] (9,-4)-- (5,0);
\draw [line width=0.5pt, opacity = 0.1] (6,1)-- (9,4);
\draw [line width=0.5pt, opacity = 0.1] (5,0)-- (6,1);
\begin{scriptsize}
\draw [fill=black] (-7,4) circle (3.5pt);
\draw [fill=black] (-7,3) circle (3.5pt);
\draw [fill=black] (-6,3) circle (3.5pt);
\draw [fill=black] (-5,2) circle (3.5pt);
\draw [fill=black] (-4,3) circle (3.5pt);
\draw [fill=black] (-3,3) circle (3.5pt);
\draw [fill=black] (-2,3) circle (3.5pt);
\draw [fill=black] (-3,4) circle (3.5pt);
\draw [fill=black] (-1,2) circle (3.5pt);
\draw [fill=black] (0,3) circle (3.5pt);
\draw [fill=black] (1,3) circle (3.5pt);
\draw [fill=black] (2,3) circle (3.5pt);
\draw [fill=black] (1,4) circle (3.5pt);
\draw [fill=black] (3,2) circle (3.5pt);
\draw [fill=black] (4,3) circle (3.5pt);
\draw [fill=black] (5,3) circle (3.5pt);
\draw [fill=black] (6,3) circle (3.5pt);
\draw [fill=black] (5,4) circle (3.5pt);
\draw [fill=black] (7,2) circle (3.5pt);
\draw [fill=black] (8,3) circle (3.5pt);
\draw [fill=black] (9,3) circle (3.5pt);
\draw [fill=black] (10,3) circle (3.5pt);
\draw [fill=black] (9,4) circle (3.5pt);
\draw [fill=black] (11,2) circle (3.5pt);
\draw [fill=black] (12,3) circle (3.5pt);
\draw [fill=black] (13,3) circle (3.5pt);
\draw [fill=black] (13,4) circle (3.5pt);
\draw [fill=blue] (13,0) circle (3.5pt);
\draw [fill=black] (13,-2) circle (3.5pt);
\draw [fill=black] (13,-4) circle (3.5pt);
\draw [fill=black] (11,-2) circle (3.5pt);
\draw [fill=black] (12,-1) circle (3.5pt);
\draw [fill=black] (12,-3) circle (3.5pt);
\draw [fill=black] (11,-4) circle (3.5pt);
\draw [fill=black] (10,-3) circle (3.5pt);
\draw [fill=black] (10,-1) circle (3.5pt);
\draw [fill=black] (9,-2) circle (3.5pt);
\draw [fill=black] (9,-4) circle (3.5pt);
\draw [fill=black] (8,-3) circle (3.5pt);
\draw [fill=black] (8,-1) circle (3.5pt);
\draw [fill=black] (7,-2) circle (3.5pt);
\draw [fill=black] (7,-4) circle (3.5pt);
\draw [fill=black] (6,-3) circle (3.5pt);
\draw [fill=black] (6,-1) circle (3.5pt);
\draw [fill=black] (5,-2) circle (3.5pt);
\draw [fill=black] (5,-4) circle (3.5pt);
\draw [fill=black] (4,-3) circle (3.5pt);
\draw [fill=black] (4,-1) circle (3.5pt);
\draw [fill=black] (3,-2) circle (3.5pt);
\draw [fill=black] (3,-4) circle (3.5pt);
\draw [fill=black] (2,-3) circle (3.5pt);
\draw [fill=black] (2,-1) circle (3.5pt);
\draw [fill=black] (1,-2) circle (3.5pt);
\draw [fill=black] (1,-4) circle (3.5pt);
\draw [fill=black] (0,-3) circle (3.5pt);
\draw [fill=black] (0,-1) circle (3.5pt);
\draw [fill=black] (-1,-2) circle (3.5pt);
\draw [fill=black] (-1,-4) circle (3.5pt);
\draw [fill=black] (-2,-3) circle (3.5pt);
\draw [fill=black] (-2,-1) circle (3.5pt);
\draw [fill=black] (-3,-2) circle (3.5pt);
\draw [fill=black] (-3,-4) circle (3.5pt);
\draw [fill=black] (-4,-1) circle (3.5pt);
\draw [fill=black] (-4,-3) circle (3.5pt);
\draw [fill=black] (-5,-2) circle (3.5pt);
\draw [fill=black] (-5,-4) circle (3.5pt);
\draw [fill=black] (-6,-3) circle (3.5pt);
\draw [fill=black] (-6,-1) circle (3.5pt);
\draw [fill=black] (-7,-2) circle (3.5pt);
\draw [fill=black] (-7,-4) circle (3.5pt);
\draw [fill=black] (-8,-3) circle (3.5pt);
\draw [fill=black] (-8,-1) circle (3.5pt);
\draw [fill=black] (-9,-2) circle (3.5pt);
\draw [fill=black] (-9,-4) circle (3.5pt);
\draw [fill=black] (-8,3) circle (3.5pt);
\draw [fill=black] (-9,2) circle (3.5pt);

\draw [fill=black] (11,4) circle (3.5pt);
\draw [fill=black] (7,4) circle (3.5pt);
\draw [fill=black] (3,4) circle (3.5pt);
\draw [fill=black] (-1,4) circle (3.5pt);
\draw [fill=black] (-5,4) circle (3.5pt);
\draw [fill=black] (-9,4) circle (3.5pt);

\draw (13.95,0) node {\large{$0$}};
\draw (13.95,1) node {\large{$1$}};
\draw (13.95,2) node {\large{$2$}};
55\draw (13.95,3) node {\large{$3$}};
\draw (13.95,4) node {\large{$4$}};
\draw (13.7,-1) node {\large{$-1$}};
\draw (13.7,-2) node {\large{$-2$}};
\draw (13.7,-3) node {\large{$-3$}};
\draw (13.7,-4) node {\large{$-4$}};

\draw  (9,0)-- ++(-3.5pt,0 pt) -- ++(7pt,0 pt) ++(-3.5pt,-3.5pt) -- ++(0 pt,7pt);
\draw  (5,0)-- ++(-3.5pt,0 pt) -- ++(7pt,0 pt) ++(-3.5pt,-3.5pt) -- ++(0 pt,7pt);
\draw  (1,0)-- ++(-3.5pt,0 pt) -- ++(7pt,0 pt) ++(-3.5pt,-3.5pt) -- ++(0 pt,7pt);
\draw  (-3,0)-- ++(-3.5pt,0 pt) -- ++(7pt,0 pt) ++(-3.5pt,-3.5pt) -- ++(0 pt,7pt);
\draw  (-7,0)-- ++(-3.5pt,0 pt) -- ++(7pt,0 pt) ++(-3.5pt,-3.5pt) -- ++(0 pt,7pt);
\draw [fill=black] (-10,3) circle (3.5pt);
\draw [fill=black] (-10,-1) circle (3.5pt);
\draw [fill=black] (-10,-3) circle (3.5pt);

\draw [fill=black] (-10,-4) circle (3.5pt);
\draw [fill=black] (-10,4) circle (3.5pt);

\draw [fill=red] (-5,-3) circle (3.5pt);
\draw  (-5,0) circle (3.5pt);
\draw  (-6,1) circle (3.5pt);
\draw  (-4,1) circle (3.5pt);
\draw  (-1,0) circle (3.5pt);
\draw  (-2,1) circle (3.5pt);
\draw  (0,1) circle (3.5pt);
\draw  (-3,2)-- ++(-3.5pt,-3.5pt) -- ++(7pt,7pt) ++(-7pt,0) -- ++(7pt,-7pt);
\draw  (2,1) circle (3.5pt);
\draw  (3,0) circle (3.5pt);
\draw  (4,1) circle (3.5pt);
\draw  (7,0) circle (3.5pt);
\draw  (8,1) circle (3.5pt);
\draw  (6,1) circle (3.5pt);
\draw  (1,2)-- ++(-3.5pt,-3.5pt) -- ++(7pt,7pt) ++(-7pt,0) -- ++(7pt,-7pt);
\draw  (5,2)-- ++(-3.5pt,-3.5pt) -- ++(7pt,7pt) ++(-7pt,0) -- ++(7pt,-7pt);
\end{scriptsize}
\end{tikzpicture}
\caption{In the figure, the blue dot, the red dot and the black dots correspond to initial infections. With the infection rules of $\square^4$, one may consecutively check the following claims. First, if the blue site (on line 0) were healthy, no further infection would arise. With the current initial conditions, the sites marked by $+$ (referred to as $+$-sites) are consecutively infected from right to left at the beginning. Then, the $\circ$-sites follow: note that the first of them gets infected only when the last two $+$-sites appear. In turn, infecting the $\circ$-sites leads to infecting the $\times$-sites. Finally, by using the three rules of thumb in Figure~\ref{fig:3}, one may easily show that all sites get infected eventually.}
\label{fig:2}
\end{subfigure}
\begin{subfigure}{\textwidth}
\centering
\begin{tikzpicture}[scale=0.5,x=1cm,y=1cm]
\draw [color=gray, xstep=1cm,ystep=1cm, opacity = 0.2] (-9,-3) grid (-3,1);
\draw [color=gray, xstep=1cm,ystep=1cm, opacity = 0.2] (0,-3) grid (6,3);
\draw [color=gray, xstep=1cm,ystep=1cm, opacity = 0.2] (9,-3) grid (13,1);
\clip(-9.1,-4.2) rectangle (13.1,3.1);
\begin{scriptsize}
\draw [fill=black] (-9,0) circle (3.5pt);
\draw [fill=black] (-8,0) circle (3.5pt);
\draw [fill=black] (-7,0) circle (3.5pt);
\draw [fill=black] (-6,0) circle (3.5pt);
\draw [fill=black] (-5,0) circle (3.5pt);
\draw [fill=black] (-4,0) circle (3.5pt);
\draw [fill=black] (-3,0) circle (3.5pt);
\draw [fill=black] (-8,-1) circle (3.5pt);
\draw [fill=black] (-7,-1) circle (3.5pt);
\draw [fill=black] (-6,-1) circle (3.5pt);
\draw [fill=black] (-5,-1) circle (3.5pt);
\draw [fill=black] (-4,-1) circle (3.5pt);
\draw [fill=black] (-7,-2) circle (3.5pt);
\draw [fill=black] (-6,-2) circle (3.5pt);
\draw [fill=black] (-5,-2) circle (3.5pt);
\draw [fill=black] (-6,-3) circle (3.5pt);

\draw (-6,-3.7) node {\large{$\rm{(i)}$}};
\draw (3,-3.7) node {\large{$\rm{(ii)}$}};
\draw (11,-3.7) node {\large{$\rm{(iii)}$}};

\draw  (-6,1)-- ++(-3.5pt,-3.5pt) -- ++(7pt,7pt) ++(-7pt,0) -- ++(7pt,-7pt);
\draw [fill=black] (0,0) circle (3.5pt);
\draw [fill=black] (2,0) circle (3.5pt);
\draw  (3,0)-- ++(-3.5pt,-3.5pt) -- ++(7pt,7pt) ++(-7pt,0) -- ++(7pt,-7pt);
\draw [fill=black] (3,1) circle (3.5pt);
\draw [fill=black] (3,-1) circle (3.5pt);
\draw [fill=black] (4,0) circle (3.5pt);
\draw [fill=black] (1,-1) circle (3.5pt);
\draw [fill=black] (2,-2) circle (3.5pt);
\draw [fill=black] (3,-3) circle (3.5pt);
\draw [fill=black] (4,-2) circle (3.5pt);
\draw [fill=black] (5,-1) circle (3.5pt);
\draw [fill=black] (6,0) circle (3.5pt);
\draw [fill=black] (5,1) circle (3.5pt);
\draw [fill=black] (4,2) circle (3.5pt);
\draw [fill=black] (3,3) circle (3.5pt);
\draw [fill=black] (2,2) circle (3.5pt);
\draw [fill=black] (1,1) circle (3.5pt);
\draw  (13,1)-- ++(-3.5pt,-3.5pt) -- ++(7pt,7pt) ++(-7pt,0) -- ++(7pt,-7pt);
\draw [fill=black] (12,1) circle (3.5pt);
\draw [fill=black] (11,1) circle (3.5pt);
\draw [fill=black] (10,1) circle (3.5pt);
\draw [fill=black] (9,1) circle (3.5pt);
\draw [fill=black] (10,0) circle (3.5pt);
\draw [fill=black] (11,0) circle (3.5pt);
\draw [fill=black] (12,0) circle (3.5pt);
\draw [fill=black] (13,0) circle (3.5pt);
\draw [fill=black] (13,-1) circle (3.5pt);
\draw [fill=black] (13,-2) circle (3.5pt);
\draw [fill=black] (13,-3) circle (3.5pt);
\draw [fill=black] (12,-2) circle (3.5pt);
\draw [fill=black] (12,-1) circle (3.5pt);
\draw [fill=black] (11,-1) circle (3.5pt);
\end{scriptsize}
\end{tikzpicture}
\caption{In subfigures (i) (resp.\ (ii) and (iii)) one has 16 (resp.\ 16 and 14) infected sites at $\ell^1$-distance at most 4 from the $\times$-site. Thus, to infect the $\times$-site, we need one (resp.\ one and three) additional infected sites within its fourth neighbourhood.}
\label{fig:3}
\end{subfigure}
\caption{Counterexample for the droplet algorithm for $\square^4$.}
\label{fig:2,3}
\end{figure}
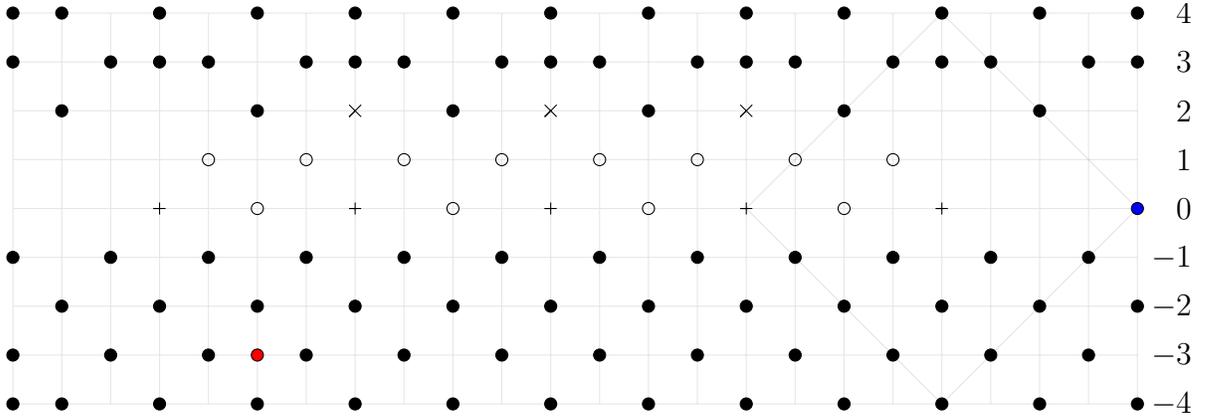
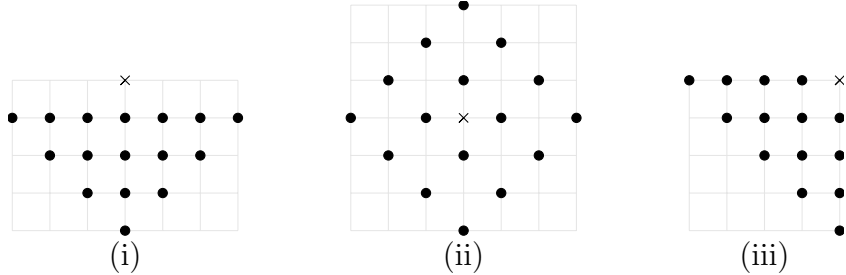

Let us make a few comments before moving on to the more difficult part of Proposition~\ref{prop:long:rectangle}. The only fragile part of the above proof is Lemma~\ref{lem:single:site}, which looks completely innocent for $\square$ but already a bit fiddly for $\triangle$. 
For the 4-neighbour model on the square-diagonal lattice (the $\boxtimes$ model), it is already false. While for this case a more complicated version of this lemma (and the main result) can be proved, the verifications become more laborious. We next show that no version of Lemma~\ref{lem:single:site} or Corollary~\ref{cor:AL} can hold even for slightly larger neighbourhoods.

One such example is the model defined by
\begin{align*}
\cK_{\square^4} &{}= \left\{x\in\bbZ^2:\|x\|_1\le 4\right\},&r_{\square^4}&{}=17,\end{align*}
which is $(\cK_s,r_s)$ for $s=4$ and $K$ given by the $\ell^1$ ball. 
Indeed, in the initial configuration in Figure~\ref{fig:2}, one can notice that, without the blue site, no further infection takes place while, without the red site, only the sites marked with $+$ get infected throughout the process. 
Hence, despite the fact that all sites get infected eventually, all internally filled droplets must contain both the blue and the red site. 
Thus, the diameter of every internally filled droplet is comparable to the size of the entire rectangle. Note that this example belongs to the family of isotropic threshold rules with difficulty 1 studied in~\cite{Duminil-Copin24}.

\section{Penetrating sparsely infected regions}
\label{sec:penetration}
In this section, we fix a $\pi/2$-rotation invariant closed convex $K\subset\bbR^2$ with $\max\{\|k\|_2:k\in K\}=1$ and the associated model $(\cK,r)=(\cK_s,r_s)$ from Section~\ref{subsec:models}. Throughout, we systematically assume $s$ to be large enough. 
Next, we argue that
\begin{equation}\label{eq:sizeK}
2r+1\le |\cK|\le 2r+2s-1\qquad \text{and}\qquad s^2+2s\le (\pi/2+o(1))s^2\le |\cK|\le (\pi+o(1))s^2\le 4s^2,
\end{equation}
where the $o(1)$ above is with respect to $s\to\infty$. In particular,~\eqref{eq:sizeK} implies that
\begin{equation}
\label{eq:r:s:bounds}
r\in \left[s^2/2, 2s^2\right].
\end{equation}
The first chain of inequalities in~\eqref{eq:sizeK} follows from Lemma~\ref{lem:Ksrs} and the fact that the horizontal coordinate axis intersects $\cK$ in at least $3$ and at most $2s+1$ sites. The second chain of inequalities in~\eqref{eq:sizeK} follows from the fact that $\cK$ is contained in a ball with radius $s$ and furthermore, by the $\pi/2$-rotation invariance and convexity of $K$, $\cK$ contains all integer points in the ball of radius $s/\sqrt 2$.

\subsection{Infecting a large square}
\label{subsec:seed}

Recall that, for every $t\ge 0$, we write $\Lambda_t = [-t,t]^2\cap \bbZ^2$. The first step in the proof of Proposition~\ref{prop:long:rectangle} consists in showing the following lemma.

\begin{lemma}\label{lem:big square}
Fix $C_2 > 0$. For every suitably large $C_1=C_1(C_2)>0$, there is $s_0 = s_0(C_1,C_2)$ such that the following holds as long as $s\ge s_0$.
Fix a set $A\subset \bbZ^2$ satisfying $\Lambda_{C_1s}\cap A=\varnothing$ and $0\in[A]$. Then, $\Lambda_{C_1s}\cap [A]$ contains an axis-parallel square of side length $C_2 s$.
\end{lemma}

The proof of Lemma~\ref{lem:big square} is based on a double-counting technique using the following concept. 

\begin{definition}[Good vertex]
Fix $A\subset\bbZ^2$ and consider $(\cK,r)$-bootstrap percolation with initial condition $A$ whose state at time $\theta$ is denoted $A_\theta$, as defined in \eqref{eq:def:BP}.
We define a directed graph with vertex set $\bbZ^2$ by constructing a directed edge $uv$ if $u$ is infected before $v$ and is then used to infect $v$. More formally, for each $v\in A_{\theta+1}\setminus A_{\theta}$ with $\theta\ge 0$, we arbitrarily select a set $\cK_v\subset A_{\theta}\cap (\cK+v)$ with $|\cK_v|=r$ and construct the directed edges $\{uv:u\in\cK_v\}$.  
We say that a vertex $u\in [A]$ is \emph{good} if its out-degree in this graph is at least $0.9 r$, and \emph{bad} otherwise.
\end{definition}

In the remainder of Section~\ref{subsec:seed}, we work with respect to a fixed set of infections $A\subset \mathbb Z^2$.
Our next lemma shows that good vertices are surrounded by a small fully infected ball.

\begin{lemma}\label{lem:good1}
Every vertex at distance at most $s/50$ from a good vertex in $[A]\setminus A$ belongs to $[A]$.
\end{lemma}
\begin{proof}
The set $\cK$ has vertices on at most $2s+1$ different rows and columns of $\mathbb Z^2$. Thus, for every two vertices $u,w\in \mathbb Z^2$ at distance 1, $|(u+\cK)\setminus (w+\cK)|\le 2s+1$.

Fix a good vertex $u\in [A]\setminus A$ and a vertex $v$ at distance at most $s/50$ from $u$. Then, there in a path in $\mathbb Z^2$ of length at most $s/25$ between $u$ and $v$. The previous observation implies that $|(u+\cK)\setminus (v+\cK)|\le (2s+1)s/25$ and, as a result,
\[|(v+\cK)\cap [A]|\ge |(u+\cK)\cap [A]| - (2s+1)s/25\ge r+1+0.9r-s^2/10\ge r+1,\]
where the last inequality uses \eqref{eq:r:s:bounds}. Thus, $v\in [A]$, as desired.
\end{proof}

In order to increase the size of the infected ball, we next show that most of the vertices in $[A]$ in a large square box $\Lambda$ avoiding initial infections are good. The proof idea is the following. Suppose that the centre of $\Lambda$ becomes infected. Each infected vertex in $\Lambda$ (which belongs to $[A]\setminus A$) has at least $r$ incoming edges and at most $|\cK|-r-1\approx r$ outgoing ones. 
By double counting this would imply that most infected sites are good, unless boundary effects interfere. If there are few infections along the boundary, we are done. If there is a concentric square along whose boundary there are few infections, we conclude similarly. 
We are left with the case when there are a lot of infections at the boundary of each concentric square, in which case they largely outnumber the maximal possible boundary effect.

\begin{lemma}\label{lem:good2}
Fix any suitably large constant $C_1 > 0$. Then, for every $s$ large enough and any set $A\subset\bbZ^2$ satisfying $\Lambda_{C_1s}\cap A=\varnothing$ and $0\in [A]$, there exists $l\in[C_1s/2,C_1s]$ such that 
\begin{equation}
\label{eq:bad:density}\frac{|\{v\in [A]\cap\Lambda_{l}:v\text{ is bad}\}|}{|[A]\cap\Lambda_l|}\le \frac{10}{C_1^{1/3}}.
\end{equation}
\end{lemma}
\begin{proof}
Define $\alpha = 10/C_1^{1/3}$, $a = C_1^{3/5}$ and the annuli $R_i = \Lambda_{(i+1)s}\setminus\Lambda_{is}$.
We will show the statement of the lemma by contradiction. We consider two cases.

\vspace{1em}

\noindent
\textbf{Case 1.} Suppose that each of the annuli $R_{C_1/2}, \ldots, R_{C_1-2}$ contains at least $a s^2$ vertices in $[A]$.
Let $I=[A]\cap\Lambda_{C_1s}$. On the one hand, the number of edges in the graph $G$ coming into $I$ is at least $r|I|$. 
On the other hand, each such edge either crosses the boundary of $\Lambda_{(C_1-1)s}$ or goes out of a vertex in $I$. 
As a result, the number of these edges is bounded from above by
\[\left|R_{C_1-1}\right|\cdot |\cK| + 0.9 r\cdot \alpha |I| + (|\cK|-r-1)\cdot (1-\alpha) |I|.\]
Moreover, by combining the upper bounds in~\eqref{eq:sizeK} and the inequality $|R_{C_1-1}|\le 4(C_1+1)(s+1)^2\le 5C_1 s^2$, the latter expression is bounded from above by
\begin{equation}\label{eq:inf_bound}
20C_1 s^4 + 0.9 r\cdot \alpha |I| + (r+2s)\cdot (1-\alpha)|I|\le 20C_1 s^4 + r |I| + (2s - \alpha r/10)|I|.
\end{equation}
By choosing $s^2\ge C_1$, we deduce that $\alpha r/10\ge 4s$. Hence, the difference between~\eqref{eq:inf_bound} and $r|I|$ is bounded from above by
\[20C_1 s^4 - \alpha r|I|/20\le 20C_1 s^4 - \alpha (C_1-1-C_1/2) as^4/40\le (20C_1 - C_1^{19/15}/12) s^4 < 0\]
for large enough $C_1$, which leads to a contradiction.

\vspace{1em}

\noindent
\textbf{Case 2.} There is $i\in [C_1/2, C_1-2]$ such that $R_i$ contains less than $a s^2$ vertices in $[A]$. Set $I=[A]\cap\Lambda_{is}$.
We begin by bounding the size of $|I|$ from below. To this end, observe that, since $0$ becomes infected eventually, each annulus among $R_0,\ldots,R_{i-1}$ must contain a vertex in $[A]$.
Consider a subset of $\lfloor i/3\rfloor$ such vertices that are pairwise at distance more than $2s$ from each other.
Then, the in-neighbourhoods of these vertices must be disjoint, so $|I|\ge \lfloor i/3\rfloor (r+1)\ge ir/4$.
By double-counting the edges of $G$ towards a vertex in $\Lambda_{is}$ as before and using the second upper bound in~\eqref{eq:sizeK}, we conclude that
\[r|I|\le a s^2\cdot |\cK| + r|I|-\alpha r|I|/20\le 4as^4 + r|I| - ir^2/(8C_1^{1/3}).\]
However, using that $i\ge C_1/3$ and \eqref{eq:r:s:bounds} leads to a contradiction in this case as well, as desired.
\end{proof}

We are ready to prove Lemma~\ref{lem:big square}.

\begin{figure}
\centering
\begin{minipage}{.45\textwidth}
  \centering
    \begin{tikzpicture}
        \draw (0,1)--(4,1)--(4,5)--(0,5)--cycle;
        \draw (1,2) node[left]{$u$}--(0.5,4) node[left]{$v$};
        \draw [shift={(0.6,3.6)},fill=black,fill opacity=0.3]  (0,0) --  plot[domain=-1.33:1.82,variable=\t]({1*0.21*cos(\t r)+0*0.21*sin(\t r)},{0*0.21*cos(\t r)+1*0.21*sin(\t r)}) -- cycle ;
        \draw [shift={(0.4,0.1)},<->] (0.5,4)--(0.7,3.2) node[midway,right]{$s/50$};
        \draw (2,3) node{$R$};
    \end{tikzpicture}
  \captionof{figure}{Illustration of the proof of Lemma~\ref{lem:big square}. Points in the shaded semi-circle are closer to $u$ than $v$ is, their distance from $v$ is less than $s/50$ and they lie in the rectangle $R$.}
  \label{fig:good2}
\end{minipage}
\qquad
\begin{minipage}{.45\textwidth}
  \centering
    \begin{tikzpicture}[scale=0.7,line cap=round,line join=round,x=1cm,y=1cm]
\clip(-6,-3.2) rectangle (12.716038947045966,3.2);
\fill[line width=0pt,color=black,fill=black,fill opacity=0.1] (-2.3333347838492724,0.9999956484521824) -- (1.4907692307692308,2.0061538461538464) -- (0,3) -- (-2,2) -- cycle;
\draw [line width=0.5pt] (2,-2)-- (3,1);
\draw [line width=0.5pt] (3,1)-- (0,3);
\draw [line width=0.5pt] (0,3)-- (-2,2);
\draw [line width=0.5pt] (-2,2)-- (-3,-1);
\draw [line width=0.5pt] (0,-3)-- (-3,-1);
\draw [line width=0.5pt] (0,-3)-- (2,-2);
\draw [line width=0.5pt] (-2.3333347838492724,0.9999956484521824)-- (1.4907692307692308,2.0061538461538464);
\draw [line width=0.5pt] (-1.4907692307692308,-2.0061538461538464)-- (2.3333347838492724,-0.9999956484521824);
\draw [line width=0.5pt] (2.3333347838492724,-0.9999956484521824)-- (1.4907692307692308,2.0061538461538464);
\draw [line width=0.5pt] (-1.4907692307692308,-2.0061538461538464)-- (-2.3333347838492724,0.9999956484521824);
\draw [line width=0.5pt] (0,0)-- (-0.4212827765400208,1.5030747473030144);
\draw [line width=0.5pt,dash pattern=on 1.9pt off 1.9pt] (-1.9120520073092515,-0.503079098850832)-- (1.9120520073092515,0.503079098850832);
\draw [line width=0.5pt,dash pattern=on 1.9pt off 1.9pt] (1.9120520073092515,0.503079098850832)-- (0.42128277654002094,1.4969252526969856);
\draw [line width=0.5pt,dash pattern=on 1.9pt off 1.9pt] (0.42128277654002094,1.4969252526969856)-- (-1.5787172234599791,0.49692525269698573);
\draw [line width=0.5pt,dash pattern=on 1.9pt off 1.9pt] (-1.5787172234599791,0.49692525269698573)-- (-1.9120520073092515,-0.503079098850832);
\begin{scriptsize}
\draw [fill=black] (0,0) circle (1.9pt);
\draw (0.13169139356578016,0.30863562750440965) node {$0$};
\draw [fill=black] (2,-2) circle (1.5pt);
\draw [fill=black] (3,1) circle (1.5pt);
\draw [fill=black] (0,3) circle (1.5pt);
\draw [fill=black] (-2,2) circle (1.5pt);
\draw [fill=black] (-3,-1) circle (1.5pt);
\draw [fill=black] (0,-3) circle (1.5pt);
\draw [fill=black] (-2.3333347838492724,0.9999956484521824) circle (1.5pt);
\draw (-2.3333347838492724-0.3,0.9999956484521824) node {$u$};
\draw [fill=black] (1.4907692307692308,2.0061538461538464) circle (1.5pt);
\draw (1.4907692307692308+0.2,2.0061538461538464+0.2) node {$v$};
\draw [fill=black] (2.3333347838492724,-0.9999956484521824) circle (1.5pt);
\draw (2.3333347838492724+0.3,-0.9999956484521824) node {$u'$};
\draw [fill=black] (-1.4907692307692308,-2.0061538461538464) circle (1.5pt);
\draw (-1.4907692307692308-0.22,-2.0061538461538464-0.22) node {$v'$};
\draw [fill=black] (-0.4212827765400208,1.5030747473030144) circle (1.5pt);
\draw (-0.2984065607936692+0.3,1.8060136908299036+0.8) node {$R$};
\draw (-0.2984065607936692-0.1,1.8060136908299036) node {$w$};
\draw [fill=black] (-1.9120520073092515,-0.503079098850832) circle (1.5pt);
\draw [fill=black] (1.9120520073092515,0.503079098850832) circle (1.5pt);
\draw [fill=black] (-1.5787172234599791,0.49692525269698573) circle (1.5pt);
\draw [fill=black] (0.42128277654002094,1.4969252526969856) circle (1.5pt);
\end{scriptsize}
\end{tikzpicture}
  \captionof{figure}{Illustration of the proof of Lemma~\ref{lem:scan}. The translate $R_1-w$ of the shaded polygon $R_1$ is dashed.}
  \label{fig:scan}
\end{minipage}
\end{figure}

\begin{proof}[Proof of Lemma~\ref{lem:big square}]
Consider $\Lambda_{l}$ as provided by Lemma~\ref{lem:good2} and tessellate it into axis-parallel squares of side length $C_2s$ (ignoring divisibility issues as usual).
Note that these squares are seen as subsets of $\mathbb Z^2$.
Suppose for a contradiction that no square is contained in $[A]$. By \eqref{eq:bad:density}, there exists a square $R\subset \bbZ^2$ of the tessellation such that $[A]\cap R\neq \varnothing$ and 
\[\frac{|\{v\in[A]\cap R:v\text{ is good}\}|}{|[A]\cap R|}\ge 1-\frac{10}{C_1^{1/3}}.\]
Fix a good vertex $u\in [A]\cap R$ (which does not belong to $A$ by the assumption that $\Lambda_{C_1s}\cap A=\varnothing$) and a vertex $v\in R\setminus[A]$ at minimal distance from $u$.
By Lemma~\ref{lem:good1}, we have $d(u,v)>s/50$ and all $w\in R$ such that $d(v,w)\le s/50$ and $d(u,w)<d(u,v)$ are bad vertices in $[A]$. 
Moreover, it is not hard to check that there is a universal constant $c>0$ such that there are at least $cs^2$ such vertices (e.g.\ those in an appropriately chosen semi-circle of radius $s/200$ centred at the point on the segment $uv$ at distance $s/100$ from $v$, see Figure~ \ref{fig:good2}).
By choosing $C_1 = C_1(c, C_2)$ so that $(2C_2s)^2\cdot 10/C_1^{1/3} < c s^2$, this leads to a contradiction with the choice of $R$. 
Thus, there exists a rectangle in the tessellation contained in $[A]$, which readily implies the statement of the lemma.
\end{proof}

\subsection{Fat convex sets}
\label{subsec:fat}

In Sections~\ref{subsec:fat} and~\ref{subsec:step} we will need to fix several positive constants 
\begin{equation}\label{eq:const}
1\ll C_6\ll C_5\ll C_4\ll C_3 \ll C_2\ll C_1
\end{equation}
where we write $a\ll b$ to denote that, having chosen $a$ large enough, $b$ is chosen suitably large with respect to some function of $a$ (and the constants on the left of $a$ in \eqref{eq:const}, if any). The size $s$ of the neighbourhood $\cK$ is also fixed, but it is assumed large enough with respect to (some functions of) the constants in~\eqref{eq:const}.

To prove Proposition~\ref{prop:long:rectangle}, we will expand the fully infected $C_2s\times C_2s$ square provided by Lemma~\ref{lem:big square} into a much larger infected convex set.
As preparation, let us fix some notation and terminology related to convex sets.

\begin{definition}
\label{def:polygon} 
We denote the Euclidean area of a convex set $P\subset\bbR^2$ by $|P|$. By abuse, we also denote by $|S|$ the Euclidean length of a (straight-line) segment $S$. 
Define $\sigma_1$ and $\sigma_2$ to be the projections on the $x$-axis and on the $y$-axis, respectively.
A \emph{chord} of $P$ is a non-empty (straight-line) segment with two endpoints on the boundary of $P$. A convex set is \emph{fat} if all vertical and horizontal chords have length at least $C_5s$. We set $x_-=\min\sigma_1(P)$, $x_+=\max\sigma_1(P)$, $y_-=\min\sigma_2(P)$ and $y_+=\max\sigma_2(P)$. We say that $P$ is \emph{horizontal} if $x_+-x_-\ge y_+-y_-$, and \emph{vertical} otherwise. For every convex set $P\subset\bbR^2$ and $x\in [x_-,x_+]$, we denote the vertical chord of $P$ lying on $\{x\}\times \bbR$ by $P_{1,x}$ and, similarly, for every $y\in [y_-, y_+]$, we denote the horizontal chord of $P$ lying on $\bbR\times \{y\}$ by $P_{2,y}$.
\end{definition}

\begin{remark}
\label{rem:fatness}
Note that a convex set is fat if and only if it has two vertical and two horizontal sides of length at least $C_5s$. We refer to these sides as \emph{top, bottom, left and right sides}.
\end{remark}

We will also use the following a simple corollary of the definition.
\begin{claim}\label{cl:square}
Let $P\subset\bbR^2$ be a fat convex set. Then, for every $x\in [x_-, x_+-C_5s]$, there is an axis-parallel square $S_x\subset P$ of side length $C_5s$ such that $\sigma_1(S_x) = [x,x+C_5s]$. Similarly, for every $y\in [y_-+C_5s, y_+]$, there is an axis-parallel square $S_y'\subset P$ of side length $C_5s$ such that $\sigma_2(S_y') = [y-C_5s,y]$.
\end{claim}
\begin{proof}[Proof of Claim~\ref{cl:square}]
We prove the first statement; the proof of the second one is similar.
Define $S_{x_-}$ (resp.\ $S_{x_+-C_5s}$) to be any axis-parallel square whose left side (resp.\ right side) is contained in $P_{x_-}$ (resp.\ in $P_{x_+}$); each of these squares is contained in $P$ by fatness of $P$.
Fix $x\in [x_-, x_+-C_5s]$ and write it in the form $\lambda x_- + (1-\lambda)(x_+-C_5s)$ for some $\lambda\in [0,1]$.
Then, the square $S_x := \lambda S_{x_-}+(1-\lambda) S_{x_+-C_5s}$ is contained in $P$ by convexity and satisfies the claim.
\end{proof}

The main technical result we aim to prove in Sections~\ref{subsec:fat} and~\ref{subsec:step} is the following.

\begin{restatable}{prop}{propstep}\label{prop:step}
Fix $d\ge 2s$, a set $A\subset\bbZ^2\setminus\Lambda_d$ and suppose that a fat convex set $P\subset\bbR^2$ satisfies $P\cap\bbZ^2\subset [A]\cap \Lambda_{d-2s}$ and $|P|\ge(C_2 s)^2$. Then, there exists a fat convex set $P'$ such that $P'\cap\bbZ^2\subset [A]\cap \Lambda_d$ and $|P'|\ge |P|+
s^2$.
\end{restatable}

In this section, we prepare for the proof of Proposition~\ref{prop:step} by studying a well-chosen strip of our convex set.
We first require a property of convex sets which will be applied to the neighbourhood $\cK$.

\begin{lemma}\label{lem:scan}
Fix a convex set $R$ satisfying $R=-R$. Suppose that a chord $uv$ divides $R$ into two convex sets $R_1$ and $R_2$, where $R_2$ contains the origin. 
Then, there is a point $w$ on the segment $uv$ such that $w+R$ contains $R_1$.
\end{lemma}
\begin{proof}
If $uv$ contains the origin, our claim is trivial. Otherwise, let $u' = -u$ and $v' = -v$ so that $v'u'vu$ is a parallelogram contained in $R_2$, see Figure~\ref{fig:scan}.
On the one hand, by symmetry of $R$, each of $u'$ and $v'$ lies on $\partial R$. 
On the other hand, none of the parallel lines containing the chords $uv'$ and $vu'$ meets the interior of $R_1$.
As a result, $R_1\subset w+R$ for $w = (u+v)/2$, as desired.
\end{proof}

We approach Proposition~\ref{prop:step} by finding a small fully infected ball near the (endpoints of) a suitably chosen vertical chord of a horizontal convex set. Namely, we require the boundary around the chord to be rather flat and not too steep.

\begin{definition}[Admissible chord]
\label{def:admissible}
Fix a horizontal convex set $P$. For $x\in (x_-,x_+)$, let $\alpha^+_x$ (resp.\ $\alpha^-_x$) denote the angle between the $x$ axis and the tangent to the upper (resp.\ lower) boundary of $P$ at $P_{1,x}$. If the tangent is not uniquely defined, we make the convention that any condition stated holds for all possible tangents.

Given $x\in(x_-+C_4s,x_+-C_4s)$, the vertical chord $P_{1,x}$ is \emph{admissible} if
\begin{align*}
\alpha^+_x&{}\in [-4\pi/9,4\pi/9],& \alpha^+_{x-C_4s}-\alpha^+_{x+C_4s}&{}\le 1/C_3,& \alpha^-_x&{}\in[-4\pi/9,4\pi/9], & \alpha^-_{x+C_4s}-\alpha^-_{x-C_4s}&{}\le 1/C_3.
\end{align*}
\end{definition}

\begin{lemma}[Admissible chords exist]\label{lem:adm_strips}
Fix a fat horizontal convex set $P$ of area at least $(C_2 s)^2$. 
Then, there exists an admissible chord.
\end{lemma}
\begin{proof}
Since $P$ is fat and horizontal, we have $x_+-x_-\ge \diam(P)/\sqrt{2}\ge C_2s/4$. The interval of $x$ such that $\alpha^+_x<-4\pi/9$ has length at most $(y_+-y_-)/\tan(4\pi/9)\le (x_+-x_-)/\tan(4\pi/9)$, and similar bounds hold for $\alpha^+_x>4\pi/9$ and for $\alpha^-_x$. The set of $x$ such that $\alpha^+_{x-C_4s}-\alpha^+_{x+C_4s}>1/C_3$ has measure at most $4\pi C_3C_4s$, and similar bound holds for $\alpha^-$.
As a result, the set of $x$ such that $P_{1,x}$ is admissible has measure at least
\[(x_+-x_-)(1-4/\tan(4\pi/9))-4\cdot4\pi C_3C_4s>0,\]
where we used that $x_+-x_-\ge C_2s/4$, $\tan(4\pi/9)>4$ and $C_2\gg C_3,C_4$.
\end{proof}

Recall that, for $x\in \mathbb R^2$ and $\rho\ge 0$, we denote by $B(x,\rho)$ the Euclidean ball with centre $x$ and radius~$\rho$.

\begin{lemma}\label{lem:boundaries}
Fix a fat horizontal convex set $P$, an admissible chord $P_{1,x}$ and a point $u$ at distance at most $s$ from $P_{1,x}$.
Then, $B(u, 2s)$ intersects at most one of the upper and the lower boundaries of $P$.
\end{lemma}
\begin{proof}
Suppose for contradiction that point $y$ on the upper boundary of $P$ and point $z$ on the lower boundary of $P$ are both at distance at most $2s$ from $u$. 
Also, fix arbitrary tangents to $P$ at points $y$ and $z$, and denote by $y'$ and $z'$ their intersection points with the vertical line containing $P_{1,\sigma_1(u)}$. Then, by admissibility of $P_{1,x}$, we have $\alpha^+_{\sigma_1(y)},\alpha^-_{\sigma_1(z)}\in[-4\pi/9-1/C_3,4\pi/9+1/C_3]$ and
\[|y'z'|\le |y'y|+|yu|+|uz|+|zz'|\le \frac{2s}{\cos(4\pi/9+1/C_3)}+2s+2s+\frac{2s}{\cos(4\pi/9+1/C_3)}.\]
By choosing $C_3$ and $C_5$ appropriately large, the latter bound is at most $C_5s/2$, contradicting the fatness of $P$, as desired.
\end{proof}

The following lemma shows that $[A]\setminus P$ contains a small bump touching $P$ next to every admissible chord. For two sets $A,B$ intersecting the boundary of $P$, the \emph{distance between $A$ and $B$ along the boundary of $P$} is the length of the shortest piecewise-linear curve contained in the boundary of $P$ with one endpoint in $A$ and the other endpoint in $B$.

\begin{lemma}\label{lem:blob}
Fix $d\ge 2s$ and a set $A\subset \bbZ^2$ such that $A\cap\Lambda_d=\varnothing$. Let $P\subset\Lambda_{d-2s}$ be a fat horizontal convex set with $P\cap \bbZ^2\subset[A]$, and let $P_{1,x}$ be an admissible chord. Then, there exists a point $z$ on the boundary of the convex set $P$ such that the distance between $z$ and $P_{1,x}$ along the boundary of $P$ is at most $3s$ and $B(z,s/C_6)\cap\bbZ^2\subset[A]$.
\end{lemma}
\begin{proof}
Let $S=\bigcup_{x'\in[x-s,x+s]}P_{1,x'}$ and $\bar S=[x-s,x+s]\times\bbR$. 
Consider the first vertex $u\in S$ which becomes infected by the bootstrap percolation process with initial condition $A$ (breaking ties arbitrarily).
By Lemma~\ref{lem:boundaries}, only one of the upper and the lower boundaries of $P$ can be at distance at most $2s$ from $u$. However, since all vertices in $S\setminus \{u\}$ may only get infected after $u$, $u$ must have a neighbour outside $P$.
Thus, without loss of generality, we may assume that $u+\cK$ intersects only the top and/or left boundaries of $S$. We set $\delta=C_6^{-1/5}$ and consider three cases illustrated in Figure~\ref{fig:three pics 2}.

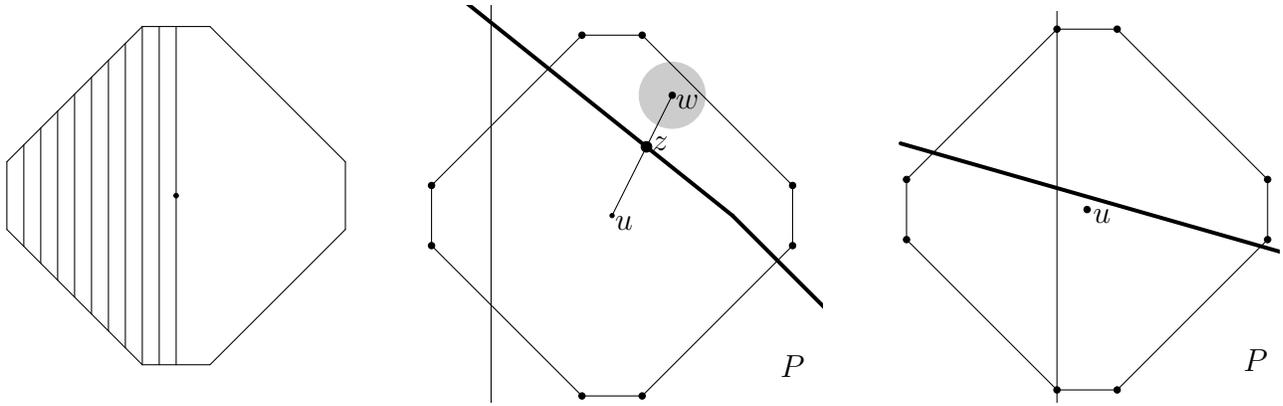
\begin{figure}
     \centering
     \begin{subfigure}[b]{0.22\textwidth}
         \centering
\begin{tikzpicture}[scale=0.3375,line cap=round,line join=round,x=1cm,y=1cm]
\clip(-16,-6.1) rectangle (-4.9,6.1);
\draw  (-11,5)-- (-9,5);
\draw  (-9,5)-- (-5,1);
\draw  (-5,1)-- (-5,-1);
\draw  (-5,-1)-- (-9,-5);
\draw  (-9,-5)-- (-11,-5);
\draw  (-11,-5)-- (-15,-1);
\draw  (-15,-1)-- (-15,1);
\draw  (-15,1)-- (-11,5);
\draw  (-10,5)-- (-10,-5);
\draw  (-11,5)-- (-11,-5);
\draw  (-12,4)-- (-12,-4);
\draw  (-13,3)-- (-13,-3);
\draw  (-14,2)-- (-14,-2);
\draw  (-10.5,5)-- (-10.5,-5);
\draw  (-11.5,4.5)-- (-11.5,-4.5);
\draw  (-12.5,3.5)-- (-12.5,-3.5);
\draw  (-13.5,2.5)-- (-13.5,-2.5);
\draw  (-14.5,1.5)-- (-14.5,-1.5);
\begin{scriptsize}
\draw [fill=black] (-10,0) circle (1.9pt);
\end{scriptsize}
\end{tikzpicture}
\caption{Case 1.}
\label{fig:segments}
     \end{subfigure}
     \begin{subfigure}[b]{0.26\textwidth}
         \centering
\begin{tikzpicture}[scale=0.4875,line cap=round,line join=round,x=1cm,y=1cm]
\clip(-7,-3.6) rectangle (2,4.4);

\draw[<->] (-5,-3.5) -- (-1,-3.5);
\draw (-4,-3.5) node[above]{$s$};
\draw[<->] (-5,0.5)--(-2.5,0.5) node[midway,below]{$\ge\delta s$};
\draw (-1,2.4)--(-1,-4);
\draw (-1,0.5) node[right]{$P_{1,x}$};

\fill[fill=black,fill opacity=0.1] (-4.95,2.0467146278752925) -- (-0.4926213732128431,2.4926213732128432) -- (-2,4) -- (-3,4) -- cycle;
\draw (-6,0)-- (-6,1);
\draw (-6,1)-- (-3,4);
\draw (-3,4)-- (-2,4);
\draw (-2,4)-- (1,1);
\draw (1,1)-- (1,0);
\draw (1,0)-- (-2,-3);
\draw (-2,-3)-- (-3,-3);
\draw (-3,-3)-- (-6,0);
\draw (-5,-3.813301972010373) -- (-5,4.47511585686439);
\draw [line width=1.5pt] (-7,1)-- (-3.1615159284348486,2.015938903462461);
\draw [line width=1.5pt] (-3.1615159284348486,2.015938903462461)-- (-1,2.4);
\draw [line width=1.5pt] (-1,2.4)-- (2,2.7);
\draw (-7,1.8)-- (1.978484071565152,2.75);
\draw  (-0.6926213732128431,2.6926213732128432)-- (-2,4);
\draw (-2,4)-- (-3,4);
\draw (-3,4)-- (-4.753285372124707,2.2467146278752925);
\draw [fill=black] (-2.67,2.270390802608277) circle (1.5pt);
\draw (-2.57,2.7) node {$z'$};
\draw [fill=black] (-2.648866188921178,2.11268408776757) circle (1.5pt);
\draw (-2.5412542513251077,1.85) node {$z$};
\draw [fill=black] (-6,0) circle (1.5pt);
\draw [fill=black] (-6,1) circle (1.5pt);
\draw (-6.7,0.6374123783176) node {$P$};
\draw [fill=black] (-3,4) circle (1.5pt);
\draw [fill=black] (-2,4) circle (1.5pt);
\draw [fill=black] (1,1) circle (1.5pt);
\draw [fill=black] (1,0) circle (1.5pt);
\draw [fill=black] (-2,-3) circle (1.5pt);
\draw [fill=black] (-3,-3) circle (1.5pt);
\draw[color=black] (-4.4,4) node {$\partial \bar S$};
\draw [fill=black] (-2.5,0.5) circle (1.5pt);
\draw (-2.14,0.5) node {$u$};
\draw (-6.5,2.5) node {$J$};
\draw (-2.515712131667558,3.61) node {$\mathcal K_1$};
\end{tikzpicture}
\caption{Case 1.}
\label{fig:revision}
     \end{subfigure}
     \hfill
     \begin{subfigure}[b]{0.24\textwidth}
         \centering
\begin{tikzpicture}[scale=0.6,line cap=round,line join=round,x=1cm,y=1cm]
\clip(-7.1,-3.1) rectangle (-0.5,3.5);
\draw [line width=1.5pt] (-7,4)-- (-18,9);
\draw [line width=1.5pt] (-7,4)-- (-2,0);
\draw [line width=1.5pt] (-2,0)-- (9,-11);
\draw  (-4.5,3)-- (-3.5,3);
\draw  (-3.5,3)-- (-1,0.5);
\draw  (-1,0.5)-- (-1,-0.5);
\draw  (-1,-0.5)-- (-3.5,-3);
\draw  (-3.5,-3)-- (-4.5,-3);
\draw  (-4.5,-3)-- (-7,-0.5);
\draw [line width=0.19pt,fill=black,opacity=0.2] (-3,2) circle (0.55cm);
\draw  (-4,0)-- (-3,2);
\draw  (-4.5,3)-- (-7,0.5);
\draw  (-7,0.5)-- (-7,-0.5);
\draw  (-3,0.8)-- (-3,-10);
\draw  (6,-3.323572931005381) -- (6,4.74692335565218);
\draw  (-6.010701625881464,-3.323572931005381) -- (-6.010701625881464,4.74692335565218);
\draw [fill=black] (-4,0) circle (1pt);
\draw (-3.8,-0.12036613445058617) node {$u$};
\draw [fill=black] (-4.5,3) circle (1.5pt);
\draw [fill=black] (-3.5,3) circle (1.5pt);
\draw [fill=black] (-1,0.5) circle (1.5pt);
\draw [fill=black] (-1,-0.5) circle (1.5pt);
\draw [fill=black] (-3.5,-3) circle (1.5pt);
\draw [fill=black] (-4.5,-3) circle (1.5pt);
\draw [fill=black] (-7,-0.5) circle (1.5pt);
\draw [fill=black] (-3,2) circle (1.5pt);
\draw (-2.75,1.9) node {$w$};
\draw [fill=black] (-3.4285714285714284,1.1428571428571428) circle (1.5pt);
\draw (-3.2,1.2) node {$z$};
\draw [fill=black] (-7,0.5) circle (1.5pt);
\draw (-1,-2.5) node {$P$};
\draw (-6,0) node[right]{$\partial\bar S$};
\draw (-3.2,-1) node[right]{$P_{1,x}$};
\end{tikzpicture}
\caption{Case 2.}
\label{Ivo8}
     \end{subfigure}
     \hfill
     \begin{subfigure}[b]{0.24\textwidth}
         \centering
\begin{tikzpicture}[scale=0.6,line cap=round,line join=round,x=1cm,y=1cm]
\clip(-7.3,-2.7) rectangle (-13.7,3.8);
\draw  (-11,-3.229663518878698) -- (-11,4.758590928903883);
\draw  (7,-3.229663518878698) -- (7,4.758590928903883);
\draw [line width=1.5pt] (-13.6,1.6)-- (-1,-2);
\draw  (-1,-2)-- (18.724931506849313,-9.396849315068492);
\draw  (-10,3.5)-- (-7.5,1);
\draw  (-7.5,1)-- (-7.5,0);
\draw  (-7.5,0)-- (-10,-2.5);
\draw  (-10,-2.5)-- (-11,-2.5);
\draw  (-11,-2.5)-- (-13.5,0);
\draw  (-13.5,0)-- (-13.5,1);
\draw  (-13.5,1)-- (-11,3.5);
\draw  (-11,3.5)-- (-10,3.5);

\draw  (-8,0)-- (-8,-10);

\draw [fill=black] (-10.5,0.5) circle (1.5pt);
\draw (-10.25,0.4) node {$u$};
\draw [fill=black] (-10,3.5) circle (1.5pt);
\draw [fill=black] (-7.5,1) circle (1.5pt);
\draw [fill=black] (-7.5,0) circle (1.5pt);
\draw [fill=black] (-10,-2.5) circle (1.5pt);
\draw [fill=black] (-11,-2.5) circle (1.5pt);
\draw [fill=black] (-13.5,0) circle (1.5pt);
\draw [fill=black] (-13.5,1) circle (1.5pt);
\draw [fill=black] (-11,3.5) circle (1.5pt);
\draw (-7.55,-2) node {$P$};
\draw (-11,0) node[left]{$\partial\bar S$};
\draw (-7.8,-0.3) node[left]{$P_{1,x}$};
\end{tikzpicture}
\caption{Case 3.}
\label{Ivo9}
     \end{subfigure}
        \caption{Figures accompanying the proof of Lemma~\ref{lem:blob}. The thick contour is the boundary of~$P$.}
        \label{fig:three pics 2}
\end{figure}

\vspace{1em}
\noindent
\textbf{Case 1.} Suppose that the vertex $u$ is at distance at least $\delta s$ from the left boundary of $\bar S$. 
Note that, by symmetry and convexity, the lengths of the chords $\cK_{x'}$ of $\cK$ are non-decreasing with respect to $x'$ for $x'\le 0$ (see Figure~\ref{fig:segments}).
As a result, the area of $(u+\cK)\cap \bar S$ is at least a $(1+\delta)/2$-proportion of the area of $u+\cK$ and, therefore, $\bar S\setminus S$ must contain at least $\delta r/3$ vertices in $u+\cK$ infected prior to $u$.

Fix a line $J$ tangent to $P$ at the point where its upper boundary meets $P_{1,x}$.
Then, by choosing $C_3$ appropriately large (recall Definition~\ref{def:admissible}), we can ensure that $J$ divides $u+\cK$ into two parts, $\cK_1$ and $\cK_2$, satisfying $\cK_2 \supset(u+\cK)\cap P$ and $|\cK_2\setminus P| \le \delta^5 s^2$.
We conclude that, on the one hand, $\cK_1$ must contain at least $\delta r/3 - \delta^5 s^2 \ge
\delta r/4$ vertices in $[A]$. 
On the other hand, by Lemma~\ref{lem:scan}, one can find a point $z'\in J$ such that $z'+\cK$ contains $\cK_1$. 
Again, by admissibility of $P_{1,x}$ and for appropriately large $C_3$, this means that at least $r-1-\delta^5 s^2+\delta r/4\ge r + \delta r/5$ vertices in $z'+\cK$ are in $[A]$.
Moreover, there is a point $z$ on the boundary of $P$ at distance at most $\delta^5 s$ from $z'$ (see Figure~\ref{fig:revision}). Note that, for every $y\in B(z,s/C_6)\cap\bbZ^2$, $B(z',1)$ and $y$ are connected by a path in $\mathbb Z^2$ of length at most $1+2s(\delta^5+1/C_6)\le s\delta/60$. Similarly to the argument from the proof of Lemma~\ref{lem:good1}, by combining~\eqref{eq:sizeK} and \eqref{eq:r:s:bounds}, we get that
\[|(z'+\cK)\setminus (y+\cK)|\le (2s+1)\cdot (s\delta/60) \le 
\delta r/10,\]
which implies that at least $r + \delta r/5 - \delta r/10 \ge r + 1$ vertices in $y+\cK$ belong to $[A]$, as desired.

\vspace{1em}
\noindent
\textbf{Case 2.} Suppose that the vertex $u$ is at distance at least $\delta s$ from the upper boundary of $P$. 
Fix a vertex $w$ in $(u+\cK)\cap (\bar S\setminus P)$ that was infected before $u$. Note that $w$ is located above the upper boundary of $P$.
Define the region $W = (w+\cK)\cap S$ and note that every vertex in $W$ gets infected after $w$; our next task will be to bound from below the number of vertices in this region.
Denote $u_1 = u$ and let $u_2$, $u_3$ and $u_4$ be obtained by rotation of $u$ around $w$ at angle $\pi/2$, $\pi$ and $3\pi/2$, respectively.
Since $\cK$ is convex and $\pi/2$-rotation invariant, all vertices in the square $R = u_1u_2u_3u_4$ belong to $w+\cK$.
Moreover, by our assumption on $u$, the square $R'\subset R$ with corner $u_1$ and diameter $\delta s$ is contained in $P$.
Using that $R'$ contains approximately as many points as its area, we obtain that this number is at least $(1/2+o(1))(\delta s)^2\ge \delta^2 s^2/3$ (for $s$ large enough).
Thus, $w+\cK$ contains at least $r+\delta^2s^2/3$ vertices in $[A]$ and, therefore, all points at distance at least $\delta^2s^2/(3\cdot 2\cdot (2s+1))\ge \delta^2 s/20$ from $w$ get infected eventually.

Now, consider the point $z$ obtained by intersecting the segment $uw$ with the upper boundary of $P$; that is, $z = (1-\lambda) u+\lambda w$ for some $\lambda\in [\delta,1]$, see Figure~\ref{Ivo8}.
Then, on the one hand, $z+\cK$ must contain $w+\delta \cK\subset w+\lambda \cK$, which itself contains the ball $B(w,\delta^2 s/20)$.
On the other hand, since $P_{1,x}$ is an admissible strip, $z+\cK$ must contain at least $r-\delta^5 s^2$ infected vertices in $P$ and, since at least half of the ball $B(w,\delta^2 s/20)$ is outside $P$, at least $(\pi/2+o(1)) (\delta^2 s/20)^2\ge \delta^4 s^2/400 > 2\delta^5 s^2$ infected vertices outside $P$.
Thus, all vertices at distance at most $\delta^4 s^2/(800\cdot 2\cdot (2s+1))\ge \delta^4 s/4000 >s/C_6$ from $z$ belong to $[A]$.

\vspace{1em}
\noindent
\textbf{Case 3.} Suppose that the vertex $u$ is at distance at most $\delta s$ both from the left boundary of $\bar S$ and from the upper boundary of $P$.
Then, by Definition~\ref{def:admissible}, the angle between the upper and the left boundaries of $S$ is at least $\pi/20$, so $u$ must be at distance at most $\delta s/\sin(\pi/40)\le 20\delta s$ from their intersection point.
At the same time, the bound on the slope of the upper boundary of $P$ implies that $B(u,s/\sqrt{2})\subset u+\cK$ contains at least $s^2/100$ vertices in $S$.
As a result, there is a point $z$ on the upper boundary of $P$ for which $z+\cK$ contains at least $r+s^2/100-2\delta s\cdot (2s+1) \ge r+s^2/200$ vertices in $[A]$. In turn, this implies that every point at distance at most $s^2/(200\cdot 2\cdot (2s+1))\ge s/1000$ from $z$ belongs to $[A]$, as desired.
\end{proof}

\subsection{Expansion step}
\label{subsec:step}
In this section, we will show how to use the small infected ball ensured by Lemma~\ref{lem:blob} to construct the convex set $P'$ required in Proposition~\ref{prop:step}. To begin with, we show that, for some large constant $C_5$, there is a rectangle with dimensions $s/C_5\times 3s$ in $[A]$ in the vicinity of the ball ensured by Lemma~\ref{lem:blob}. We note that the constant $C_5$ here will play an independent role from its occurrence in Definition~\ref{def:polygon}, but is reused for convenience.

\begin{lemma}\label{lem:rectangle}
Fix a fat convex set $P$ and an admissible chord $P_{1,x}$ as in Lemma~\ref{lem:blob}. There is a rectangle $\Pi$ with the following properties:
\begin{itemize}
    \item it has dimensions $s/C_5\times 3s$,
    \item $\Pi\cap \mathbb Z^2\subset [A]$,
    \item one of its longer sides, called $\Pi_-$, is entirely contained in $P$ while its other long side, called $\Pi_+$, is at distance at least $s/(2C_5)$ from $P$, and
    \item $\Pi_-$ is at distance at most $3s$ from $P_{1,x}$.
\end{itemize}
\end{lemma}
\begin{proof}
Fix a point $z$ as in Lemma~\ref{lem:blob}.
Let $ab$ be a chord in $P$ where both $a$ and $b$ belong to the upper boundary of $P$, $\sigma_1(a) = x-C_4s$ and $\sigma_1(b) = x+C_4s$.
Let $z'\in ab$ be such that $\sigma_1(z)=\sigma_1(z')$.
Let $a_1,a_2,a_3,a_4,a_5,z',b_5,b_4,b_3,b_2,b_1$ lie at intervals of length $s/2$ in that order on $ab$ (see Figure~\ref{fig:3iterations}). Note that, $a_1,b_1$ remain at distance at least $C_4s/2$ from $a$ and $b$. Finally, let $P'$ and $P''$ be the convex sets into which $P$ is divided by the chord $ab$ with $P''$ lying below $ab$ and $z\in P'$ (note that, if $ab\subset\partial P$, $P'$ may degenerate into the segment $ab$).

Combining Definitions~\ref{def:polygon} and \ref{def:admissible},
we obtain that $P''$ must contain a rectangle $\Pi'$ with longer side $a_1b_1$ and width $s$ (see Figure~\ref{fig:3iterations}). We claim that, for suitably large $C_5 = C_5(C_6)$, after one step of the bootstrap percolation process with initial condition $\Pi'\cup B(z,s/C_6)$, the rectangle $\Pi''\not\subset\Pi'$ with longer side $a_5b_5$ and width $s/C_5^{1/3}$ becomes infected. Indeed, any $v\in\Pi''$ satisfies $v+\cK\supset B(z,s/C_6)$, while $|(v+\cK)\cap\Pi'\cap\bbZ^2|\ge r-1-2(2s-1)s/C_5^{1/3}$, which is sufficient since $C_5\gg C_6$. 
By proceeding analogously, we obtain that, at the first step of the process with initial condition $\Pi'\cup\Pi''$, the rectangle $\Pi'''\not\subset\Pi'$ with longer side $a_4b_4$ and width $s/C_5^{2/3}$ becomes infected. 
Finally, at the first step of the process with initial condition $\Pi'\cup\Pi'''$, the rectangle $\Pi\not\subset \Pi'$ with longer side $a_3b_3$ and width $s/C_5$ becomes infected.

It remains to check that $\Pi$ has the desired properties.
The first and the second properties are immediate to verify, 
while the fourth property follows from the fact that, by Lemma~\ref{lem:blob}, the distance between $z$ and $P_{1,x}$ along the boundary of $P$ is at most $3s$.
For the third property, note that the admissibility of $P_{1,x}$ guarantees that the distance between any point in $P'$ to $ab$ is at most $\sin(1/C_3)\cdot 2C_4s <s/(2C_5)$, as desired.
\end{proof}
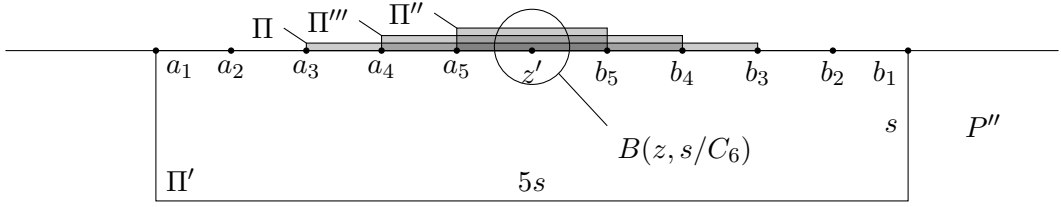
\begin{figure}
    \centering
\begin{tikzpicture}
\draw [fill=black] (0,0) circle (1pt) node[below right]{$a_{1}$};
\draw [fill=black] (1,0) circle (1pt) node[below]{$a_{2}$};
\draw [fill=black] (2,0) circle (1pt) node[below]{$a_{3}$};
\draw [fill=black] (3,0) circle (1pt) node[below]{$a_{4}$};
\draw [fill=black] (4,0) circle (1pt) node[below]{$a_{5}$};
\draw [fill=black] (5,0) circle (1pt) node[below]{$z'$};
\draw [fill=black] (6,0) circle (1pt) node[below]{$b_{5}$};
\draw [fill=black] (7,0) circle (1pt) node[below]{$b_{4}$};
\draw [fill=black] (8,0) circle (1pt) node[below]{$b_{3}$};
\draw [fill=black] (9,0) circle (1pt) node[below]{$b_{2}$};
\draw [fill=black] (10,0) circle (1pt) node[below left]{$b_{1}$};
\draw (0,0) -- (0,-2) -- (10,-2) node[midway,above]{$5s$} -- (10,0) node[midway, left]{$s$} -- cycle;
\draw[fill=black,fill opacity=0.2] (4,0) -- (6,0) -- (6,0.3) -- (4,0.3) -- cycle;
\draw[fill=black,fill opacity=0.2] (3,0) -- (7,0) -- (7,0.2) -- (3,0.2) -- cycle;
\draw[fill=black,fill opacity=0.2] (2,0) -- (8,0) -- (8,0.1) -- (2,0.1) -- cycle;
\draw (0,-2) node[above right]{$\Pi'$};
\draw (2,0.1)--(1.7,0.3) node[left]{$\Pi$};
\draw (3,0.2)--(2.7,0.4) node[left]{$\Pi'''$};
\draw (4,0.3)--(3.7,0.5) node[left]{$\Pi''$};
\draw (5,0.05) circle (0.5);
\draw (5.35355,-0.30355)--(6,-1) node[below right]{$B(z,s/C_6)$};
\draw (-2,0)--(12,0);
\draw (11,-1) node{$P''$};
\draw (3.5,-2.1)--(3.5,0) node[midway,left]{$P_{1,x}$};
\end{tikzpicture}
    \caption{Illustration of the proof of Lemma~\ref{lem:rectangle}. The curvature of $\partial P$ being extremely small, the convex set $P'$ is much thinner than the rectangle $\Pi$.
    As a result, neither the upper boundary of $P'$ nor the point $z$ lying on it very close to $z'$ are depicted. Note that the rectangles depicted need not be axis-parallel.}
    \label{fig:3iterations}
\end{figure}

Before completing the proof of Proposition~\ref{prop:step}, we state a last simple lemma used several times in its proof.

\begin{lemma}
\label{lem:two:rectangles}
Fix any $\eps > 0$ and $\alpha > 2$ allowed to depend on $s$. Let $a,b',a',b\in\bbR^2$ be four collinear points found on their common line in this order such that $|ab'| = |a'b| = |b'a'|/(\alpha-2) = s$.
Consider rectangles $\Gamma$ with dimensions $\alpha s\times s$ with side $ab$, $\Gamma_a$ with dimensions $s\times \eps s$ with side $ab'$ and $\Gamma_b$ with dimensions $s\times \eps s$ with side $a'b$ with disjoint interiors.
Then, $[\Gamma\cup \Gamma_a\cup \Gamma_b]$ contains the smallest rectangle $\bar \Gamma$ containing $\Gamma\cup\Gamma_a\cup\Gamma_b$ (see Figure~\ref{fig:two:rectangles}).
\end{lemma}

\begin{figure}
    \centering
\begin{tikzpicture}[x=2cm,y=2cm]
\draw [fill=black] (0,0) circle (1pt) node[below right]{$a\phantom{'}$};
\draw [fill=black] (-0.15,0.25) node[below]{$\varepsilon s$};
\draw [fill=black] (0.5,0.53) node[below]{$s$};
\draw [fill=black] (3.5,0.53) node[below]{$s$};
\draw [fill=black] (1,0) circle (1pt) node[below]{$b'$};
\draw [fill=black] (3,0) circle (1pt) node[below]{$a'$};
\draw [fill=black] (4,0) circle (1pt) node[below left]{$b\phantom{'}$};
\draw[fill=black, fill opacity=0.2](0,0) -- (0,-1) -- (4,-1) node[midway,above,opacity=1]{$\alpha s$} -- (4,0) node[midway, left,opacity=1]{$s$} -- cycle;
\draw[fill=black,fill opacity=0.2] (0,0) rectangle (1,0.3);
\draw[fill=black,fill opacity=0.2] (3,0) rectangle (4,0.3);
\draw (0.5,0.15) node{$\Gamma_a$};
\draw (3.5,0.15) node{$\Gamma_b$};
\draw (2,-0.5) node{$\Gamma$};
\draw (2,0.15) node{$\bar\Gamma$};
\draw [very thick] (0,-1) rectangle (4,0.3);
\end{tikzpicture}
    \caption{Illustration of Lemma~\ref{lem:two:rectangles}. The rectangles $\Gamma_a,\Gamma_b,\Gamma$ are shaded, while the rectangle $\bar \Gamma$ is thickened. Note that the rectangles depicted need not be axis-parallel.}
    \label{fig:two:rectangles}
\end{figure}
\begin{proof}
Let $\Pi_0 = \varnothing$ and $\Pi_1\subset \Pi_2\subset \ldots\subset \Pi_k$ be the list of all rectangles with parallel sides $a'b'$ and $\ell_i\subset \bar\Gamma\setminus \Gamma$ such that $\ell_i\cap \mathbb Z^2\neq \varnothing$. Also, denote by $u\in S^1$ the normal vector to $\Gamma$ at $a'$.

We show by induction that, for every $i\in \{0,\ldots,k\}$, $\Pi_i\subset [\Gamma\cup \Gamma_a\cup \Gamma_b]$.
The base case being trivial, suppose that $\Pi_{i-1}\subset [\Gamma\cup \Gamma_a\cup \Gamma_b]$ for some $i\in \{1,\ldots,k\}$. 
Let $x_1,\ldots,x_t$ be the set of integer points met on $\ell_i$ in this order, with $x_1$ closest to $\Gamma_a$.
If $u\notin \cS$, then, for every $j\in \{1,\ldots,t\}$,
\[|(x_j+\cK)\cap (\Pi_{i-1}\cup \Gamma\cup \Gamma_a\cup \Gamma_b)| \ge r,\]
and the conclusion is immediate.
Otherwise, the point $x_0 := 2x_1-x_2$ belongs to $\Gamma_a\subset [A]$ and $|x_0x_1|=|x_1x_2|=\ldots=|x_{t-1}x_t|\le \sqrt{2}\le \min\{|y|: y\in \partial (sK)\}$ for $s\ge 2$. Moreover, for every $j\in \{1,\ldots,t\}$,
\[|(x_j+\cK)\cap (x_j+\bbH_u)\cap (\Pi_{i-1}\cup \Gamma\cup \Gamma_a\cup \Gamma_b)| = r-1.\] 
Together with the induction hypothesis,
this implies that $x_1,\ldots,x_t$ may be infected in this order, which finishes the induction and the proof.
\end{proof}

The proof of Proposition~\ref{prop:step} is a bit long, so we present a high-level overview of it first.
Given a fat horizontal convex set $P$, we first find an admissible vertical chord $P_{1,x}$ at distance $\Theta(\diam(P))$ from each of the left and the right end of $P$.
Then, Lemma~\ref{lem:rectangle} ensures the existence of a rectangle $\Pi\subset [A]$ near $P_{1,x}$ sticking out of $P$.
We draw two secant lines, $\ell_-$ and $\ell_+$, from the midpoint $z$ of the side $\Pi_+$ of $\Pi$ (recall Lemma~\ref{lem:rectangle}) to $P$ so that $\ell_-$ and $\ell_+$ intersect $P$ in chords of length $C_6 s$.
A new convex~set~$\bar P$ is obtained from $P$ by adding the region between $\ell_-, \ell_+$ and $P$ (the vertices there turn out to be in $[A]$ thanks to Lemma~\ref{lem:two:rectangles}) and discarding the vertices on the other side of $\ell_-$ and $\ell_+$.
While elementary geometric considerations suffice to see that $|\bar P|\ge |P|+s^2$, $\bar P$ might not be fat. 
We correct this defect by several additional cuts near the top, bottom, left and right sides of $P$. In doing so, we discard a region of area much smaller than $|\bar P\setminus P|$. 
We note that the constant $C_6$ in the above argument will play an independent role from its occurrence in Lemma~\ref{lem:blob}, but is reused for convenience. We also restate Proposition~\ref{prop:step} for the reader's convenience.

\propstep*

\begin{figure}[t]
\centering
\begin{tikzpicture}[scale=0.3,line cap=round,line join=round,rotate=0,x=1cm,y=1cm]
\clip(-22,-10.8) rectangle (31.555829861442263,12.2);
\fill[fill=black,fill opacity=0.10000000149011612] (-4.337005381108515,5.988712578748761) -- (-6.309311075041817,3.356307775016381) -- (-3.6769062713094365,1.3840020810830784) -- (-1.7046005773761341,4.016406884815459) -- cycle;
\fill[fill=black,fill opacity=0.10000000149011612] (-4.337005381108515,5.988712578748761) -- (-3.469613458482735,2.8158312299431367) -- (-0.2967321096771114,3.6832231525689165) -- (-1.1641240323028912,6.856104501374541) -- cycle;
\fill[fill=black,fill opacity=0.10000000149011612] (11.555905471635983,10.333464551863495) -- (12.423297394261763,7.160583203057872) -- (15.596178743067389,8.027975125683652) -- (14.72878682044161,11.200856474489274) -- cycle;
\fill[fill=black,fill opacity=0.10000000149011612] (-13.955486488871877,-6.848920020785611) -- (-11.323081685139497,-8.82122571471891) -- (-9.350775991206195,-6.1888209109865295) -- (-11.983180794938571,-4.216515217053232) -- cycle;
\fill[fill=black,fill opacity=0.10000000149011612] (-11.323081685139497,-8.82122571471891) -- (-8.690676881407125,-10.79353140865221) -- (0.9278043835380159,2.0441014006701628) -- (-1.7046005773761341,4.016406884815459) -- cycle;
\fill[fill=black,fill opacity=0.10000000149011612] (-3.469613458482735,2.8158312299431367) -- (-2.6022215358569554,-0.3570501188624877) -- (16.463570665693165,4.855093776878029) -- (15.596178743067389,8.027975125683652) -- cycle;

\draw  (-15.803130385088094,-8.26023715799769)-- (-8,0);
\draw  (-8,0)-- (-6,2);
\draw  (0,6)-- (8,9);
\draw  (8,9)-- (16,12);
\draw  (16,12)-- (22,12);
\draw  (-15.803130385088094,-8.26023715799769)-- (-18,-16);
\draw  (-18,-16)-- (-18,-18);

\draw  (-7.323013693450844,3.984520540176266)-- (-1.3355258673267407,8.003288800990111);
\draw  (-7.323013693450844,3.984520540176266)-- (-6,2);
\draw  (-1.3355258673267407,8.003288800990111)-- (0,6);
\draw  (-4.337005381108515,5.988712578748761)-- (17.652018546459985,12);
\draw  (-4.337005381108515,5.988712578748761)-- (-16.285081624152532,-9.958193316364184);
\draw  (-4.337005381108515,5.988712578748761)-- (-6.309311075041817,3.356307775016381);
\draw  (-6.309311075041817,3.356307775016381)-- (-3.6769062713094365,1.3840020810830784);
\draw  (-3.6769062713094365,1.3840020810830784)-- (-1.7046005773761341,4.016406884815459);
\draw  (-1.7046005773761341,4.016406884815459)-- (-4.337005381108515,5.988712578748761);
\draw  (-4.337005381108515,5.988712578748761)-- (-3.469613458482735,2.8158312299431367);
\draw  (-3.469613458482735,2.8158312299431367)-- (-0.2967321096771114,3.6832231525689165);
\draw  (-0.2967321096771114,3.6832231525689165)-- (-1.1641240323028912,6.856104501374541);
\draw  (-1.1641240323028912,6.856104501374541)-- (-4.337005381108515,5.988712578748761);
\draw  (11.555905471635983,10.333464551863495)-- (12.423297394261763,7.160583203057872);
\draw  (12.423297394261763,7.160583203057872)-- (15.596178743067389,8.027975125683652);
\draw  (15.596178743067389,8.027975125683652)-- (14.72878682044161,11.200856474489274);
\draw  (14.72878682044161,11.200856474489274)-- (11.555905471635983,10.333464551863495);
\draw  (-13.955486488871877,-6.848920020785611)-- (-11.323081685139497,-8.82122571471891);
\draw  (-11.323081685139497,-8.82122571471891)-- (-9.350775991206195,-6.1888209109865295);
\draw  (-9.350775991206195,-6.1888209109865295)-- (-11.983180794938571,-4.216515217053232);
\draw  (-11.983180794938571,-4.216515217053232)-- (-13.955486488871877,-6.848920020785611);
\draw  (-11.323081685139497,-8.82122571471891)-- (-8.690676881407125,-10.79353140865221);
\draw  (-8.690676881407125,-10.79353140865221)-- (0.9278043835380159,2.0441014006701628);
\draw  (0.9278043835380159,2.0441014006701628)-- (-1.7046005773761341,4.016406884815459);
\draw  (-1.7046005773761341,4.016406884815459)-- (-11.323081685139497,-8.82122571471891);
\draw  (-3.469613458482735,2.8158312299431367)-- (-2.6022215358569554,-0.3570501188624877);
\draw  (-2.6022215358569554,-0.3570501188624877)-- (16.463570665693165,4.855093776878029);
\draw  (16.463570665693165,4.855093776878029)-- (15.596178743067389,8.027975125683652);
\draw  (15.596178743067389,8.027975125683652)-- (-3.469613458482735,2.8158312299431367);
\draw  (-6,2)-- (-3,4.5);
\draw  (-3,4.5)-- (0,6);
\draw [fill=black] (-4.337005381108515,5.988712578748761) circle (3.5pt);
\draw (-5,6.5) node {$z$};
\draw [fill=black] (17.652018546459985,12) circle (3.5pt);
\draw (17.7,11.4) node {$w'$};

\draw (3.433043572909109,9.234807923139085) node {$\ell_+$};
\draw [fill=black] (11.555905471635983,10.333464551863495) circle (3.5pt);
\draw (11,9.5) node {$u'$};
\draw [fill=black] (-16.285081624152532,-9.958193316364184) circle (1.5pt);
\draw (-17,-10) node {$w$};

\draw (-8.947096043238638,1.3380110486415382) node {\large{$\ell_-$}};
\draw [fill=black] (-11.983180794938571,-4.216515217053232) circle (3.5pt);
\draw (-12.5,-3.5) node {$u$};
\end{tikzpicture}
\caption{Illustration of the application of Lemma~\ref{lem:two:rectangles} in the proof of Proposition~\ref{prop:step}. The two gray $U$-shaped regions are copies of the region in Figure~\ref{fig:two:rectangles} for $\eps=1$ and possibly different values of the parameter~$\alpha$.}
\label{fig:rectangles}
\end{figure}

\begin{proof}
Suppose that $P$ is a fat horizontal convex set; if not, it suffices to exchange the $x$-axis and the $y$-axis in the following considerations. 
Fix an admissible vertical chord $P_{1,x}$, as given by Lemma~\ref{lem:adm_strips}, and let $\Pi$ be the rectangle ensured by Lemma~\ref{lem:rectangle}.
Without loss of generality, we assume that $\Pi$ intersects the upper boundary of $P$.
Let $z$ be the midpoint of the side $\Pi_+$ of $\Pi$, and consider a line $\ell$ through $z$ not intersecting $P$. 
Let $\ell_+$ (resp.\ $\ell_-$) be obtained by rotating $\ell$ around $z$ clockwise (resp.\ anti-clockwise) direction until the corresponding chord $u'w'=\ell_+\cap P$ (resp.\ $uw=\ell_-\cap P$) satisfies $|u'w'|=C_6s$ (resp.\ $|uw|=C_6s$) for the first time.
Moreover, we assume that $w,u,u',w'$ appear on $\partial P$ in this order, see Figure~\ref{fig:rectangles}.
Note that $uw$ may be contained in a side of $P$ of length larger than $C_6s$: in this case, $u$ is defined as the point of contact of $\ell_-$ and $P$ closest to $z$, and a similar convention holds for $u'$. 

Notice that, by Definition~\ref{def:polygon} (or Claim~\ref{cl:square}), $P$ contains an axis-parallel square of side length $C_5s$ with left side contained in the left side of $P$. 
Since $C_5\gg C_6$, this implies that $w$ is either on the upper boundary of $P$ or on its left side. A similar argument applies to $w'$.

By admissibility of the chord $P_{1,x}$ and using that $1\ll C_6\ll C_4\ll C_3$, the angle between $\ell_-$ and the $x$-axis is in $[-5\pi/11, 5\pi/11]$ and similarly for $\ell_+$.
In particular, $x_-+\sqrt 2 s\le \sigma_1(u) < \sigma_1(z) < \sigma_1(u')\le x_+-\sqrt 2 s$.
In turn, the latter observation and the (vertical) fatness of $P$ imply that $P$ contains two squares of dimensions $s\times s$ with one side lying on $\ell_-$ and corners $z$ and $u$, and a rectangle with dimensions $s$ and $|zu|+s$ sharing a segment of length $s$ with each of the said squares (an analogous statement for $\ell_+$ and $u'$ holds as well, see Figure~\ref{fig:rectangles}).
In both cases, applying Lemma~\ref{lem:two:rectangles} for the two squares (playing the roles of $\Gamma_1$ and $\Gamma_2$) and the rectangle (playing the role of $\Gamma$), we obtain that all vertices between $P$ and the segments $zu$ and $zu'$ become eventually infected.
Denote by $\bar P$ the convex set obtained by adding the triangle $uu'z$ to $P$ and discarding the parts of $P$ cut away by the lines $\ell_-$ and $\ell_+$. 

\begin{claim}\label{cl:w}
$|P\setminus \bar P|\le 2(C_6 s)^2$ and $|\bar P\setminus P|\ge C_4 s^2/8C_5$.
\end{claim}
\begin{proof}[Proof of Claim~\ref{cl:w}]
First, the inequality $|P\setminus \bar P|\le 2(C_6 s)^2$ follows from the fact that $P\setminus \bar P$ is contained in the union of the two axis-parallel rectangles with diagonals $uw$ and $u'w'$, and each of them has area at most $(C_6s)^2$.

Before proving the second inequality, we show by contradiction that $|wz|\ge C_4s$; the inequality $|w'z|\ge C_4 s$ follows similarly.
Suppose that $|wz| < C_4s$ and let $z_w$ be the unique point on the upper boundary of $P$ such that $zz_ww$ is a right angle (obtained by intersecting the upper boundary of $P$ with the circle with diameter $wz$). In particular, $\sigma_1(z_w)\in [x-C_4s, x+C_4s]$.
 However, $|zz_w|\ge s/2C_5$ and angle $z_wwz$ has measure $\arcsin(|zz_w|/|wz|) \ge \arcsin(1/(2C_5C_4))\gg 1/C_3$.
As a result, $\alpha^+_{x-C_4s}-\alpha^+_{x+C_4s}> 1/C_3$, contradicting the admissibility of $P_{1,x}$. Hence, $|wz|\ge C_4s$.

Now, we focus on the inequality $|\bar P\setminus P|\ge C_4 s^2/20C_5$.
Define $\hat u$ and $\hat u'$ to be the points on the upper boundary of $P$ with $x$-coordinates $x-C_4s/2$ and $x+C_4s/2$, respectively.
Since $\min\{|uz|,|u'z|\}\ge \min\{|wz|,|w'z|\}-C_6s\ge (C_4-C_6)s$ and $C_3\gg 1$, triangle $\hat u\hat u'z$ is contained in triangle $uu'z$ and more than half of its area is outside $P$.
At the same time, $|\hat u\hat u'| \ge C_4s$ and, by definition of $\Pi$ (see Lemma~\ref{lem:rectangle}), the altitude from $z$ to $\hat u\hat u'$ is at least $s/(2C_5)$.
Hence, triangle $\hat u\hat u'z$ has area at least $C_4 s^2/4C_5$ and, since at least half of it is outside $P$, the second inequality follows.
\end{proof}

We would ideally like to set $P' = \bar P$. Indeed, if $\bar P$ is fat, Claim~\ref{cl:w} would allow us to conclude.  
The aim of the proof is to discard some parts of $\bar P$ of suitably small total area, thus correcting the absence of fatness. To do so, we analyse different cases for the positions of $z,w,u,u',w'$ relative to the top, bottom, left and right sides of $P$. In fact, we explain the (local) modifications around $uw$ required to obtain $P'$: the modifications around $u'w'$ are similar and do not interfere with the ones around $uw$.

\vspace{1em}
\noindent
\textbf{Case 1.} Suppose that $\sigma_2(z) \ge y^+$. First of all, define $\bar y = \max\{y: |\bar{P}_{2,y}| = C_5s\}$ and note that the $x$-coordinates of the endpoints of $\bar P_{2,\bar y}$ are within distance $C_5s$ from $\sigma_1(z)$ and $C_5\ll C_4$.
Hence, discarding the part of $\bar P$ above $\bar P_{2,\bar y}$ reduces the area of $\bar P$ by at most $|B(z,C_5s)|\le 4(C_5 s)^2$.
It remains to analyse whether discarding additional vertices around $w$ (and similarly around $w'$) is necessary to preserve the vertical fatness.

\vspace{1em}
\noindent
\textbf{Case 1.1.} Suppose that $|\bar P_{1,x_-}|\ge C_5s$. Then, no further modifications around $w$ are necessary.

\vspace{1em}
\noindent
\textbf{Case 1.2.} Suppose that $|\bar P_{1,x_-}| < C_5s$; in particular, $\sigma_1(w) = x_-$. 
Define $\bar x = \min\{x: |\bar{P}_{1,x}| = C_5s\}$ and suppose that $\bar x\le \max\{\sigma_1(\bar P_{2,y_-})\} - C_5s$.
Note that $\bar x\in (x_-, x_-+C_6s]$: indeed, the vertical chords through $u$ in $P$ and $\bar P$ coincide and have length at least $C_5s$ by fatness of $P$, see Figure~\ref{fig:case1.2}.
Hence, discarding the part of $\bar P$ on the left of $\bar P_{1,\bar x}$ ensures that the left side of the obtained convex set has length $C_5s$, leaves its bottom side of length at least $C_5s$ and reduces the area of $\bar P$ by at most $C_5 s\cdot C_6s$.

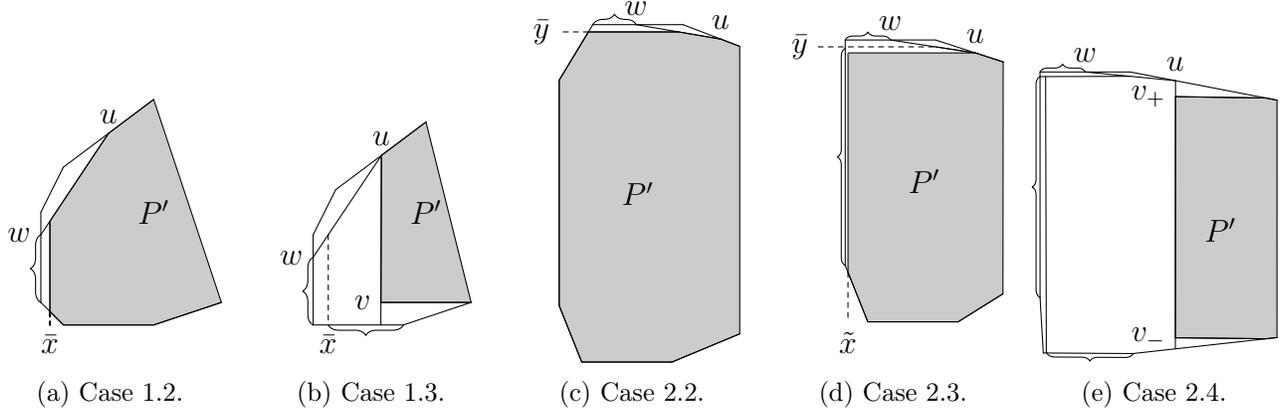
\begin{figure}[t]
     \centering
     \begin{subfigure}[b]{0.19\textwidth}
         \centering
\begin{tikzpicture}[scale=0.3,line cap=round,line join=round,x=1cm,y=1cm]

\fill[fill=black,fill opacity=0.2] (-3,8) -- (-4.999282986493508,6.500537760129869) -- (-7.599971324257282,2.5999713242572824) -- (-7.599971324257282,-1.4000286757427172) -- (-7,-2) -- (-3,-2) -- (0,-1) -- cycle;
\draw  (-8,3)-- (-8,-1);
\draw  (-8,3)-- (-7,5);
\draw  (-7,5)-- (-3,8);
\draw  (-4.999282986493508,6.500537760129869)-- (-8,2);
\draw  (-8,-1)-- (-7,-2);
\draw  (-7,-2)-- (-3,-2);
\draw  (-3,-2)-- (0,-1);
\draw  (-7.599971324257282,2.5999713242572824)-- (-7.599971324257282,-1.4000286757427172);
\draw  (-3,8)-- (-4.999282986493508,6.500537760129869);
\draw  (-4.999282986493508,6.500537760129869)-- (-7.599971324257282,2.5999713242572824);
\draw  (-7.599971324257282,2.5999713242572824)-- (-7.599971324257282,-1.4000286757427172);
\draw  (-7.599971324257282,-1.4000286757427172)-- (-7,-2);
\draw  (-7,-2)-- (-3,-2);
\draw  (-3,-2)-- (0,-1);
\draw  (0,-1)-- (-3,8);
\draw [line width=0.8pt,dash pattern=on 1.9pt off 1.9pt] (-7.6,-1.4) -- (-7.6,-2);
\draw [decorate,decoration={brace,amplitude=4pt},xshift=0pt,yshift=0pt]
(-8,-1) -- (-8,2);

\draw (-5,6.5) node[above] {\large{$u$}};
\draw (-8,2) node[left] {\large{$w$}};
\draw (-3,3) node {\large{$P'$}};
\draw (-7.6,-2) node[below] {\large{$\bar x$}};
\end{tikzpicture}
\caption{Case 1.2.}
\label{fig:case1.2}
     \end{subfigure}
     \begin{subfigure}[b]{0.19\textwidth}
         \centering
\begin{tikzpicture}[scale=0.3,line cap=round,line join=round,x=1cm,y=1cm]
\fill[fill=black,fill opacity=0.2] (9,8) -- (7.023541816948786,6.517656362711589) -- (7.003131506010898,-0.000009806426062283711) -- (11,0) -- cycle;
\draw  (9,8)-- (5,5);
\draw  (5,5)-- (4,3);
\draw  (4,3)-- (4,-1);
\draw  (4,-1)-- (8,-1);
\draw  (8,-1)-- (11,0);
\draw  (4,2)-- (7.023541816948786,6.517656362711589);
\draw  (7.023541816948786,6.517656362711589)-- (7,-1);
\draw  (7.003131506010898,-0.000009806426062283711)-- (11,0);
\draw  (9,8)-- (7.023541816948786,6.517656362711589);
\draw  (7.023541816948786,6.517656362711589)-- (7.003131506010898,-0.000009806426062283711);
\draw  (7.003131506010898,-0.000009806426062283711)-- (11,0);
\draw  (11,0)-- (9,8);
\draw [dash pattern=on 1.9pt off 1.9pt] (4.67,3) -- (4.67,-1);
\draw [decorate,decoration={brace,amplitude=4pt},xshift=0pt,yshift=0pt]
(8,-1) -- (4.67,-1);
\draw [decorate,decoration={brace,amplitude=4pt},xshift=0pt,yshift=0pt]
(4,-1) -- (4,2);

\draw (4.67,-1) node[below] {\large{$\bar x$}};
\draw (4,2) node[left] {\large{$w$}};
\draw (7.0235,6.518) node[above] {\large{$u$}};
\draw (7,0) node[left] {\large{$v$}};
\draw (9,4) node {\large{$P'$}};
\end{tikzpicture}
\caption{Case 1.3.}
\label{fig:case1.3}
     \end{subfigure}
     \begin{subfigure}[b]{0.19\textwidth}
         \centering
\begin{tikzpicture}[scale=0.3,line cap=round,line join=round,yscale=2.5,x=1cm,y=1cm]
\fill[line width=0.4pt,color=black,fill=black,fill opacity=0.2] (-5.488995772928142,3.615123580606025) -- (-6.330299916468469,3.744075225886489) -- (-8.194343665704999,3.8714539711318072) -- (-12.194343665704995,3.8714539711318072) -- (-13.488995772928142,3.015123580606025) -- (-13.488995772928142,-0.984876419393975) -- (-12.488995772928142,-1.984876419393975) -- (-8.488995772928142,-1.984876419393975) -- (-5.488995772928142,-1.484876419393975) -- cycle;

\draw [line width=0.4pt] (-8,4)-- (-12,4);
\draw [line width=0.4pt] (-12,4)-- (-13.488995772928142,3.015123580606025);
\draw [line width=0.4pt] (-13.488995772928142,3.015123580606025)-- (-13.488995772928142,-0.984876419393975);
\draw [line width=0.4pt] (-13.488995772928142,-0.984876419393975)-- (-12.488995772928142,-1.984876419393975);
\draw [line width=0.4pt] (-12.488995772928142,-1.984876419393975)-- (-8.488995772928142,-1.984876419393975);
\draw [line width=0.4pt] (-8.488995772928142,-1.984876419393975)-- (-5.488995772928142,-1.484876419393975);
\draw [line width=0.4pt] (-8,4)-- (-5.488995772928142,3.615123580606025);
\draw [line width=0.4pt] (-6.330299916468469,3.744075225886489)-- (-10.075469289445891,4);
\draw [line width=0.4pt] (-8.194343665704999,3.8714539711318072)-- (-12.194343665704995,3.8714539711318072);
\draw [line width=0.4pt,dash pattern=on 1.9pt off 1.9pt] (-12.194343665704995,3.8714539711318072)-- (-13.463634174373997,3.8714539711318072);
\draw [line width=0.4pt,color=black] (-5.488995772928142,3.615123580606025)-- (-6.330299916468469,3.744075225886489);
\draw [line width=0.4pt,color=black] (-6.330299916468469,3.744075225886489)-- (-8.194343665704999,3.8714539711318072);
\draw [line width=0.4pt,color=black] (-8.194343665704999,3.8714539711318072)-- (-12.194343665704995,3.8714539711318072);
\draw [line width=0.4pt,color=black] (-12.194343665704995,3.8714539711318072)-- (-13.488995772928142,3.015123580606025);
\draw [line width=0.4pt,color=black] (-13.488995772928142,3.015123580606025)-- (-13.488995772928142,-0.984876419393975);
\draw [line width=0.4pt,color=black] (-13.488995772928142,-0.984876419393975)-- (-12.488995772928142,-1.984876419393975);
\draw [line width=0.4pt,color=black] (-12.488995772928142,-1.984876419393975)-- (-8.488995772928142,-1.984876419393975);
\draw [line width=0.4pt,color=black] (-8.488995772928142,-1.984876419393975)-- (-5.488995772928142,-1.484876419393975);
\draw [line width=0.4pt,color=black] (-5.488995772928142,-1.484876419393975)-- (-5.488995772928142,3.615123580606025);

\draw [decorate,decoration={brace,amplitude=4pt},xshift=0pt,yshift=0pt]
(-12,4) -- (-10.075,4);

\draw (-6.33,3.74) node[above] {\large{$u$}};
\draw (-10.075,4) node[above] {\large{$w$}};
\draw (-13.463634174373997,3.8714539711318072) node[left] {\large{$\bar y$}};
\draw (-10,1) node {\large{$P'$}};
\end{tikzpicture}
\caption{Case 2.2.}
\label{fig:case2.2}
     \end{subfigure}
     \begin{subfigure}[b]{0.19\textwidth}
         \centering
\begin{tikzpicture}[scale=0.3,line cap=round,line join=round,yscale=2.5,x=1cm,y=1cm]

\fill[line width=0.4pt,color=black,fill=black,fill opacity=0.2] (4.999023650235393,3.6107719335267796) -- (3.7689878304589643,3.7704120429960075) -- 
(-1.8781699664039329,3.7704120429960075) -- (-1.8781699664039329,-0.1218300335960672) -- (-1,-1) -- (3,-1) -- (4.9882462399321525,-0.5033010132382648) -- cycle;

\draw [line width=0.4pt] (3.7689878304589643,3.7704120429960075) -- 
(-1.8781699664039329,3.7704120429960075);
\draw [line width=0.4pt] (-2,4)-- (2,4);
\draw [line width=0.4pt] (5.003314064531629,3.610215102419175)-- (2,4);
\draw [line width=0.4pt] (-2,4)-- (-2,0);
\draw [line width=0.4pt] (-2,0)-- (-1,-1);
\draw [line width=0.4pt] (-1,-1)-- (3,-1);
\draw [line width=0.4pt] (3,-1)-- (4.9882462399321525,-0.5033010132382648);
\draw [line width=0.4pt] (0,4)-- (3.7689878304589643,3.7704120429960075);
\draw [line width=0.4pt,dash pattern=on 1.9pt off 1.9pt] (-3.2,3.8781699664039326)-- (2,3.8781699664039326);
\draw [line width=0.4pt,dash pattern=on 1.9pt off 1.9pt] (-1.8781699664039329,-0.1218300335960672)-- (-1.8799670857858175,-1.005178812759611);
\draw [line width=0.4pt,color=black] (4.999023650235393,3.6107719335267796)-- (3.7689878304589643,3.7704120429960075);
\draw [line width=0.4pt,color=black] (3.7689878304589643,3.7704120429960075)-- (2,3.8781699664039326);
\draw [line width=0.4pt,color=black] (-1.8781699664039329,3.75)-- (-1.8781699664039329,-0.1218300335960672);
\draw [line width=0.4pt,color=black] (-1.8781699664039329,-0.1218300335960672)-- (-1,-1);
\draw [line width=0.4pt,color=black] (-1,-1)-- (3,-1);
\draw [line width=0.4pt,color=black] (3,-1)-- (4.9882462399321525,-0.5033010132382648);
\draw [line width=0.4pt,color=black] (4.9882462399321525,-0.5033010132382648)-- (4.999023650235393,3.6107719335267796);
\draw [decorate,decoration={brace,amplitude=3pt},xshift=0pt,yshift=0pt]
(-2,0)--(-2,3.85) node {};
\draw [decorate,decoration={brace,amplitude=3pt},xshift=0pt,yshift=0pt]
(-2,4)--(0,4) node {};

\draw (0,4) node[above] {\large{$w$}};
\draw (3.769,3.77) node[above] {\large{$u$}};
\draw (-1.88,-1) node[below] {\large{$\tilde x$}};
\draw (1.5,1.5) node {\large{$P'$}};
\draw (-3.2,3.88) node[left] {\large{$\bar y$}};
\end{tikzpicture}
\caption{Case 2.3.}
\label{fig:case2.3}
     \end{subfigure}
     \begin{subfigure}[b]{0.19\textwidth}
         \centering
\begin{tikzpicture}[scale=0.3,line cap=round,line join=round,yscale=2.5,x=1cm,y=1cm]

\fill[line width=0.4pt,color=black,fill=black,fill opacity=0.2] (18.5,3.5) -- (18.00294117647059,3.538235294117647) -- (14.002454150908752,3.562172788587173) -- (13.996349408848515,-0.7149699767805711) -- (18.012829426860307,-0.7328582943943924) -- (18.498374465722225,-0.7112395795747376) -- cycle;

\draw [line width=0.4pt] (8,4)-- (12,4);
\draw [line width=0.4pt] (12,4)-- (18.5,3.5);
\draw [line width=0.4pt] (8,4)-- (8,0);
\draw [line width=0.4pt] (8,0)-- (8.143357953746044,-0.9849921906326786);
\draw [line width=0.4pt] (8.143357953746044,-0.9849921906326786)-- (12.124236912143592,-0.9950457169313532);
\draw [line width=0.4pt] (12.124236912143592,-0.9950457169313532)-- (18.498374465722225,-0.7112395795747376);
\draw [line width=0.4pt] (14.002859161578709,3.8459339106477914)-- (10,4);
\draw [line width=0.4pt] (12.004316115599147,3.922855855454534)-- (8,3.9175421531195167);
\draw [line width=0.4pt] (8.247358310550638,3.917870396044005)-- (8.278459772697719,-0.9853333840485107);
\draw [line width=0.4pt] (14.002859161578709,3.8459339106477914)-- (13.996068612761416,-0.9117030968092386);
\draw [line width=0.4pt] (18.00294117647059,3.538235294117647)-- (14.002454150908752,3.562172788587173);
\draw [line width=0.4pt] (18.012829426860307,-0.7328582943943924)-- (13.996349408848515,-0.7149699767805711);
\draw [line width=0.4pt,color=black] (18.5,3.5)-- (18.00294117647059,3.538235294117647);
\draw [line width=0.4pt,color=black] (18.00294117647059,3.538235294117647)-- (14.002454150908752,3.562172788587173);
\draw [line width=0.4pt,color=black] (14.002454150908752,3.562172788587173)-- (13.996349408848515,-0.7149699767805711);
\draw [line width=0.4pt,color=black] (13.996349408848515,-0.7149699767805711)-- (18.012829426860307,-0.7328582943943924);
\draw [line width=0.4pt,color=black] (18.012829426860307,-0.7328582943943924)-- (18.498374465722225,-0.7112395795747376);
\draw [line width=0.4pt,color=black] (18.498374465722225,-0.7112395795747376)-- (18.5,3.5);
\draw [decorate,decoration={brace,amplitude=3pt},xshift=0pt,yshift=0pt]
(12.124236912143592,-0.9950457169313532) -- (8.3,-0.9849921906326786) node {};
\draw [decorate,decoration={brace,amplitude=3pt},xshift=0pt,yshift=0pt]
(8,4) -- (10,4) node {};
\draw [decorate,decoration={brace,amplitude=3pt},xshift=0pt,yshift=0pt]
(8,0) -- (8,3.9175421531195167) node {};
\draw (14,3.846) node[above] {\large{$u$}};
\draw (10,4) node[above] {\large{$w$}};
\draw (16,1.2) node {\large{$P'$}};
\draw (14,3.562) node[left] {\large{$v_+$}};
\draw (14,-0.715) node[left] {\large{$v_-$}};

\end{tikzpicture}
\caption{Case 2.4.}
\label{fig:case2.4}
     \end{subfigure}
        \caption{Illustrations of the proof of Proposition~\ref{prop:step}. Proportions are not respected for better visibility; the important segments whose length is smaller than $C_5s$ are marked by braces instead.}
        \label{fig:cases}
\end{figure}

\vspace{1em}
\noindent
\textbf{Case 1.3.} Suppose that $|\bar P_{1,x_-}| < C_3s$ and $\bar x > \max\{\sigma_1(\bar P_{2,y_-})\} - C_5s$. 
Note that $\sigma_2(w) \le \max\{\sigma_2(\bar P_{1,\bar x})\}$ and $\min\{\sigma_2(\bar P_{1,\bar x})\} = y_-$, see Figure~\ref{fig:case1.3}.
Hence, for every $x'\in [x_-, \sigma_1(u)]$,
\[\left|\bar P_{1,x'}\right|\le |\sigma_2(u)-y_-|\le |\sigma_2(u) - \sigma_2(w)| + \left|\bar P_{1,\bar x}\right|\le 2 C_5 s.\]

Now, denote by $\hat P$ the convex set obtained from $\bar P$ by discarding the part of $\bar P$ on the left of $\bar P_{1,\sigma_1(u)}$.
Moreover, fix the point $v\in \bar P_{1,\sigma_1(u)}$ with the smallest $y$-coordinate which satisfies $|\hat P_{2,\sigma_2(v)}| = C_5s$.
By Claim~\ref{cl:square}, $|uv|\ge C_5s$.
Hence, discarding the part of $\hat P$ below $\hat P_{2,\sigma_2(v)}$ ensures that each of the left vertical side and the lower horizontal side of the obtained convex set have length $C_5s$, and reduces the area of $\hat P$ by at most $C_5 s\cdot |\bar P_{1,\sigma_1(x)}|\le C_5s\cdot 2 C_5 s$.

Define $P'$ to be the convex set obtained by applying the described modifications around $uw$ and around $u'w'$. 
On the one hand, $P'\subset [A]\cap (P+B(0,s/C_5))\subset [A]\cap \Lambda_d$.
On the other hand, since $C_5,C_6\ll C_4$,
\begin{align*}
\left|P'\right|
&\ge |P| + \left|\bar P\setminus P\right| - \left|B(z,C_5s)\right| - \left|\bar P\setminus P\right| - \left|\hat P\setminus \bar P\right|\\
&\ge |P|+C_4s^2/8C_5
- 4(C_5s)^2 - 2\cdot C_5C_6s^2 - 2\cdot 2C_5^2 s^2\ge |P|+s^2,
\end{align*}
which finishes the proof in this case.

\vspace{1em}
\noindent
\textbf{Case 2.} Suppose that $\sigma_2(z) < y_+$. If $\sigma_1(z) < \min\{\sigma_1(P_{2,y_+})\}$, the modifications needed around $uw$ are exactly as described in Case 1 without any modification.
Since $\sigma_1(z)\in \sigma_1(P_{2,y_+})$ takes us back to Case~1, in the remainder of the proof, we assume that $\sigma_1(z) > \max\{\sigma_1(P_{2,y_+})\}$ and consider several subcases. While they are essentially analogous to Case 1, we spell the details out, since an additional sub-case arises.

\vspace{1em}
\noindent
\textbf{Case 2.1.} Suppose that $|\bar P_{2,y_+}|\ge C_5s$. Then, no further modifications around $w$ are necessary.

\vspace{1em}
\noindent
\textbf{Case 2.2.} Suppose that $|\bar P_{2,y_+}| < C_5s$; in particular, $\sigma_2(w) = y_+$. 
Define $\bar y = \min\{y: |\bar{P}_{2,y}| = C_5s\}$ and suppose that $\bar y\ge \min\{\sigma_2(\bar P_{1,x_-})\} + C_5s$.
Note that $\bar y\in [y_+-C_6s,y_+)$ since the horizontal chords through $u$ in $P$ and $\bar P$ coincide and have length at least $C_5s$ by fatness of $P$, see Figure~\ref{fig:case2.2}.
Hence, discarding the part of $\bar P$ above $\bar P_{2,\bar y}$ ensures that the upper horizontal side of the new convex set has length $C_5s$, leaves the left vertical side of length at least $C_5s$ and reduces the area of $\bar P$ by at most $C_5 s\cdot C_6s$.

\vspace{1em}
\noindent
\textbf{Case 2.3.} Suppose that $|\bar P_{2,y_+}| < C_5s$ and $\bar y < \min\{\sigma_2(\bar P_{1,x_-})\} + C_5s$. 
Specifically, $\sigma_1(w)\le \max\{\sigma_1(\bar P_{2,\bar y})\}$ and $\min\{\sigma_1(\bar P_{2,\bar y})\} = x_-$, implying that $\sigma_1(u) - x_-\le (\sigma_1(u) - \sigma_1(w)) + |\bar P_{2,\bar y}|\le (C_6+C_5)s$, see Figure~\ref{fig:case2.3}.

Define $\tilde P$ to be the convex set obtained from $\bar P$ by discarding the part of $\bar P$ above $\bar P_{2,\sigma_2(u)}$, and let $\tilde x = \min\{x: |\tilde{P}_{1,x}| = C_5s\}$.
Then, by Claim~\ref{cl:square}, one can deduce that $\tilde x\le \sigma_2(u)-C_5s$.

Suppose that $\tilde x\le \max\{\sigma_1(\tilde P_{1,y_-})\} - C_5s$.
Then, discarding the part of $\tilde P$ on the left of $\tilde P_{1,\tilde x}$ ensures that each of the upper horizontal, 
the left vertical and the lower horizontal sides of the new convex set have length at least $C_5s$ and reduces the area of $\tilde P$ by at most $C_5 s\cdot C_6s$.

\vspace{1em}
\noindent
\textbf{Case 2.4.} Finally, suppose that $|\bar P_{2,y_+}| < C_5s$, $\bar y < \min\{\sigma_2(\bar P_{1,x_-})\} + C_5s$ and $\tilde x > \max\{\sigma_1(\tilde P_{1,y_-})\} - C_5s$.
Then, $\min\{\sigma_2(\tilde P_{2,\tilde x})\} = y_-$ and $y_+ - y_- = (\sigma_2(w) - \sigma_2(u)) + |\sigma_2(\tilde P_{2,\tilde x})|\le (C_6+C_5)s$.

Denote by $\hat P$ the convex set obtained from $\bar P$ by discarding the part of $\bar P$ on the left of $\bar P_{1,\sigma_1(u)}$. Also, fix the points $v_+, v_-\in \bar P_{1,\sigma_1(u)}$ with the largest and the smallest $y$-coordinate, respectively, such that
\[\left|\hat P_{2,\sigma_2(v_+)}\right| = \left|\hat P_{2,\sigma_2(v_-)}\right| = C_5s,\] 
see Figure~\ref{fig:case2.4}.
By Claim~\ref{cl:square}, $|v_+v_-|\ge C_5s$.
Hence, discarding the part of $\hat P$ above $\hat P_{2,\sigma_2(v_+)}$ and below $\hat P_{2,\sigma_2(v_-)}$ ensures that each of the upper horizontal, the left vertical and the lower horizontal sides of the new convex set have length at least $C_5s$ and reduces the area of $\hat P$ by at most $C_5 s\cdot C_6s$.

By defining $P'$ to be the convex set obtained by applying the described modifications around $uw$ and around $u'w'$, one can verify that $P'\subset [A]\cap \Lambda_d$ and $|P'|\ge |P|+s^2$ as in Case 1, which finishes the~proof.
\end{proof}

\subsection{Expanding a very large rectangle}
\label{subsec:extensions}

Throughout this section, $C>0$ is a fixed constant chosen sufficiently large depending on $s$. We next adapt a tool from \cite{Bollobas15a}. 
\begin{definition}[Quasi-stable directions]
\label{def:quasi-stable}
The set of \emph{quasi-stable} directions is 
\[\cQ=\left\{u\in S^1:u\bbR\cap [-s,s]^2\cap\bbZ^2\neq \{0\}\right\}.\]
Two quasi-stable directions $u,v\in\cQ$ are \emph{consecutive}, if there does not exist $w\in\cQ$ such that $\<u,w\><0<\<v,w\>$.
\end{definition}
The next lemma follows directly from Definition~\ref{def:quasi-stable} and \eqref{eq:defrs} (see \cite{Bollobas15a}*{Lemma~5.3}).
\begin{lemma}
    For any two consecutive elements $u,v$ of $\cQ$, $|\Hb_u\cap\Hb_v\cap\cK\setminus\{0\}|\ge r$.
\end{lemma}
A \emph{droplet} is a convex set of the form 
\[\bigcap_{u\in\cQ}\Hb_u(l_u)\]
for $l\in(\bbR\cup\{\infty\})^\cQ$, where $\Hb_u(\infty)=\bbR^2$.
Recall that $C\gg s$ and that $\cS \subset \cQ$ denotes the set of stable directions. For $u\in \cQ$ and a droplet $D$, the \emph{$u$-side of $D$} is given by 
\[D\cap \partial\Hb_u\left(\inf\left\{l\in\bbR\cup\{\infty\}:\Hb_u(l)\supset D\right\}\right).\]
A droplet is \emph{non-degenerate} if, for every $u\in \cQ$, its $u$-side has Euclidean length at least $\sqrt[3]C$, if $u\in\cQ\setminus\cS$, and at least $\sqrt C$, if $u\in\cS$. 
Notice that $|\cQ|\le 4(s^2+1)$, so every axis-parallel square of side length $2\sqrt C(s^2+1)$ contains a non-degenerate droplet. Given a droplet 
\[D=\bigcap_{u\in\cQ}\Hb_u(l_u)\]
and $v\in\cQ$, the \emph{$v$-extension} of $D$ is the droplet 
\[D'=\left(\bigcap_{u\in\cQ\setminus\{v\}}\Hb_u(l_u)\right)\cap \Hb_v(l'_v)\]
with $l'_v>l_v$ minimal such that $(D'\setminus D)\cap\bbZ^2\neq\varnothing$. Notice that, if the droplet $D$ is non-degenerate, its extensions are well-defined: 
indeed, for all triplets of consecutive directions $v,u,w\in \cQ$ and large enough $C$, the triangle $(\Hb_v(l_v)\cap\Hb_w(l_w))\setminus\Hb_u(l_u)$ contains a unit square and, hence, an integer point as well.

For any $A,B\subset\bbR^2$, we denote by $A_0^B=A\cap\bbZ^2$ and, for every $t\ge 0$, 
\[A_{t+1}^B=A_t^B\cup\left\{x\in B\cap\bbZ^2:\left|(x+\cK)\cap A_{t}^B\right|\ge r\right\}.\]
Moreover, we denote $[A]_B=\bigcup_{t\ge 0}A_t^B$, that is, $[A]_B$ is the closure of $A$ for the bootstrap percolation process restricted to $B$.

The next lemma roughly states that a non-degenerate droplet can grow by itself in a direction $u\in \cQ\setminus \cS$, while a single infection nearby suffices for it to grow in a direction $u\in \cS$.
\begin{lemma}
\label{lem:extension}
    For every $u\in \cQ\setminus\cS$ and non-degenerate droplet $D$, the $u$-extension $D'$ of $D$ satisfies
    \[[D]_{D'\setminus D}\supset D'\cap\bbZ^2.\]
    For every $u\in\cS$, non-degenerate droplet $D$, the $u$-extension $D'$ of $D$ satisfies 
    \[[\{x\}\cup D]_{D'\setminus D}\supset D'\cap\bbZ^2\]
    for every $x\in (((D'\setminus D)\cap\bbZ^2)+\cK)\setminus D$.
\end{lemma}
\begin{proof}
    First, consider consecutive directions $v,u,w\in \cQ$ with $u\in\cQ\setminus\cS$. Then,
    $\bbH_u\cap \cK= \Hb_v\cap\Hb_w\cap \cK\setminus\{0\}$ has cardinal at least $r$.
    This immediately implies that, starting from $D\cap\bbZ^2$, $D'\cap\bbZ^2$ becomes infected on the first step.

    Next, assume that $u\in\cS$. Notice that $\Delta=(D'\setminus D)\cap\bbZ^2$ is a discrete segment. For any $y\in\Delta$, we have $|(y+\cK)\cap D|=r-1$ and there exists $z\in \cK$ with $\<z,u\>=0$. As a result, infection propagates along $\Delta$ starting from the initially infected site $x$.
\end{proof}

The next lemma allows us to assume that any rectangle is directed along a quasi-stable direction, see Figure~\ref{fig:rectangle:alignment} for an illustration.
\begin{lemma}
\label{lem:rectangle:alignment}
Fix $d\ge C$, $x\in\bbR^2$ and $v\in S^1$ such that the closest direction in $\cQ$ to $v$ is $u\in\cQ$, breaking ties arbitrarily. Let $u_-$ and $u_+$ be the consecutive directions to $u$ in $\cQ$ so that $v\in[u_-,u]$. Consider the rectangle
    \[R=x+\Hb_v\cap\Hb_{v+\pi/2}\cap\Hb_{v+\pi}(s)\cap\Hb_{v-\pi/2}(4d).\]
    Then, for some $x'\in R$, it holds that
    \[[R]\supset x'+\left(\Hb_u\cap\Hb_{u+\pi/2}\cap\Hb_{u+\pi}(s)\cap\Hb_{u-\pi/2}(d)\right)\cap\bbZ^2.\] 
\end{lemma}
\begin{proof}
If $u=v$, there is nothing to prove (take $x'=x$), so we assume $v\in[(u_-+u)/2,u)$.
Let $x,y,z,w$ be the corners of $R$ listed in clockwise direction, see Figure~\ref{fig:rectangle:alignment}. Let $\ell_+$ be the line orthogonal to $u_+$ through $w$, which intersects $xy$ at point $x'$. 
Also, let $\ell_-$ be the line orthogonal to $u_-$ through $z$, and denote $T=(x'+\Hb_u)\cap (z+\Hb_{u-})\setminus (x+\Hb_v)$. Since $s\ll d$, it suffices to prove that $T\subset [R]_{T}$.

Define $D$ as the trapezoid between the lines $\ell_-$, $\ell_+$ and the boundary of $x'+\Hb_u$, and $D' = D\cap (z+\Hb_u)$.
Notice that, as $R$ is fully infected and all sites at distance at most $s$ from $T$ in $D\setminus T$ are in $R$, we have $[R]_{T}\cap T=[R\cup D']_{D\setminus D'}\cap T$. Hence, it suffices to prove that $[R\cup D']_{D}\supset T$, which follows directly from repeated application of Lemma~\ref{lem:extension}.
\end{proof}

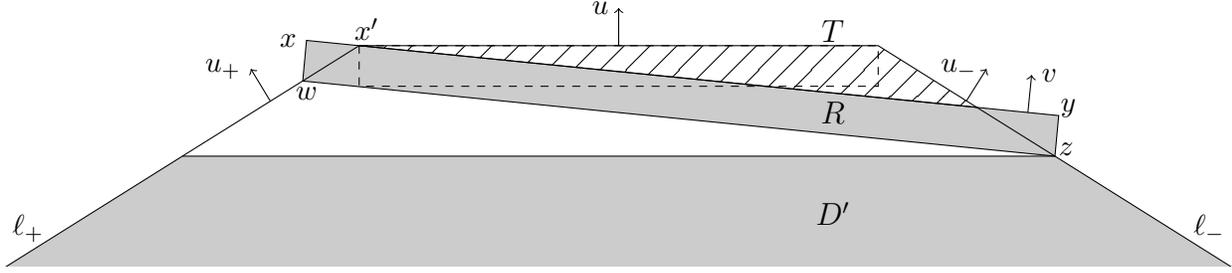
\begin{figure}
    \centering
    \begin{tikzpicture}
        \clip(-4,-3) rectangle (13,0.6);
        \draw[pattern=my north east lines] (0.7,-0.07) -- (7.6,-0.07) -- (8.91,-0.89) -- cycle;
        \draw[fill=black,fill opacity=0.2] (-0.05,-0.54) -- (0,0) node[text opacity=1,left]{$x$} -- (10,-1) -- (9.95,-1.54) -- cycle;
        \draw[fill=black,fill opacity=0.2] (-6.61,-4.65) -- (-1.65,-1.54) -- (9.95,-1.54) -- (15.92,-5.29) -- cycle;
        \draw[dashed] (0.7,-0.07) -- (0.7,-0.61) -- (7.6,-0.61) -- (7.6,-0.07) -- cycle;
        \draw (0.7,-0.07)-- (-1.65,-1.54);
        \draw (8.91,-0.89)-- (9.95,-1.54);
        \draw [->] (4.15,-0.07) -- (4.15,0.43) node[left]{$u$};
        \draw [->] (-0.48,-0.81) -- (-0.74,-0.38) node[left]{$u_+$};
        \draw [->] (8.77,-0.81) -- (9.04,-0.38) node[left]{$u_-$};
        \draw [->] (9.59,-0.96) -- (9.64,-0.46) node[right]{$v$};
        
        \draw (7,0.13) node {\large{$T$}};
        \draw (7,-0.96) node {\large{$R$}};
        \draw (7,-2.3) node {\large{$D'$}};

        \draw (10.13,-0.9) node {$y$};
        \draw (10.1,-1.45) node {$z$};
        \draw (0,-0.7) node {$w$};

        \draw (-3.7,-2.5) node {$\ell_+$};
        \draw (12,-2.5) node {$\ell_-$};
        \draw (0.8,0.15) node {$x'$};
\end{tikzpicture}
    \caption{Illustration of Lemma~\ref{lem:rectangle:alignment}. The shaded rectangle $R$ with dimensions $s\times 4d$ is infected. We prove the hatched triangle $T$ becomes infected by comparing to the process with both the shaded shapes $R$ and $D'$ infected. In particular, we infect the dashed rectangle with shorter side of length $s$ and longer side of length at least $d$.}
    \label{fig:rectangle:alignment}
\end{figure}

As we will see, if the direction $u$ in Lemma~\ref{lem:rectangle:alignment} is in $\cQ\setminus\cS$, 
it will be easy to produce a square of side length polynomial in $d$ from the rectangle. However, if $u\in\cS$, we need to ensure the presence of the additional infected site $x$ in Lemma~\ref{lem:extension}. For this, we require two preliminary lemmas.

\begin{lemma}
\label{lem:band}
    Fix $A\subset\bbZ^2$ and $u\in\cS$. Let $s_u=\max\{\<k,u\>:k\in\cK\}$. Consider the rectangles 
    \begin{align*}
    R&{}=x+\Hb_u\cap\Hb_{u+\pi/2}\cap\Hb_{u+\pi}(l)\cap\Hb_{u-\pi/2}(2s),\\
    R_+&{}=x+\Hb_u(s_u)\cap\Hb_{u+\pi/2}\cap\bbH_{u+\pi}\cap\Hb_{u-\pi/2}(2s),\\
    R_-&{}=x+\bbH_u(-l)\cap\Hb_{u+\pi/2}\cap\Hb_{u+\pi}(l+s_u)\cap\Hb_{u-\pi/2}(2s),
    \end{align*}
    for some $x\in\bbR^2$ and $l\ge 0$. Assume that $[A]\cap R_-=[A]\cap R_+=A\cap R=\varnothing$. Then, $[A]\cap R=\varnothing$.
\end{lemma}
\begin{proof}
    By symmetry, $u+\pi/2$ and $u-\pi/2$ are stable directions. Assume for contradiction that $R\cap [A]\neq \varnothing$, and consider $y\in R\cap [A]$ with minimal infection time. 
    Without loss of generality, let $\<y-x,u-\pi/2\>\le s$. Then, $r\le |(\cK+y)\setminus (R\cup R_+\cup R_-)|\le |\cK\cap \bbH_{u-\pi/2}|<r$
    since $u+\pi/2\in \cS$, a contradiction.
\end{proof}

The second preliminary lemma roughly says that if an infection is seen by an open half-plane, then it is seen by its boundary. 
While this is quite clear in the continuum, the discrete setting makes this property rather fragile.

\begin{lemma}[Infection does not jump over the boundary]
\label{lem:rotation}
    Let $u\in\cS$. Then, setting $s_u=\max\{\<k,u\>:k\in\cK\}$, we have
    \[\left(\bbH_u(s_u)\setminus\bbH_u\right)\cap\bbZ^2=\left((\cK+\bbH_u)\setminus\bbH_u\right)\cap\bbZ^2\subset \left(\left(\cK+\partial \Hb_u\right)\setminus\bbH_u\right)\cap\bbZ^2.\]
\end{lemma}
\begin{proof}
The equality is immediate since $\cK+\bbH_u=\bbH_u(s_u)$, so we focus on the inclusion. Let us first assume that $u=(0,1)$.
Fix $x\in \bbH_u\cap\bbZ^2$ and $k'\in (\cK+x)\setminus\bbH_u$.
If $k'\in \Hb_u$, the conclusion holds. 
Else, by $\pi/2$-rotation invariance of $K$ (the convex set such that $\cK=(sK)\cap\bbZ^2$), $sK+x$ contains the square with centre $x$ and corner $k'$, which intersects the line $\bbR\times \{0\}$ in an interval of length at least $2\<u,k'\>\ge 2$.
This interval necessarily contains an integer point $y$, which satisfies that $k'\in y+\cK$, as desired.

Next assume that $u=(1,1)/\sqrt{2}$. Then, either $k'\in \Hb_u$ or the square with centre $x$ and corner $k$ intersects the line $\mathbb R\cdot (-1,1)$ in a segment of length at least $2\<u,k'\>\ge \sqrt{2}$. 
Since $\sqrt{2}$ is exactly the distance between consecutive integer points on that line, the conclusion also holds in this case.

Finally, by \eqref{eq:stable}, the above two cases exhaust all possible stable directions up to symmetry.
\end{proof}

We are now ready to prove that a long thin rectangle in a quasi-stable direction infects a suitably large square.
\begin{proposition}
\label{prop:very:big:square}
Fix $x\in\bbR^2$, $u\in \cQ$ and $A\subset\bbZ^2$. Let 
\begin{equation}\label{eq:R}
R=x+\Hb_u\cap\Hb_{u+\pi/2}\cap\Hb_{u+\pi}(s)\cap\Hb_{u-\pi/2}(C).    
\end{equation}
Assume that $R\cap\bbZ^2\subset[A]$ and $d(A,R)\ge C$. Then, there exists $x'\in B(x,C)$ such that $R'\cap\bbZ^2\subset[A]$ with
\[R'=x'+\Hb_u\cap\Hb_{u+\pi/2}\cap\Hb_{u+\pi}\left(C^{3/4}\right)\cap\Hb_{u-\pi/2}\left(C^{3/4}\right).\]
\end{proposition}
\begin{proof}
We assume that $x=0$ for simplicity, but this does not affect the proof. Let $x,y,z,w$ be the corners of $R$ in clockwise order, and let $u_+, u, u_-$ be consecutive quasi-stable directions met in this clockwise order. 
Also consider the trapezoid $T=(w+\Hb_{u_+})\cap\Hb_u(C^{3/4})\cap(z+\Hb_{u_-})\setminus\Hb_u$ and the infinite droplet $D=(w+\Hb_{u_+})\cap \Hb_u\cap (z+\Hb_{u_-})$ (see Figure~\ref{fig:very:big:square:unstable}). Notice that the short base of $T$ has length at least $C-C^{3/4}/\tan
(u-u_-)-C^{3/4}/\tan
(u_+-u) > C^{3/4}$, so it suffices to show that $T\subset[A]$.

\begin{figure}
    \centering
\begin{tikzpicture}[x=1.5cm,y=1.5cm]
    \clip(-2,-1.5) rectangle (9,1.8);
    \draw[fill=black,fill opacity=0.3] (0,0) -- (7,0) -- (7,-0.5) -- (0,-0.5) -- cycle;
    \draw[fill=black,fill opacity=0.3] (-2.31,-1.95) -- (9.42,-2.01) -- (7,-0.5) -- (0,-0.5) -- cycle;
    \draw[pattern=my north east lines] (2.37,0.99) -- (0.8,0) -- (6.2,0) -- (4.63,0.99) -- cycle;
    \draw[pattern=mydots] (3.5,1.1) -- (4,1.1) -- (4,0) -- (3.5,0) -- cycle;
    \draw [->] (-1.16,-1.22) -- (-1.42,-0.8) node[right]{$u_+$};
    \draw [->] (3,0.99) -- (3,1.49) node[left]{$u$};
    \draw [->] (8.21,-1.26) -- (8.47,-0.83) node[left]{$u_-$};
    \draw (0.8,0)-- (0,-0.5);
    \draw (6.2,0)-- (7,-0.5);

    \draw (-0.15,0) node {$x$};
    \draw (7.15,0) node {$y$};
    \draw (7,-0.7) node {$z$};
    \draw (0,-0.7) node {$w$};

    \draw (2.15,1.05) node {$T$};
    \draw (2.5,-0.25) node {$R$};
    \draw (2.5,-1.25) node {$D$};
\end{tikzpicture}
    \caption{Illustration of Proposition~\ref{prop:very:big:square}. If the shaded rectangle $R$ is infected, the hatched trapezoid $T$ becomes infected as well. This is shown by comparison with the process with the shaded infinite droplet $D$ infected instead of the rectangle $R$. 
    In the case when $u$ is a stable direction, we additionally use infections found in the dotted rectangle of width $2s$.}
    \label{fig:very:big:square:unstable}
\end{figure}
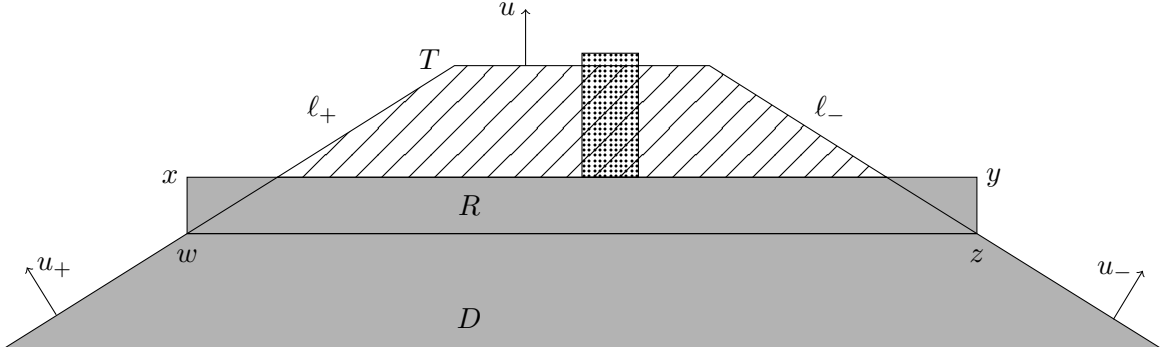

First, assume that $u\in\cQ\setminus\cS$. Then, $[R]_T\cap T=[D]_T\cap T$ as in the proof of Lemma~\ref{lem:rectangle:alignment}. Thus, a repeated application of Lemma~\ref{lem:extension} gives that $[D]_T\supset T$, as desired.

Next, assume that $u\in\cS$. We apply Lemma~\ref{lem:band} to $x\in\bbR^2$ such that $\<x,u-\pi/2\>=C/2$ and $l\ge s$ such that the rectangles $R_+$ and $R_-$ in Lemma~\ref{lem:band} are on different sides of $R$ defined in~\eqref{eq:R}. 
Up to switching $u$ and $u+\pi$ if necessary (this is the case if $R_+\cap [A]=\varnothing$), we obtain that, for every $l\in[0,C^{3/4}]$,
\[[A]\cap \Hb_u(s_u+l)\cap\Hb_{u+\pi/2}(-C/2)\cap\bbH_{u+\pi}(-l)\cap\Hb_{u-\pi/2}(C/2+2s)\neq\varnothing.\]
Let $A'\subset [A]$ be a minimal set of infections witnessing this for all values of $l$. Then, as in the unstable case, $[A'\cup R]_T\cap T=[A'\cup D]_T\cap T$. 
Hence, by applying Lemma~\ref{lem:extension} repeatedly, we obtain $[A'\cup D]_T\supset T$, as desired (note that Lemma~\ref{lem:extension} applies thanks to our choice of $s_u$ in Lemma~\ref{lem:band} and to Lemma~\ref{lem:rotation}).
\end{proof}

We can extract a non-degenerate droplet from the square provided by Proposition~\ref{prop:very:big:square}. 
This droplet will need to grow to much larger scales, possibly in the presence of a sparse set of infections, which prevents the use of Lemma~\ref{lem:band}. Instead, we rely on the following algorithm.

\begin{definition}[Extension algorithm]
\label{def:algorithm}
    Given a finite non-degenerate droplet $D$ and a set $A'\subset\bbZ^2$, we define a sequence of droplets $D_i$ as follows. Set $D_0=D$. Assume $D_i$ is defined for some $i$. 
    If there exists $u\in\cQ\setminus \cS$ such that the $u$-side of $D_i$ has length at least $2\sqrt[3] C$, choose one such $u$ arbitrarily and let $D_{i+1}$ be the $u$-extension of $D_i$. We refer to this operation as \emph{unstable extension}. 
    If no unstable extension is possible but there exists $u\in\cS$ such that the $u$-side of $D_i$ has length at least $2\sqrt C$, and there exists $x\in A'\cap((((D'\setminus D_i)\cap\bbZ^2)+\cK)\setminus D_i)$ where $D'$ is the $u$-extension of $D_i$, then set $D_{i+1}=D'$ for an arbitrary such choice of $u$. 
    We refer to this operation as \emph{stable extension}. If both unstable and stable extensions are impossible, the algorithm terminates and outputs the final droplet $D_i$.
\end{definition}

We next seek to show that, when applied to $A'=[A]$ for suitably sparse $A$ and a sufficiently large initial droplet, the extension algorithm does not terminate.
Recall the definition of $C$-connected set from Section~\ref{sec:theproof}.

\begin{lemma}
\label{lem:terminate}
Let $D$ be a non-degenerate droplet with diameter at least $C^{3/4}$. Let $A$ be such that $D\subset[A]$ and $A\cap D$ contains no $C$-connected set of $8$ vertices.
Then, some stable or unstable extension is possible taking this $D$ and $A'=[A]$ in Definition~\ref{def:algorithm}.
\end{lemma}
\begin{proof}
Assume that no unstable extension is possible. Then, for all $u\in\cQ\setminus\cS$, the length of the $u$-side of $D$ is at most $2\sqrt[3]C$. 
Since $\diam(D)\ge C^{3/4}$, there exist stable directions $u\in\cS$ such that the $u$-side of $D$ has length at least $(C^{3/4}-2|\cQ|\sqrt[3]C)/|\cS|>2\sqrt C$. 
We refer to $u\in\cS$ such that the $u$-side of $D$ has length at least $2\sqrt C$ as \emph{long} directions,
and denote the set of long directions by $\cL$. We refer to elements of $\cQ\setminus\cL$ as \emph{short} directions, and to their corresponding sides as \emph{short} sides.
Assume that stable extensions for long directions are not possible, that is, 
\begin{equation}
\label{eq:long:extensions}
[A]\cap ((D_u\setminus D)+\cK)\subset D
\end{equation}
for all $u\in\cL$, where $D_u$ is the $u$-extension of $D$.

Let $X=(D+\cK)\setminus D$. By definition of the closure, we have $[(A\cap D)\cup ([A]\cap X)]_D\cap D=[A]\cap D=D\cap\bbZ^2$. 
Consequently, setting $\bar A=(A\cap D)\cup ([A]\cap X)$, we have $[\bar A]\supset D$. We show that every $x\in[A]\cap X$ is at distance at most $2s$ from some short side of $D$. Indeed, by assuming the contrary for some $x\in[A]\cap X$, $x$ is at distance at most $s_u$ (recall Lemma~\ref{lem:rotation}) from the $u$-side of $D$ for some $u\in\cL$, but at distance at least $2s$ from its endpoints. By Lemma~\ref{lem:rotation}, this gives $x\in \cK+\partial\Hb_u(l_u)$, where $D_u=\bigcap_{v\in\cQ}\Hb_v(l_v)$, contradicting \eqref{eq:long:extensions}.
For each short direction $u\in\cQ$, let $R_u$ be the smallest axis-parallel rectangle containing all sites at distance at most $2s$ from the $u$-side of $D$. 
Then, by the above, $[A]\cap X\subset \bigcup_{u\in\cQ\setminus\cL} R_u$ and $\diam(R_u)\le 3\sqrt C$ for all $u\in\cQ\setminus\cL$.

We next run a crude merging process as in the proof of Claim~\ref{cl:cases} to bound from above the closure $\Delta$ of $(A\cap D)\cup\bigcup_{u\in\cQ\setminus\cL}R_u$. Start with the collection of axis-parallel rectangles given by the $R_u$ and a singleton rectangle for each element of $A\cap D$. 
If possible, arbitrarily pick two rectangles $R,R'$ in the current collection such that $d(R,R')\le 2s$ and replace them by the smallest axis-parallel rectangle containing both. Repeat this as long as possible. 
On the one hand, the union of the rectangles in the final collection contains $\Delta$. On the other hand, since $A\cap D$ contains no $C$-connected set of 8 vertices, for every rectangle $R$ in the initial collection, there are at most $8+|\cQ|$ other rectangles within distance $C^{2/3}$ from $R$ in the initial collection. 
Moreover, upon merging two rectangles, the maximum diameter of a rectangle in this merging step may increase at most by a factor of $3s$ 
(with the convention that the diameter of singleton rectangles is 1). 
But then, in the end of the process, each rectangle must have diameter at most $3\sqrt C(3s)^{8+|\cQ|} < C^{2/3}/2$ and, in particular, no two rectangles originally at distance more than $C^{2/3}$ can merge. 
Since one of these rectangles must contain $D\subset\Delta$, we have $\diam(D)\le 3\sqrt C(3s)^{8+|\cQ|} < C^{3/4}$, which contradicts the statement of the lemma and finishes the proof.
\end{proof}

Combining the results in this section, we are now able to prove the following proposition.
\begin{proposition}
\label{prop:very:big:to:supercritical}
Fix $d\ge C^4/2$. Let $A\subset\bbZ^2$ be such that $A\cap\Lambda_d$ contains no $C$-connected set of $8$ vertices and $A\cap\Lambda_{C^3/4}=\varnothing$.
Assume that there exists $x\in\Lambda_{C^3/8}$ and $v\in S^1$ such that the rectangle 
\begin{equation}
\label{eq:R:v:4C}
R=x+\Hb_v\cap\Hb_{v+\pi/2}\cap\Hb_{v+\pi}(s)\cap\Hb_{v-\pi/2}(4C)
\end{equation}
satisfies $R\cap\bbZ^2\subset[A]$. Then, there exists $u\in\cS$ and $\tilde x\in\Lambda_d$ such that $\tilde R\cap\bbZ^2\subset[A]$ for 
\[\tilde R=\tilde x+\Hb_u\cap\Hb_{u+\pi/2}\cap\Hb_{u+\pi}(s)\cap\Hb_{u-\pi/2}(d/9).\]
\end{proposition}
\begin{proof}
We start by applying Lemma~\ref{lem:rectangle:alignment} to obtain a rectangle of the form $R'=x'+\Hb_{w}\cap\Hb_{w+\pi/2}\cap\Hb_{w+\pi}(s)\cap\Hb_{w-\pi/2}(C)$ with $x'\in \Lambda_{C^3/8+5C}$ and $w\in\cQ$ such that $R'\cap\bbZ^2\subset [A]$. We may then apply Proposition~\ref{prop:very:big:square} to find a square
\[R''=x''+\Hb_{w}\cap\Hb_{w+\pi/2}\cap\Hb_{w+\pi}\left(C^{3/4}\right)\cap\Hb_{w-\pi/2}\left(C^{3/4}\right)\]
with $x''\in\Lambda_{C^3/8+6C}$ such that $R''\cap\bbZ^2\subset[A]$. 
Moreover, observe that $R''$ contains a non-degenerate droplet with diameter at least $C^{3/4}$, to which we can apply the extension algorithm with $A'=[A]$. 
By Lemma~\ref{lem:terminate}, the algorithm does not terminate until the droplet reaches $\partial\Lambda_d$. Let $D$ be the first droplet obtained in the extension algorithm of diameter at least $d-(C^3/8+7C)$. We will extract the rectangle $\tilde R$ from $D$, but some care is needed to ensure that it has a side parallel to some $u\in\cS$.

We require a few properties of the extension algorithm, which are not hard to verify by induction.
\begin{itemize}
    \item Each extension increases the diameter by at most an absolute constant. 
    \item Each droplet $D'$ obtained in the algorithm is non-degenerate and satisfies $D'\cap\bbZ^2\subset[A]$. 
    \item The droplets in the algorithm are nested and the number of integer points they contain is increasing. 
    \item Each extension only modifies side lengths by at most a constant amount depending on $s$ and, consequently, sides never become shorter than $\sqrt[3]C$.
\end{itemize}

We further make the following claim.
\begin{claim}
\label{claim:algo}
    The total length of unstable sides of any droplet in the algorithm is at most $8C^{3/4}$.
\end{claim}
\begin{proof}
To begin with, $R''$ is contained in an axis-parallel square of side $\sqrt2 C^{3/4}$. Recalling \eqref{eq:stable}, we obtain that this square contains any droplet in the extension algorithm up to the first stable extension. 
Since the perimeter of a convex set is increasing with respect to inclusion, this implies the desired result for these droplets. 

Let $D'=\bigcap_{u\in\cQ}\Hb_u(l_u)$ be a droplet formed by a stable extension (not necessarily the first one), and suppose that the claim holds at all previous steps.
By Definition~\ref{def:algorithm}, just before obtaining $D'$, the total length of unstable sides is at most $2\sqrt[3]C |\cQ|<C^{3/4}/2$, and the stable extension producing $D'$ increases this amount by at most $C^{3/4}/2$.

We show the claim for all steps until the next stable extension.
Define $D''= \bigcap_{u\in\cS}\Hb_u(l_u)$ and note that unstable extensions produce droplets contained in $D''$.
Then, the total length of unstable sides of any droplet $D'''$ such that $D'\subset D'''\subset D''$ is at most the difference between the total stable side lengths of $D''$ and $D'$, that is, $|\partial D''\setminus \partial D'|$. Moreover, recalling \eqref{eq:stable}, $D''\setminus D'$ is a union of $|\cS|$ regions (one at each corner of $D''$) such that, for each such region $R$, there exist points $x,y,z\in \mathbb R^2$ satisfying each of the following properties:
\begin{itemize}
    \item the segments $xy$ and $yz$ are contained in $(\partial D''\setminus \partial D')\cup \{x,z\}$,
    \item the angle $xyz$ has measure at least $\pi/2$,
    \item $\partial D'\cap R$ is a monotone curve connecting $x$ and $z$ inside the triangle $xyz$.
\end{itemize} 
As a result, for each such region $R$, $|R\cap \partial D'|\le |R\cap \partial D''|\le \sqrt{2}|R\cap \partial D'|$.
Therefore, the total length of unstable sides of every droplet $D'''$ as above is at most $C^{3/4}+|\partial D''\setminus \partial D'|\le (1+\sqrt 2)C^{3/4}$, as desired.
\end{proof}

By Claim~\ref{claim:algo}, there exists at least one $u\in\cS$ such that the $u$-side of $D$ has length at least $(\diam(D)-8C^{3/4})/|\cS|\ge d/9$. But the consecutive directions $u_+,u_-$ satisfy 
$|u_+-u|, |u-u_-|\ge c(s)$ for some $c(s)>0$ and the $u_+$ and $u_-$-sides of $D$ have length at least $\sqrt[3]C\gg s/c(s)$. 
Therefore, by convexity of $D$, we can fit the desired rectangle $\tilde R$ inside the quadrilateral spanned by the $u_-$, $u$ and $u_+$-sides of $D$, concluding the proof.
\end{proof}

We are now ready to conclude the proof of our main result restated here for convenience.

\propmain*

\begin{proof}[Proof of Proposition~\ref{prop:long:rectangle} for $\cK=\cK_s$]
Fix $A$ as in the statement. In order to apply the results of previous sections, we need to deal with the possible presence of one infection close to the origin, rather than none. Assume that $A\cap\Lambda_{C^3/4}=A\cap\Lambda_{C^3}=\{x\}$. In this case, we notice that, before any infection reaches distance $2s$ from $x$, necessarily a site $y$ within distance $2s$ from $\partial \Lambda_{2C^3/3}$ needs to become infected. Thus, by replacing $A$ with $A-y$, we may assume that $A\cap\Lambda_{C^3/4}=\varnothing$ and $A\cap \Lambda_{d-C^3}$ contains no $C$-connected set of 8 vertices, while $0\in[A]\setminus[A\cap\Lambda_{d-C^3}]$.

Since $C^3/4\ge C_1s$, Lemma~\ref{lem:big square} shows that there exists an axis-parallel square $P^{(0)}\subset\Lambda_{C_1s}$ of side length $C_2s$ with $P^{(0)}\cap \bbZ^2\subset[A]$. Since $C_2\ge C_5$, $P^{(0)}$ is a fat convex set. We may then apply Proposition~\ref{prop:step} repeatedly to obtain a sequence of fat convex sets $(P^{(j)})_{j=0}^k$ such that $k=\lceil C^3/(20s)\rceil$ and, for all $j\in\{1,\dots,k\}$,
\begin{align*}
\left|P^{(j)}\right|-\left|P^{(j-1)}\right|&{}\ge s^2,
&P^{(j)}&{}\subset \Lambda_{C_1s+2sj},&P^{(j)}\cap\bbZ^2&{}\subset[A].
\end{align*}
This immediately yields $|P^{(k)}|\ge C^3s$.

Recall that $P^{(k)}$ is fat and assume without loss of generality that it is horizontal. Then, $P^{(k)}$ contains a parallelogram of sides $C_5s$ and $\diam (P^{(k)})/2$, and angles bounded away from $0$. 
Since $\diam(P^{(k)})\ge \sqrt{|P^{(k)}|/\pi}\ge C^{3/2}$, this parallelogram contains a rectangle $R$ of the form \eqref{eq:R:v:4C}. Since $P^{(k)}\cap\bbZ^2\subset[A]\cap \Lambda_{C^3/8}$, $R$ satisfies the hypotheses of Proposition~\ref{prop:very:big:to:supercritical}. Applying this proposition completes the proof of Proposition~\ref{prop:long:rectangle} and, thereby, of Theorem~\ref{th:main} for the case $\cK=\cK_s$ for $s$ large enough.
\end{proof}

\section*{Acknowledgements}
I.H.~was supported by the Austrian Science Fund (FWF) grant No.~P35428-N. L.L.~was supported by the Austrian Science Fund (FWF) grant No.~10.55776/ESP624. We thank Itai Benjamini, Moritz Dober and Asaf Shapira for interesting discussions, and the referees for many helpful comments on the presentation. We are also grateful to the organisers of the journ\'ees ALEA at CIRM in 2019, where the work was initiated.

\appendix
\section{Proof of Lemma~\ref{lem:long:rectangle:suffices}}
\label{app}
We recall the statement for convenience.
\lemBSU*
\begin{proof}
    Recall Definition~\ref{def:quasi-stable}. 
    We make the convention that $\cQ=\cS_\square$ for the $\square$ and $\cQ=\cS_\boxtimes$ for the $\triangle$ model. 
    Further recall droplets and extensions from Section~\ref{subsec:extensions}.We see $A$ as a periodic subset of $\bbZ^2$ and show that $[A]=\bbZ^2$. 
    
    We follow the proof of Proposition~\ref{prop:very:big:square} (recall Figure~\ref{fig:very:big:square:unstable}). Note that every rectangle $R$ as described in the statement can be written as
    \[R=x+\Hb_u\cap\Hb_{u+\pi/2}\cap\Hb_{u+\pi}(\|\cK\|)\cap\Hb_{u-\pi/2}(d)\]
    for some $x\in \bbR^2$ and $u\in \cQ$. Fix such $x,u$ and $R$, let $x,y,z,w$ be the corners of $R$ in clockwise order, and let $u_+,u,u_-$ be consecutive quasi-stable directions in this clockwise order. Consider the trapezoid 
    \[T{}=(w+\Hb_{u_+})\cap\Hb_u(\varepsilon d)\cap(z+\Hb_{u_-})\setminus\Hb_u\]
    and the infinite droplet
    \[D{}=(w+\Hb_{u_+})\cap\Hb_u\cap(z+\Hb_{u_-}).\]
    We first claim that $[D\cup (A\cap T)]_T\supset T$ for all $\eps$ small enough depending on $\|\cK\|$ so that the $u$-side of $T$ (that is, the side with outer normal $u$) has length at least $\eps d$. 
    Indeed, since $A$ intersects every integer point segment of length $\eps^3 d$, our claim follows by applying Lemma~\ref{lem:extension} repeatedly in direction $u$. 
    As in the proof of Proposition~\ref{prop:very:big:square}, we deduce that $[R\cup (A\cap T)]\supset T$.
    
    Next, observe that, for suitably small $\eps$, there exists a droplet $\Delta\subset T\subset[A]$ each of whose sides has length larger than $\varepsilon^2 d$. 
    Consider the sequence of droplets $(\Delta_i)_{i\ge 0}$ such that $\Delta_0=\Delta$ and, for every $i\ge 0$, $\Delta_{i+1}$ is the $v_i$-extension of $\Delta_i$, where $v_i\in\cQ$ is such that the length of the $v_i$-side of $\Delta_i$ is maximal (with ties broken arbitrarily). We show in two steps that, for every $i\ge 0$, every side of $\Delta_i$ has length at least $\eps^3 d$. 
    First, note that the $v_i$-extension replaces the $v_i$-side of $\Delta_i$ by a trapezoid having this side as a base.
    Combining this fact with the triangle inequality shows that the perimeter of $\Delta_i$ increases (strictly) with $i$. Second, for each $v\in\cQ\setminus\{v_i\}$, the length of the $v$-side of $\Delta_{i+1}$ is not smaller than the length of the $v$-side of $\Delta_i$, while the length of the $v_i$-side of $\Delta_i$ is larger than the length of the $v_i$-side of $\Delta_{i+1}$ by at most $1/\varepsilon$. 
    Thus, for every $i\ge 0$ and $\eps$ small, each side of $\Delta_i$ has length at least $|\partial \Delta_i|/|\cQ|-1/\varepsilon\ge \eps^2d-1/\eps>\varepsilon^3d$, as desired.

    Finally, since the perimeter of $\Delta_i$ increases strictly, by the previous argument there is $i$ such that the length of the shortest side of $\Delta_i$ is at least $n$, and thus $\Delta_i\subseteq [A]$ contains an $n\times n$ square. This implies that $[A]\supset \bbT$, completing the proof.
\end{proof}

\bibliographystyle{plain}
\bibliography{Bib}

@article{Aizenman88,
  title = {Metastability effects in bootstrap percolation},
  author = {Aizenman, M. and Lebowitz, J. L.},
  year = 1988,
  journal = {J. Phys. A},
  volume = {21},
  number = {19},
  pages = {3801--3813},
  issn = {0305-4470, 1361-6447},
  doi = {10.1088/0305-4470/21/19/017},
  url = {https://iopscience.iop.org/article/10.1088/0305-4470/21/19/017},
  mrnumber = {968311}
}

@article{Anastos21,
  title = {Packing {{Hamilton}} cycles in cores of random graphs},
  author = {Anastos, M.},
  year = 2021,
  journal = {arXiv e-prints},
  eprint = {2107.03527},
  url = {http://arxiv.org/abs/2107.03527},
  archiveprefix = {arXiv}
}

@article{Balister25,
  title = {The critical length for growing a droplet},
  author = {Balister, P. and Bollob{\'a}s, B. and Morris, R. and Smith, P.},
  year = 2025,
  journal = {Mem. Amer. Math. Soc.},
  volume = {310},
  number = {1571},
  pages = {v+180},
  issn = {0065-9266},
  doi = {10.1090/memo/1571},
  url = {https://doi.org/10.1090/memo/1571},
  mrnumber = {4935836}
}

@article{Balogh03,
  title = {Sharp thresholds in bootstrap percolation},
  author = {Balogh, J. and Bollob{\'a}s, B.},
  year = 2003,
  journal = {Phys. A},
  volume = {326},
  number = {3-4},
  pages = {305--312},
  url = {http://chess.eecs.berkeley.edu/pubs/720.html}
}

@article{Balogh06a,
  title = {Bootstrap percolation on infinite trees and non-amenable groups},
  author = {Balogh, J. and Peres, Y. and Pete, G.},
  year = 2006,
  journal = {Combin. Probab. Comput.},
  volume = {15},
  number = {5},
  pages = {715--730},
  doi = {10.1017/S0963548306007619},
  mrnumber = {2248323}
}

@article{Balogh12,
  title = {The sharp threshold for bootstrap percolation in all dimensions},
  author = {Balogh, J. and Bollob{\'a}s, B. and {Duminil-Copin}, H. and Morris, R.},
  year = 2012,
  journal = {Trans. Amer. Math. Soc.},
  volume = {364},
  number = {5},
  pages = {2667--2701},
  issn = {0002-9947},
  doi = {10.1090/S0002-9947-2011-05552-2},
  url = {https://doi.org/10.1090/S0002-9947-2011-05552-2},
  mrnumber = {2888224}
}

@article{Bartha15,
  title = {Noise sensitivity in bootstrap percolation},
  author = {Bartha, Z. and Pete, G.},
  year = 2015,
  journal = {arXiv e-prints},
  eprint = {1509.08454},
  url = {http://arxiv.org/abs/1509.08454},
  archiveprefix = {arXiv}
}

@misc{Benjamini19,
  title = {Private communication},
  author = {Benjamini, I.},
  year = 2019
}

@article{Bollobas15a,
  title = {Monotone cellular automata in a random environment},
  author = {Bollob{\'a}s, B. and Smith, P. and Uzzell, A.},
  year = 2015,
  journal = {Combin. Probab. Comput.},
  volume = {24},
  number = {4},
  pages = {687--722},
  issn = {0963-5483},
  doi = {10.1017/S0963548315000012},
  url = {http://dx.doi.org/10.1017/S0963548315000012},
  mrnumber = {3350030}
}

@article{Bollobas23,
  title = {Universality for two-dimensional critical cellular automata},
  author = {Bollob{\'a}s, B. and {Duminil-Copin}, H. and Morris, R. and Smith, P.},
  year = 2023,
  journal = {Proc. Lond. Math. Soc. (3)},
  volume = {126},
  number = {2},
  pages = {620--703},
  issn = {0024-6115},
  doi = {10.1112/plms.12497},
  url = {https://doi.org/10.1112/plms.12497},
  mrnumber = {4550150}
}

@article{Bresar24,
  title = {Bootstrap percolation in strong products of graphs},
  author = {Bre{\v s}ar, B. and Hed{\v z}et, J.},
  year = 2024,
  journal = {Electron. J. Combin.},
  volume = {31},
  number = {4},
  pages = {Paper No. 4.35, 22},
  doi = {10.37236/11826},
  url = {https://doi.org/10.37236/11826},
  mrnumber = {4826762}
}

@article{Chalupa79,
  title = {Bootstrap percolation on a {{Bethe}} lattice},
  author = {Chalupa, J. and Leath, P. L. and Reich, G. R.},
  year = 1979,
  journal = {J. Phys. C},
  volume = {12},
  number = {1},
  pages = {L31-L35},
  issn = {0022-3719},
  doi = {10.1088/0022-3719/12/1/008},
  url = {https://iopscience.iop.org/article/10.1088/0022-3719/12/1/008}
}

@article{Dembo08,
  title = {Finite size scaling for the core of large random hypergraphs},
  author = {Dembo, A. and Montanari, A.},
  year = 2008,
  journal = {Ann. Appl. Probab.},
  volume = {18},
  number = {5},
  pages = {1993--2040},
  issn = {1050-5164},
  doi = {10.1214/07-AAP514},
  url = {https://doi.org/10.1214/07-AAP514},
  mrnumber = {2462557}
}

@article{Duminil-Copin24,
  title = {Sharp metastability transition for two-dimensional bootstrap percolation with symmetric isotropic threshold rules},
  shorttitle = {Sharp metastability for bootstrap percolation},
  author = {{Duminil-Copin}, H. and Hartarsky, I.},
  year = 2024,
  journal = {Probab. Theory Related Fields},
  volume = {190},
  number = {1-2},
  pages = {445--483},
  issn = {0178-8051},
  doi = {10.1007/s00440-024-01310-3},
  url = {https://doi.org/10.1007/s00440-024-01310-3},
  mrnumber = {4797373}
}

@article{Ferber23,
  author = {Ferber, A. and Kwan, M. and Sah, A. and Sawhney, M.},
  year = 2023,
  journal = {Duke Math. J.},
  volume = {172},
  number = {7},
  pages = {1293--1332},
  issn = {0012-7094},
  doi = {10.1215/00127094-2022-0060},
  url = {https://doi.org/10.1215/00127094-2022-0060},
  mrnumber = {4583652},
  title = {Singularity of the {$k$}-core of a random graph}
}

@article{Gao18,
  title = {The stripping process can be slow: {{Part II}}},
  author = {Gao, P.},
  year = 2018,
  journal = {SIAM J. Discrete Math.},
  volume = {32},
  number = {2},
  pages = {1159--1188},
  issn = {0895-4801},
  doi = {10.1137/15M1054948},
  url = {https://doi.org/10.1137/15M1054948},
  mrnumber = {3809529}
}

@article{Gao18a,
  title = {The stripping process can be slow: {{Part I}}},
  author = {Gao, P. and Molloy, M.},
  year = 2018,
  journal = {Random Structures Algorithms},
  volume = {53},
  number = {1},
  pages = {76--139},
  issn = {1042-9832},
  doi = {10.1002/rsa.20760},
  url = {https://doi.org/10.1002/rsa.20760},
  mrnumber = {3829437}
}

@article{Gravner08,
  title = {Slow convergence in bootstrap percolation},
  author = {Gravner, J. and Holroyd, A. E.},
  year = 2008,
  journal = {Ann. Appl. Probab.},
  volume = {18},
  number = {3},
  pages = {909--928},
  issn = {1050-5164},
  doi = {10.1214/07-AAP473},
  url = {http://dx.doi.org/10.1214/07-AAP473},
  mrnumber = {2418233}
}

@incollection{Gravner99,
  title = {Scaling laws for a class of critical cellular automaton growth rules},
  booktitle = {Random walks ({{Budapest}}, 1998)},
  author = {Gravner, J. and Griffeath, D.},
  year = 1999,
  series = {Bolyai {{Soc}}. {{Math}}. {{Stud}}.},
  volume = {9},
  pages = {167--186},
  publisher = {J\'anos Bolyai Math. Soc., Budapest},
  mrnumber = {1752894}
}

@phdthesis{Hartarsky22phd,
  title = {Bootstrap percolation and kinetically constrained models: two-dimensional universality and beyond},
  author = {Hartarsky, I.},
  year = 2022,
  organization = {Universit\'e Paris Dauphine, PSL University},
  school = {Universit\'e Paris Dauphine, PSL University},
  note = {Available at \url{https://tel.archives-ouvertes.fr/tel-03527333}}
}

@article{Hartarsky24locality,
  title = {Bootstrap percolation is local},
  author = {Hartarsky, I. and Teixeira, A.},
  year = 2024,
  journal = {arXiv e-prints},
  eprint = {2404.07903},
  url = {http://arxiv.org/abs/2404.07903},
  archiveprefix = {arXiv}
}

@book{Janson00,
  title = {Random graphs},
  author = {Janson, S. and {\L}uczak, T. and Rucinski, A.},
  year = 2000,
  series = {Wiley-interscience series in discrete mathematics and optimization},
  publisher = {Wiley-Interscience, New York},
  doi = {10.1002/9781118032718},
  url = {https://doi.org/10.1002/9781118032718},
  isbn = {0-471-17541-2},
  mrnumber = {1782847}
}

@article{Kogut81,
  title = {Bootstrap percolation transitions on real lattices},
  author = {Kogut, P. M. and Leath, P. L.},
  year = 1981,
  journal = {J. Phys. C},
  volume = {14},
  number = {22},
  pages = {3187--3194},
  issn = {0022-3719},
  doi = {10.1088/0022-3719/14/22/013},
  url = {https://iopscience.iop.org/article/10.1088/0022-3719/14/22/013}
}

@article{Kong19,
  author = {Kong, Y.-X. and Shi, G.-Y. and Wu, R.-J. and Zhang, Y.-C.},
  year = 2019,
  journal = {Phys. Rep.},
  volume = {832},
  pages = {1--32},
  issn = {0370-1573},
  doi = {10.1016/j.physrep.2019.10.004},
  url = {https://doi.org/10.1016/j.physrep.2019.10.004},
  mrnumber = {4035043},
  title = {{$k$}-core: theories and applications}
}

@article{Krivelevich14,
  title = {Cores of random graphs are born {{Hamiltonian}}},
  author = {Krivelevich, M. and Lubetzky, E. and Sudakov, B.},
  year = 2014,
  journal = {Proc. Lond. Math. Soc. (3)},
  volume = {109},
  number = {1},
  pages = {161--188},
  issn = {0024-6115},
  doi = {10.1112/plms/pdu003},
  url = {https://doi.org/10.1112/plms/pdu003},
  mrnumber = {3237739}
}

@article{Luczak91,
  author = {{\L}uczak, T.},
  year = 1991,
  journal = {Discrete Math.},
  volume = {91},
  number = {1},
  pages = {61--68},
  issn = {0012-365X},
  doi = {10.1016/0012-365X(91)90162-U},
  url = {https://doi.org/10.1016/0012-365X(91)90162-U},
  mrnumber = {1120887},
  title = {Size and connectivity of the {$k$}-core of a random graph}
}

@incollection{McDiarmid89,
  title = {On the method of bounded differences},
  booktitle = {Surveys in combinatorics, 1989 ({{Norwich}}, 1989)},
  author = {McDiarmid, C.},
  year = 1989,
  series = {London math. {{Soc}}. {{Lecture}} note ser.},
  volume = {141},
  pages = {148--188},
  publisher = {Cambridge Univ. Press, Cambridge},
  isbn = {0-521-37823-0},
  mrnumber = {1036755}
}

@article{Montgomery19,
  title = {Hamiltonicity in random graphs is born resilient},
  author = {Montgomery, R.},
  year = 2019,
  journal = {J. Combin. Theory Ser. B},
  volume = {139},
  pages = {316--341},
  issn = {0095-8956},
  doi = {10.1016/j.jctb.2019.04.002},
  url = {https://doi.org/10.1016/j.jctb.2019.04.002},
  mrnumber = {4010194}
}

@article{Pittel96,
  author = {Pittel, B. and Spencer, J. and Wormald, N.},
  year = 1996,
  journal = {J. Combin. Theory Ser. B},
  volume = {67},
  number = {1},
  pages = {111--151},
  issn = {0095-8956},
  doi = {10.1006/jctb.1996.0036},
  url = {https://doi.org/10.1006/jctb.1996.0036},
  mrnumber = {1385386},
  title = {Sudden emergence of a giant {$k$}-core in a random graph}
}

@article{Schwarz06,
  author = {Schwarz, J. M. and Liu, A. J. and Chayes, L. Q.},
  year = 2006,
  journal = {EPL},
  volume = {73},
  number = {4},
  pages = {560--566},
  doi = {10.1209/epl/i2005-10421-7},
  url = {https://doi.org/10.1209/epl/i2005-10421-7},
  title = {The onset of jamming as the sudden emergence of an infinite {$k$}-core cluster}
}

@article{Xue24,
  title = {Nucleation phenomena and extreme vulnerability of spatial k-core systems},
  author = {Xue, L. and Gao, S. and Gallos, L. K. and Levy, O. and Gross, B. and Di, Z. and Havlin, S.},
  year = 2024,
  journal = {Nat. Commun.},
  volume = {15},
  number = {1},
  pages = {5850},
  issn = {2041-1723},
  doi = {10.1038/s41467-024-50273-5},
  url = {https://www.nature.com/articles/s41467-024-50273-5},
  urldate = {2026-01-06},
  langid = {english}
}

\end{document}